\documentclass[12pt,a4paper]{article}

\usepackage{color}
\usepackage{enumerate} 

\usepackage{amsfonts,amsmath,amssymb, amsthm}

\usepackage
{subcaption}

\usepackage{graphicx}

\theoremstyle{plain}

\newtheorem{theorem}{Theorem}[section]
\newtheorem{lemma}[theorem]{Lemma}

\newtheorem{corollary}[theorem]{Corollary}

\theoremstyle{definition}

\newtheorem{definition}[theorem]{Definition}
\newtheorem{example}[theorem]{Example}

 \newtheoremstyle{TheoremNum}  
        {\topsep}{\topsep}              
        {\itshape}                      
        {}                              
        {\bfseries}                     
        {.}                             
        { }                             
        {\thmname{#1}\thmnote{ \bfseries #3}}
    \theoremstyle{TheoremNum}
    \newtheorem{theorem-repeat}{Theorem}

\DeclareMathOperator{\Tr}{Tr}
\DeclareMathOperator{\tr}{tr}

\DeclareMathOperator{\id}{id}

\DeclareMathOperator{\mo}{mod}
\DeclareMathOperator{\Real}{Re}

\DeclareMathOperator{\Arf}{Arf}

\newcommand{\SU}{{\mathrm {SU}}}

\newcommand\Cb            {\mathbb{C}}
\newcommand\Mb            {\mathbb{M}}

\newcommand\Rb            {\mathbb{R}}
\newcommand\Zb            {\mathbb{Z}}
\newcommand\Hb            {\mathbb{H}}

\begin{document}

\title{Two-dimensional state sum models and spin structures}
\date{} 
\author{John W. Barrett, Sara O. G. Tavares \\School of Mathematical Sciences\\ University of Nottingham}
\maketitle
\begin{abstract}The state sum models in two dimensions introduced by Fukuma, Hosono and Kawai are generalised by allowing algebraic data from a non-symmetric Frobenius algebra. Without any further data, this leads to a state sum model on the sphere. When the data is augmented with a crossing map, the partition function is defined for any oriented surface with a spin structure. An algebraic condition that is necessary for the state sum model to be sensitive to spin structure is determined.
Some examples of state sum models that distinguish topologically-inequivalent spin structures are calculated.
 \end{abstract}

\section{Introduction} \label{sec:intro}

The aim of this paper is to generalise the two-dimensional state sum models defined by Fukuma, Hosono and Kawai ~\cite{Fukuma} and by Lauda and Pfeiffer \cite{Lauda} to define new state sum models that depend on a choice of spin structure for the surface.

In general, a state sum model is defined using a combinatorial structure for a manifold such as a lattice or triangulation. In a physics context, this manifold could be either space or space-time. The model has variables that are defined at each `lattice site', which could be vertices or edges etc., depending on the model. A value for a variable at a particular site is called a state and a state of the whole model is a particular choice of state for each site. The model is determined by defining a number for each state of the model called the weight. Then the partition function $Z$ of the theory is calculated by summing the weights over all possible states 
\begin{equation}\label{statesum} Z=\sum_\text{states $\phi$}\text{weight}(\phi).\end{equation}
If the manifold has a boundary then usually this formula is refined by not summing over the states on the boundary; these are regarded as boundary data.

State sum models are used in a variety of areas of mathematics and physics, including statistical mechanics \cite{Baxter}, random matrices \cite{Francesco}, knot theory \cite{KauffmanState}, lattice field theory \cite{Oeckl:2000hs} and quantum gravity \cite{Barrett:2000xs}. In a quantum model, the weight is a complex number called the amplitude of the configuration. Statistical models are very similar, but the weights are real numbers that are interpreted as probabilities.

In quantum mechanics, the Feynman-Hibbs description of the path integral for the unitary evolution operator over a finite time interval is defined as a sum over intermediate states at a finite number of intermediate times. This can be viewed as a state sum model on a one-dimensional space-time, the weight being the product of matrix elements for the evolution operator. The state sum \eqref{statesum} can be viewed as the generalisation of this construction to higher-dimensional space-times and is therefore interpreted as a discrete version of the functional integral, such as is used in lattice gauge theory.
 
The models considered in this paper are on a two-dimensional manifold (a surface),  and the discrete structure on it is a triangulation. The two-dimensional case is almost the simplest possible, as the complexity of the models increases with the dimension of the manifold. Eventually one will be interested in higher-dimensional models and it is hoped that understanding the two-dimensional case will help.

The second main feature of the models is that the partition function for a closed manifold is independent of the triangulation of the manifold. Such triangulation-invariant models are important in quantum gravity, where the invariance is related to the symmetry of general relativity under diffeomorphisms. Triangulation-invariant models also model topological phases in solid state physics \cite{Levin:2004mi}. 

A general framework for two-dimensional triangulation-independent state sum models was first defined by Fukuma, Hosono and Kawai (FHK)~\cite{Fukuma}, following the earlier examples (with an infinite set of states) determined by the combinatorial construction of two-dimensional Yang-Mills quantum field theory \cite{Witten:1991we}. In the FHK models, the weight of a state is defined using a product of local factors that depend only on the state variables in a local neighbourhood (\S\ref{sec:lattice_tft}, equation \eqref{eq:Z_surf}). Here, such a state sum model is called a naive state sum model.  The fact that the algebraic data for an FHK model is a symmetric Frobenius algebra was recognised by \cite{Fuchs:2001am}. 

The generalisation of a naive state sum model is called a diagrammatic state sum model and was first defined by Lauda and Pfeiffer \cite{Lauda}; their models are related to one class of models described here (the curl-free ones in \S\ref{sec:curlfreemodels}). Here, the weight is determined by constructing a diagram in $\Rb^2$ from the triangulation and then evaluating a naive state sum model for the diagram. The diagram is constructed by taking the graph dual to the triangulation (having a vertex in the centre of every triangle and a line crossing every edge) and then projecting it into the plane, $\Rb^2$. Crucially, the horizontal direction in the plane is singled out and lines that have maxima and minima have a different evaluation in the state sum model than lines that are nowhere horizontal. There is also an amplitude factor for points where lines cross in the projection (`a crossing'). 

Such a diagrammatic calculus was first introduced by Penrose for calculating invariants of representations of $\SU(2)$, known as spin networks \cite{Major:1999md}.
It is also familiar in knot theory where topological invariants of knots are calculated in such a fashion \cite{Kauffman-handbook}. All these examples have the common feature that the overall result of the state sum model is a topological invariant of some kind, even if the individual pieces appear to break that invariance.

The partition function of a closed surface in an FHK or Lauda-Pfeiffer model depends only on the topology of the surface and so the model is called a topological state sum model. A certain refinement of the models is required to define models that depend also on the spin structure of a manifold. 

Spin structures are important in physics because the notion of a fermion field depends on a choice of a spin structure; it is used in the definition of the bundle of spinors. Therefore to be able to model fermionic systems adequately it is necessary to have a good understanding of quantum models that depend on spin structures. If the space-time is $\Rb^d$, then there is only one spin structure, so that spin structures do not need to be mentioned. However for many other manifold topologies, the spin structure plays an important role. For example, on a torus a spin structure is equivalent to posing periodic or anti-periodic boundary conditions for a fermion field for the different ways in which one can traverse the torus along a closed loop.

In general, a spin structure is a non-local structure on a manifold, similar to the notion of a cohomology class. While it can be represented by some local data, this data is not unique and often not homogeneous. For example, one can describe one of the spin structures on a circle by breaking the circle at one point and prescribing anti-periodic boundary conditions for a fermion field at that point. However it is clear that the choice of point is arbitrary and a certain `gauge transformation' of the data relates any two such choices. 

In this paper, a spin structure of a surface is represented by an immersion of the surface in $\Rb^3$, two immersions being equivalent if they are related by regular isotopy. This representation is very convenient because it immediately gives a diagrammatic state sum model by embedding the dual graph in $\Rb^3$ and then projecting to $\Rb^2$. Moreover the regular isotopy in $\Rb^3$ leads to exactly the relations on diagrams in $\Rb^2$ that are familiar as generalised Reidemeister moves from knot theory. One can regard the local data representing the spin structure as the values of the three coordinates of $\Rb^3$.

In previous work, a one-dimensional topological state sum model with fermionic characteristics is defined in \cite{Barrett:2012ii}. In that paper a state sum model is defined on a circle that has a gauge field with holonomy $Q$ around the circle. However by setting $Q$ to be plus or minus the identity matrix, this can be regarded as a model on a circle with either of the two possible spin structures for the circle; one finds that the partition functions for the two spin structures differ. In two dimensions, Kasteleyn orientations and dimer configurations have been used to represent spin structures for statistical dimer models \cite{CimasoniReshetikhin}. This representation of a spin structure is rather different to that adopted here, but it shares the feature that many configurations of local data correspond to the same spin structure.

 An outline of the contents of the paper is as follows. The construction of naive state sum models is reviewed in \S \ref{sec:lattice_tft} giving the equivalence of an FHK state sum model with a symmetric special Frobenius algebra. Over $\Cb$ or $\Rb$, these are direct sums of matrix algebras and the partition function for an oriented surface is calculated (theorem \ref{theo:state_sum_algebra}). The formula for a Frobenius algebra over $\Cb$ is essentially that given by FHK whereas the formula for an algebra over $\Rb$ is new. 

The properties of a diagrammatic calculus for planar diagrams (i.e., diagrams without crossings) are described in detail in \S \ref{sec:diagram}. The corresponding state sum model defines a partition function for a disk. These state sum models are called planar if in addition they are invariant under Pachner moves, in which case the algebraic data is
generalised to the case of Frobenius algebras that are not necessarily symmetric. The definition of these models is new but not very surprising; it amounts to observing that the definition of the diagrammatic calculus of \cite{Lauda} can be simplified in the case of a disk so that data for a crossing are not required. The state sum model data is a set of coefficients $(C,B,R)$ from which one can construct a vector space $A$ and a multiplication map $m$ on $A$.
The main result of this section is the equivalence of the state sum model data with a certain type of Frobenius algebra, generalising the corresponding result (theorem \ref{theo:state_sum_algebra}) for the FHK case.

\begin{theorem-repeat}[\ref{theo:diagram}] 
Non-degenerate diagrammatic state sum model data $(C,B,R)$ determine a planar state sum if and only if the multiplication map $m$, the bilinear form $B$ and the distinguished element $\beta = m(B)$ determine on $A$ the structure of a special Frobenius algebra with identity element $1=R\beta$.
\end{theorem-repeat}

This theorem leads to a classification of the possible state sum models with real or complex coefficients by classifying the possible Frobenius algebras (theorem \ref{theo:diagram_semi_simple}).
The models also determine a partition function for a sphere in a construction called a spherical state sum model (\S\ref{sec:spherical}).  A stronger version of spherical symmetry is an additional condition identified in lemma \ref{lem:spherical}, which would be needed to formulate state sum models on submanifolds (with boundary) of the sphere. However, manifolds with boundary are not considered further in this paper, so this condition is not incorporated into the definitions.

Diagrammatic state sum models are generalised to all surfaces without boundary in \S \ref{sec:crossing} by specifying a crossing map in addition to the Frobenius algebra. A set of axioms for the crossing is given; these, together with the axioms for a planar state sum model, ensure the state sum model is an invariant of the surface with a spin structure. Such state sum models are therefore called spin state sum models. Some aspects of this construction are the subject of independent work by Novak and Runkel \cite{Novak,Novak_Runkel}. A useful source of examples is the construction of a spin state sum model from a graded Frobenius algebra with a bicharacter (lemma \ref{lem:graded}).

If a spin state sum model does not actually depend on the spin structure then it is called topological and the partition function is a topological invariant of the surface. Particular examples of topological state sum models can be made by adding an additional axiom to the spin state sum models called the curl-free condition. The FHK state sum models can be seen as a special case of these by working with the canonical crossing map. Thus there is the following hierarchy of models on closed surfaces:
$$\text{ FHK }\subset\text{ topological }\subset\text{ spin }\to\text{ spherical}.$$ 
In this hierarchy, the FHK models are examples of topological models, which are in turn examples of spin models. Each spin model determines a spherical model by ignoring the crossing map.
 
The curl-free condition simplifies the models considerably, giving a class of models in \S\ref{sec:curlfreemodels} that is relatively easy to describe. Some explicit examples are presented, with one of them having a non-symmetric Frobenius form (example \ref{ex:non_symmetric_Frobenius}).

Finally, the spin models are investigated in detail in \S\ref{sec:spinssm}. It is shown that a spin state sum model has two distinguished elements in the center of the Frobenius algebra, $\eta$ and $\chi$, from which the partition function of any closed surface can be calculated. The difference between $\eta$ and $\chi$ is that they correspond to two inequivalent spin structures of a torus with a disk removed. The main result is the formula for the partition function, which uses the Frobenius form $\varepsilon$ and the notion of the parity of a spin structure on a surface. The parity (odd or even) distinguishes spin structures that are not related by a homeomorphism of the surface.

\begin{theorem-repeat}[\ref{theo:main}]
Let  $(C,B,R,\lambda)$ be a spin state sum model. Then the partition function $Z$ of a triangulated surface $\Sigma_g$ of genus $g$ immersed in $\Rb^3$ depends only on $g$ and the parity of the spin structure $s$. Moreover,
\begin{align}
Z(\Sigma_g,s)=\left\lbrace
\begin{array}{lcl}
R\varepsilon(\eta^g)&\quad& \text{($s$ even parity)}\\
R\varepsilon(\chi\eta^{g-1})&& \text{($s$ odd parity).}
\end{array}
\right. 
\end{align}
\end{theorem-repeat}
This is computed explicitly for some particular spin state sum models showing that the topologically-inequivalent spin structures on a surface can be distinguished by the models (examples \ref{ex:complex_crossing}--\ref{ex:cyclic}). A necessary condition for this is that $\eta\ne\chi$.

A categorical perspective on the state sum model axioms is given in \S\ref{sec:cat} with a brief discussion of category-theoretic generalisations of the models. Category theory is not mentioned in \S\ref{sec:lattice_tft}--\S\ref{sec:crossing} since we prefer to emphasise the concrete construction of models using linear maps, as is familiar from quantum mechanics. However, category theory underlies the constructions in this paper. For example, one can understand the axioms involving the crossing in definition \ref{def:spin-model} as those of a symmetric ribbon category. One expects that category-theoretic constructions will play an important role in determining higher-dimensional analogues of the theory presented in this paper \cite{Douglas:2013aea}.

\section{Naive state sum models} \label{sec:lattice_tft}

This section reviews the construction of state sum models according to the work of Fukuma, Hosono and Kawai~\cite{Fukuma}, with the calculation of examples. These state sum models are called naive state sum models to distinguish them from the generalisation to diagrammatic ones in \S\ref{sec:diagram}.

\begin{figure}
\centering
\begin{subfigure}[t!]{0.47\textwidth}
                \centering
		\includegraphics{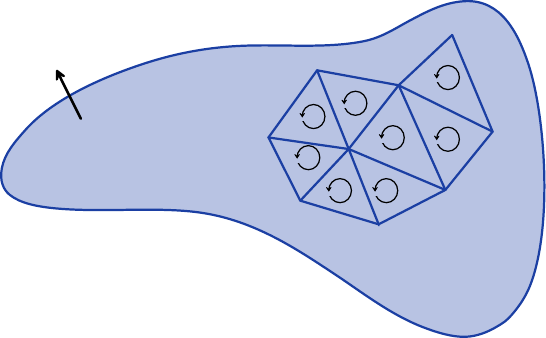}
		\caption{\emph{Orientation}. Each triangle inherits an orientation induced by the orientation of the surface.}
		\label{fig:glued_orient}
\end{subfigure}
\hspace{5mm}
\begin{subfigure}[t!]{0.47\textwidth}
                	\centering
		\vspace{8.5mm}
		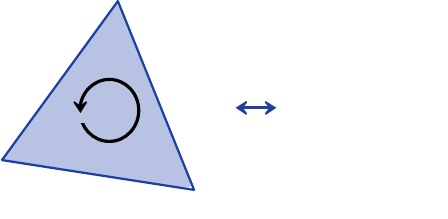
		\vspace{7.15mm}
		\caption{\emph{Triangle amplitudes}. Each edge on a triangle is associated with one of a finite set of states $S$. The quantum amplitude for this oriented triangle is $C_{abc}$.}
		\label{fig:tri_label}
\end{subfigure}
\label{fig:triang_and_cyclicity}
\caption{}
\end{figure} 

The idea of a state sum model is to calculate a quantum amplitude for a given triangulated manifold, possibly with a boundary. These amplitudes are numbers in a field $k$, for which the main examples of interest here are $k=\Rb$ or $\Cb$.   A surface $\Sigma$ is a two-dimensional compact manifold, orientable but not necessarily closed.  The surfaces are triangulated, and since they are compact, the number of vertices, edges and triangles is finite. The orientation of $\Sigma$ induces an orientation on each triangle. This means a triangle has a specified cyclic order of its vertices and these orientations are glued together coherently to preserve the overall orientation of the surface (see figure~\ref{fig:glued_orient}).

A naive state sum model on an oriented triangulated surface $\Sigma$ has a set of quantum amplitudes for each vertex, edge and triangle. These are multiplied together and summed to give an overall amplitude to $\Sigma$.   

Each edge on a triangle is associated with one of a finite set of states $S$ and the quantum amplitude for the oriented triangle with edge states $a,b,c \in S$ is $C_{abc} \in k$, as shown in figure~\ref{fig:tri_label}. These amplitudes are required to satisfy invariance under rotations, 
\begin{align} \label{eq:C_cycle}   
C_{abc}=C_{bca}=C_{cab},
\end{align}
which is to say they must respect the cyclic symmetry of an oriented triangle. If the orientation is reversed then the amplitude is $C_{bac}$ and therefore not necessarily equal to $C_{abc}$. 

The triangles are glued together using a matrix $B^{ab}$ associated to each edge of the triangulation not on the boundary (an interior edge). Since the formalism for naive state sum models does not distinguish the two triangles meeting at the edge, one must require symmetry,
\begin{align} \label{eq:metric}
B^{ab}=B^{ba}.
\end{align}
Note that this condition is relaxed in \S \ref{sec:diagram}, together with a modification of the cyclic symmetry \eqref{eq:C_cycle}. 

Finally, each interior vertex has amplitude $R \in k$. This is a slight generalisation of the  formalism presented in \cite{Fukuma}, where $R=1$ was assumed.

\begin{figure}
\centering
\begin{subfigure}[t!]{0.47\textwidth}
                \centering
		\vspace{7.3mm}
		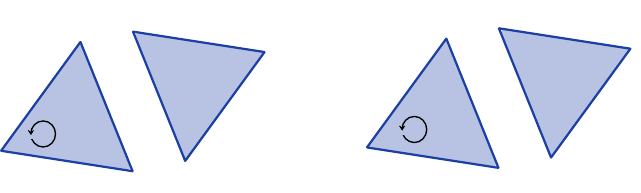 
		\vspace{0.5mm}
		\caption{\emph{Symmetry}. The symmetry relation $B^{ab}=B^{ba}$ implies that the amplitude for the two triangles glued together is invariant under a rotation by
$\pi$.}
		\label{fig:tri_metric}
\end{subfigure}
\hspace{5mm}
\begin{subfigure}[t!]{0.47\textwidth}
                \centering
		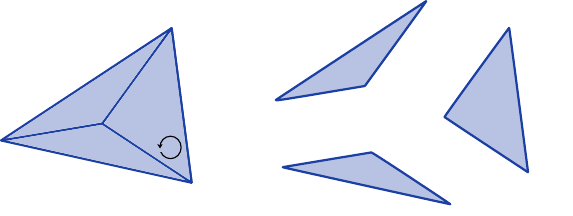
		\vspace{3.5mm}
		\caption{\emph{A triangulation of the disk}. The partition function $Z_{abc}$ is constructed from the constants $C$ associated to each triangle and matrices $B$ associated to each interior edge.}
		\label{fig:disk}
\end{subfigure}
\label{fig:metric_and_disk}
\caption{}
\end{figure} 

All the data needed to calculate the amplitude of a surface is now defined. Each edge in each triangle has a variable $a\in S$. For a given value of each of these variables, the amplitude of a triangle $t$ is $A(t)=C_{abc}$ (with $a,b,c$ the three variables on the three edges), and likewise the amplitude of an edge $e$ is $A(e)=B^{ab}$. The amplitude of the surface is called the partition function and is given by the formula that involves summing over the states on all interior edges,
\begin{align} \label{eq:Z_surf}
Z(\text{\footnotesize{boundary states}})=R^{V}\sum_{\text{interior states}}\left(\prod_{\text{triangles }t}A(t)\prod_{\text{interior edges }e}A(e)\right),
\end{align}
with $V$ the number of interior vertices. For example, the amplitude of the triangulated disk of figure~\ref{fig:disk} is
\begin{align} \label{eq:Z_disk}
Z_{abc}=R\,C_{e'dc}\,C_{af'e}\,C_{fbd'}\,B^{dd'}B^{ee'}B^{ff'},
\end{align}
using the Einstein summation convention for each paired index. The resulting partition function depends on the boundary data $a,b,c$, which are not summed. The formula \eqref{eq:Z_surf} for a naive state sum model is an instance of \eqref{statesum} with the weight a product of factors for each vertex, edge and triangle.

As a consequence of \eqref{eq:C_cycle} and \eqref{eq:metric}, the amplitude is invariant under an orientation-preserving simplicial map (a map that sends vertices to vertices, edges to edges and triangles to triangles), as in the example shown in figure~\ref{fig:tri_metric}. 

The formalism can be interpreted in terms of linear algebra. The states $a \in S$ correspond to basis elements $e_a$ of a vector space $A$. The amplitude $C_{abc}$ is the value of a trilinear form $C\colon A\times A\times A\to k$ on basis elements, 
$C(e_a,e_b,e_c)= C_{abc}$. The form $C$ can also be viewed as a linear map on the tensor product, 
 $C \colon A \otimes A \otimes A \to k$. Similarly, 
\begin{equation} B=e_a\otimes e_b \, B^{ab} \in A \otimes A.\end{equation}
 This element can be viewed as a bilinear form on $A^*$, i.e., $B\colon A^*\times A^*\to k$, with matrix elements $B^{ab}= B(e^a,e^b)$,  using the dual basis elements $e^a$. This linear algebra perspective means it is possible to regard state sum models as isomorphic if they are related by a change of basis; this is used below. 

The bilinear form $B$ can be used to raise indices; thus, 
using the definition $C_{ab}{}^c = C_{abd}\,B^{dc}$ there is a multiplication map $m \colon A \otimes A \to A$ with components 
\begin{equation}\label{eq:multiplicationmap} m(e_a \otimes e_b)= C_{ab}{}^c\, e_c.\end{equation} 
The notations $m(e_a \otimes e_b)=e_a \cdot e_b$ will be used interchangeably. The state sum model data can also be used to determine a distinguished element of $A$,
\begin{align}
\beta=m(B)=e_a\cdot e_b\,B^{ab}.
\end{align}
Throughout it is assumed the data for the state sum model are non-degenerate: $R\neq 0$, $B(\cdot,a)=0 \Rightarrow a=0$ and $C(\cdot,\cdot,a)=0 \Rightarrow a=0$. This means $B$ has an inverse 
\begin{equation}B^{-1}=B_{ab}e^a\otimes e^b \in A^{\ast} \otimes A^{\ast}.\end{equation}
 This is defined by 
\begin{equation}B_{ac}B^{cb}=\delta_a^{b}.\label{eq:snake}\end{equation}
 This determines a bilinear form on $A$ with components $B^{-1}(e_a,e_b)=B_{ab}$ and can be used to lower indices.

 Note that this discussion of the formalism in terms of linear algebra does not depend on the symmetry of $B$, and these definitions will also be used in later sections where the symmetry of $B$ is dropped.

A topological state sum is one for which the partition function of a surface is independent of the triangulation. This is made precise by the following definition.

\begin{definition}
A state sum model is said to be topological if $Z(M)=Z(M')$ whenever $M$ and $M'$ are two closed oriented triangulated surfaces on which the state sum model is defined and there is a piecewise-linear homeomorphism $f\colon M\to M'$ that preserves the orientation.
\end{definition}

Any two triangulations of a surface are connected by a sequence of the two Pachner moves, shown in figures~\ref{fig:Pach1} and \ref{fig:Pach2}, or their inverses\footnote{Triangulations are allowed to be degenerate, that is, two simplexes can intersect in more than one face. It is allowable to use a degenerate triangulation that can be subdivided by Pachner moves into a non-degenerate one.}. For a closed manifold this result is proved in  \cite{Pachner,Lickorish-moves}.  (In fact this result can be extended to a manifold with boundary \cite{CowardLackenby}, but this result is not used here.) Thus it is sufficient to check for each Pachner move that the partition functions for the disk on the two sides of the move are equal.

\begin{figure}
\centering
\begin{subfigure}[t!]{0.47\textwidth}
                \centering
		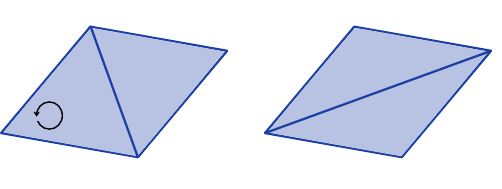 
		\vspace{3mm}
		\caption{\emph{Pachner move 2-2}. The change of triangulation of a manifold can be interpreted as exchanging two faces of a tetrahedron with the remaining two.}
		\label{fig:Pach1}
\end{subfigure}
\hspace{5mm}
\begin{subfigure}[t!]{0.47\textwidth}
                \centering
		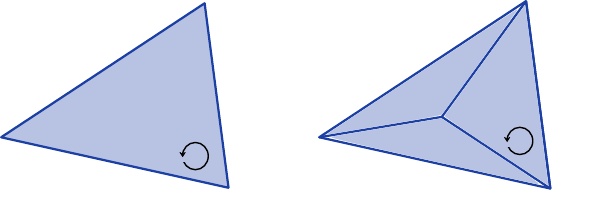
		\vspace{1.5mm}
		\caption{\emph{Pachner move 1-3}. This change of triangulation can be interpreted as replacing one face of a tetrahedron with the remaining three.}
		\label{fig:Pach2}
\end{subfigure}
\label{fig:Pachner_moves}
\caption{}
\end{figure} 

In the case of topological state sum models there is a connection between the vector space $A$ and a Frobenius algebra. Recall that a Frobenius algebra is a finite-dimensional associative algebra $A$ with unit $1 \in A$ and a linear map $\varepsilon \colon A \to k$ that determines a non-degenerate bilinear form $\varepsilon \circ m$ on $A$. The linear map $\varepsilon$ is called the Frobenius form. A Frobenius algebra is called symmetric if $\varepsilon \circ m$ is a symmetric bilinear form. Let $B\in A\otimes A$ be the inverse of $B^{-1}=\varepsilon\circ m$ according to \eqref{eq:snake}. Then the Frobenius algebra is called special if $m(B)$ is a non-zero multiple of the identity element.

 A naive state sum model that obeys the Pachner moves is the type of model discussed by Fukuma, Hosono and Kawai, and so these are called FHK state sum models. The following result is a more precisely-stated version of their result in \cite{Fukuma}.

\begin{theorem}\label{theo:state_sum_algebra}
Non-degenerate naive state sum model data $(C,B,R)$ determine an FHK state sum model if and only if the multiplication map $m$, the bilinear form $B$ and the distinguished element $\beta$ determine on $A$ the structure of a symmetric special Frobenius algebra with identity element $1=R\beta$. 
\end{theorem}
\begin{proof}
The proof begins by showing that the data determine a symmetric Frobenius algebra. The first Pachner move, shown in figure~\ref{fig:Pach1}, can be written
\begin{align} \label{eq:Pach1}
C_{ab}{}^e\, C_{ecd}= C_{bc}{}^e\, C_{aed}
\end{align}
and is equivalent to associativity of the multiplication. To see this note that using the notation \eqref{eq:multiplicationmap} of a multiplication, $(e_a \cdot e_b) \cdot e_c=C_{ab}{}^e\,C_{ec}{}^{f}\,e_f$ and $e_a \cdot (e_b \cdot e_c)=C_{bc}{}^e\,C_{ae}{}^f \,e_f$; hence, the identity \eqref{eq:Pach1} is $B^{-1}(e_a \cdot (e_b \cdot e_c),e_d)=B^{-1}((e_a\cdot e_b)\cdot e_c,e_d)$. Since the bilinear form $B^{-1}$ is non-degenerate this is equivalent to having an associative multiplication $m$. A linear functional can be defined by setting $\varepsilon(x)= B^{-1}(x,1)$. The cyclic symmetry \eqref{eq:C_cycle} implies that $B^{-1}(x\cdot y,z)=B^{-1}(x,y\cdot z)$ and so $\varepsilon(x\cdot y)=B^{-1}(x\cdot y,1)=B^{-1}(x,y)$, which is non-degenerate and symmetric. 

The move in figure~\ref{fig:Pach2} requires the partition function of the disk \eqref{eq:Z_disk} to  equal $C_{abc}$. This is equivalent to
\begin{align}C_{ab}{}^c&=R\,C_{ed}{}^c\,C_{af'}{}^e\,C_{fb}{}^d\,B^{ff'}\notag\\
&=R\,C_{f'd}{}^h\,C_{ah}{}^c\,C_{fb}{}^d\,B^{ff'}
\end{align}
using associativity. For non-degenerate $C$, and rewriting $C_{ab}{}^c=C_{ah}{}^c\delta_b^h$, this is equivalent to
\begin{align}\label{eq:Pachner_13}\delta_b^h&=R\,C_{f'd}{}^h\,C_{fb}{}^d\,B^{ff'}\notag\\
&=R\,C_{f'f}{}^d\,C_{db}{}^h\,B^{ff'}.
\end{align}
Recognising that $\beta=B^{ff'}C_{f'f}{}^d\,e_d$, expression \eqref{eq:Pachner_13} implies that $R\beta$ must be the unit element for multiplication, and hence $A$ is an algebra; it is therefore a symmetric special Frobenius algebra. It is worth noting that the non-degeneracy of $C$ is necessary here, as without it the algebra need not even be unital.

Conversely, given a symmetric Frobenius algebra with linear functional $\varepsilon$, this defines a non-degenerate and symmetric bilinear form $B^{-1}= \varepsilon \circ m$ with property \eqref{eq:C_cycle}. The fact that the algebra is unital implies that $C$ is non-degenerate. Finally, associativity and the property $R\beta=1$ guarantee the Pachner moves are satisfied, meaning the state sum model created is an FHK model. 
\end{proof}

It is worth noting that having $\beta$ proportional to the identity is a non-trivial restriction on Frobenius algebras. For the cases  $k=\Rb$ or $\Cb$ of interest in this paper the Frobenius algebras, and hence the state sum models, are easily classified.  The results for the symmetric Frobenius algebras in this section are stated here, with the proof of the classification given in a more general context in theorem~\ref{theo:diagram_semi_simple} of \S \ref{sec:diagram}.

Let $\Mb_{n}(\Cb)$ denote the algebra of $n \times n$ matrices over $\Cb$.
An FHK state sum model over the field $\Cb$ is isomorphic, by a change of basis, to one in which the algebra is a direct sum of matrix algebras, \begin{align} \label{eq:complex_matrix_algebra}
A = \bigoplus\limits_{i=1}^N \Mb_{n_i}(\Cb).
\end{align}
The Frobenius form on an element $a = \oplus_i a_i$ is defined using the matrix trace on each factor:
\begin{align} \label{eq:cxfrob}
\varepsilon(a) = R \sum\limits_{i=1}^N n_i \Tr (a_i).
\end{align}

For the real case, the classification uses the division rings  $\Rb$, $\Cb$ and $\Hb$ regarded as algebras over $\Rb$; these are denoted $\Rb$, $\Cb_{\Rb}$ and $\Hb_{\Rb}$, and the dimension of the division ring $D$ as an $\Rb$-algebra is denoted
$|D|$; thus $|\Rb|=1$, $|\Cb_\Rb|=2$, $|\Hb_\Rb|=4$. The imaginary unit in $\Cb$ is denoted $\hat{\imath}$ and the corresponding units for the quaternions $\hat{\imath}$, $\hat{\jmath}$ and $\hat{k}$.
 The real part of a quaternion is defined as $\Real(t+x\hat{\imath}+y\hat{\jmath}+z\hat{k})= t$ and the conjugate by $(t+x\hat{\imath}+y\hat{\jmath}+z\hat{k})^*=t-x\hat{\imath}-y\hat{\jmath}-z\hat{k}$. The $n\times n$ matrices with entries in $D$ are denoted $\Mb_{n}(D)$ and are algebras over $\Rb$.

An FHK state sum model over the field $\Rb$  is isomorphic by a change of basis to one in which 
\begin{align} \label{eq:real_matrix_algebra}
 A = \bigoplus\limits_{i=1}^N \Mb_{n_i}(D_i), \quad \text{ with }D_i=\Rb,\Cb_{\Rb},\text{or }\Hb_{\Rb}.
\end{align}
 The Frobenius form is defined by
\begin{align} \label{eq:realfrob}
\varepsilon(a) =R \sum\limits_{i=1}^N  |D_i|\,n_i \Real  \Tr(a_i).
\end{align}
The fact that these formulas do determine Frobenius algebras is proved here.
\begin{lemma} \label{lem:Frobenius_forms}
The equations \eqref{eq:cxfrob} and \eqref{eq:realfrob} determine symmetric Frobenius forms such that $R\beta=1$.
\end{lemma}
\begin{proof} That \eqref{eq:cxfrob} determines a symmetric Frobenius form follows from the fact that $\Tr(xy)$ is a non-degenerate symmetric bilinear form on $\Mb_{n}(\Cb)$. 

For \eqref{eq:realfrob} there are three separate cases to handle: $\Mb_{n}(D)$ for $D=\Rb$, $\Cb_{\Rb}$ and $\Hb_{\Rb}$. The bilinear form $\Real \Tr (xy) $ reduces to $\Tr(xy)$ in the first case and this is non-degenerate on $\Mb_{n}(\Rb)$. In the $D=\Cb_\Rb$ case, $\Real\Tr(xy)=0$ and $\Real\Tr(x(\hat{\imath}y))=0$ implies that $\Tr(xy)=0$. So $\Real\Tr(xy)=0$ for all $y\in\Mb_{n}(\Cb_\Rb)$ implies that $x=0$. Thus $\Real\Tr(xy)$ is a non-degenerate form. Finally, a similar proof works for $D=\Hb_\Rb$. In all these cases the bilinear form determined by $\Real\Tr$ is symmetric.

Let $k=\Cb$. A basis for \eqref{eq:complex_matrix_algebra} is given by elementary matrices $\lbrace e_{lm}^{i}\rbrace_{l,m=1,n_i}^{i=1,N}$ satisfying $\left(e_{lm}^{i}\right)_{rs}=\delta_{lr}\delta_{ms}$. Then
\begin{align}
B=\frac{1}{R}\sum_{i,lm}\frac{1}{n_i}\,e_{lm}^{i}\otimes e_{ml}^i,
\end{align}
as can be verified by applying the identity $B^{-1}B=1$ to the above expression and using equation \eqref{eq:cxfrob}. Let $1=\oplus_i 1_i$; noticing $\sum_{lm}e^{i}_{lm}e^i_{ml}=n_i 1_i$, it is straightforward to conclude that $\beta=m(B)=R^{-1}1$. 

Suppose now that $k=\Rb$ and let $A$ be as in \eqref{eq:real_matrix_algebra}. Choose as a basis for the $i$-th component of $A$ either $\lbrace e_{lm}^i\rbrace$, $\lbrace e_{lm}^i, \hat{\imath}\,e_{lm}^i\rbrace$ or $\lbrace e_{lm}^i, \hat{\imath}\,e_{lm}^i,\hat{\jmath}\,e_{lm}^i,\hat{k}\,e_{lm}^i\rbrace$ according to $D_i=\Rb$, $\Cb_{\Rb}$ or $\Hb_{\Rb}$, respectively. The element $B$ associated with \eqref{eq:realfrob} will then take the form
\begin{align}
B=\frac{1}{R}\sum_{i,lm}\sum_{w_i}\frac{1}{|D_i|n_i}\, w_i\,e_{lm}^{i}\otimes_{\Rb} w_i^{\ast}\,e_{ml}^i, \hspace{4mm}
w_i=
\begin{cases}
1 &(D_i=\Rb)\\
1,\hat{\imath} &(D_i=\Cb_{\Rb})\\
1,\hat{\imath},\hat{\jmath},\hat{k} &(D_i=\Hb_{\Rb})
\end{cases}.
\end{align}
Since the product $w_iw_i^{\ast}=1$ for all $i$ then $\sum_{lm,w_i} w_i\,e_{lm}^{i} w_i^{\ast}\,e_{ml}^i=n_i|D_i|1_i$. The identity $m(B)=R^{-1}1$ is therefore satisfied.
\end{proof}

The partition function for a surface can now be calculated for these examples. Let  $\Sigma_g$ denote an oriented surface of genus $g$.
Gluing two triangles together gives the partition function \eqref{eq:Pach1} of the disk with four boundary edges labelled with states $a,b,c,d$ which is equal to $\varepsilon(e_a\cdot e_b\cdot e_c\cdot e_d)$. Gluing these boundary edges to make the sphere $\Sigma_0=S^2$ by identifying the states $a,d$ and $b,c$ results in the partition function
\begin{equation}\label{eq:sphere} Z(\Sigma_0)=R^3\,\varepsilon(e_a\cdot e_b\cdot e_c\cdot e_d)B^{ad}B^{bc}=R\,\varepsilon(1).\end{equation}
Gluing opposite edges results in the torus 
\begin{equation}\label{eq:torus} Z(\Sigma_1)=R\,\varepsilon(e_a\cdot e_b\cdot e_c\cdot e_d)B^{ac}B^{bd}=R\,\varepsilon(z)\end{equation}
with $z=e_a\cdot e_b\cdot e_c\cdot e_d\,B^{ac}B^{bd}$.
The surface $\Sigma_g$ for $g>0$ can be constructed from a disk with $4g$ boundary edges as presented in figure~\ref{fig:4gpoly}. This results in the partition function
\begin{align} \label{eq:invariant}
Z(\Sigma_g)=R\,\varepsilon\left( z^g\right)
\end{align}
valid for all $g$. Although a specific orientation was picked when constructing expression \eqref{eq:invariant}, the result is actually independent of orientation. Such a symmetry of the partition function is to be expected as it is easy to show orientation-reversing homeomorphisms exist for closed surfaces. Alternatively, this invariance can be proved directly through the partition function. For example, for the torus the two possible partition functions corresponding to two different orientations are given by expression \eqref{eq:torus} and $Z'(\Sigma_1)=R\varepsilon(e_d\cdot e_c\cdot e_b\cdot e_a)B^{ac}B^{bd}$. By relabelling $(d,c,b,a) \to (a,b,c,d)$ and using the bilinear form symmetry it is established the two invariants are indeed equal.

\begin{figure}
\centering
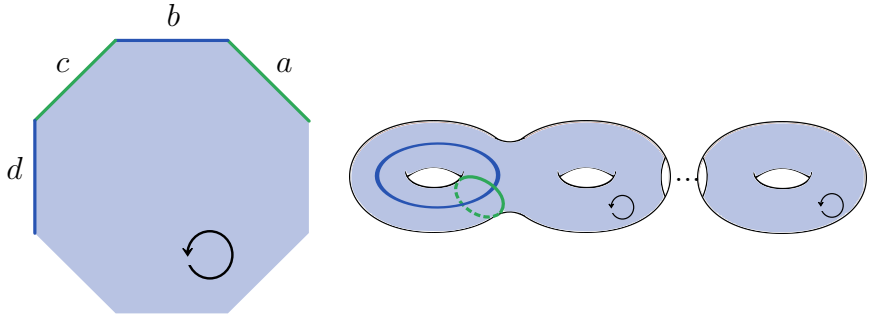  
\caption{\emph{Building $\Sigma_g$.}  A genus $g$ surface is constructed from a disk with $4g$ boundary edges (internal edges are omitted). Edges are identified following the pattern shown on the right: $a \leftrightarrow c$, $b \leftrightarrow d$.  On the left, it can be seen how the glued edges give rise to curves on the surface.} 
\label{fig:4gpoly}
\end{figure} 

The classification of FHK state sum models gives an explicit expression for $Z(\Sigma_g)$. This is based on the following calculations for the partition function in the case of simple algebras. For $A=\Mb_n(\Cb)$, choose as a basis the elementary matrices $\lbrace e_{lm}\rbrace_{l,m=1,n}$. Then for a Frobenius form \eqref{eq:cxfrob} the element $z$ is given by $z=R^{-2}n^{-2}\sum_{lm,rs}e_{lm}e_{rs}e_{ml}e_{sr}=R^{-2}n^{-2}1$. This gives the partition function 
\begin{align}\label{eq:partitionfunction}
Z(\Sigma_g, \Mb_n(\Cb))=R^{2-2g}n^{2-2g},
\end{align}
a result also found in \cite{Lauda}. The same conclusion holds for $\Mb_{n}(\Rb)$, now with $R\in\Rb$. For the case of $\Mb_{n}(\Cb_{\Rb})$, the element $z$ again takes the form $z=R^{-2}n^{-2}1$ but it produces a new partition function
\begin{align}
Z(\Sigma_g, \Mb_n(\Cb_{\Rb}))=2R^{2-2g}n^{2-2g}
\end{align}
due to the extra factor of $|\Cb_{\Rb}|=2$ present in the Frobenius form \eqref{eq:realfrob}. Further details of this calculation are explained in the more general example~\ref{ex:complex_crossing}.

Finally, for $\Mb_n(\Hb_{\Rb})$ a calculation shows that  $z=4^{-1}R^{-2}n^{-2}1$. Full details can be found in example~\ref{ex:quaternions_crossing}. The partition function reads
\begin{align}
Z(\Sigma_g, \Mb_n(\Hb_{\Rb}))=2^{2-2g}R^{2-2g}n^{2-2g}.
\end{align}
Given the information gathered above, the most general form of an invariant from a symmetric Frobenius algebra can be stated. 
\begin{theorem} \label{lem:inv} 
Let $A$ be a symmetric special Frobenius algebra over the field $k=\Cb$ or $\Rb$, as in theorem~\ref{theo:state_sum_algebra}. The topological invariant $Z(\Sigma_g)$ constructed from $A$ and an orientable surface $\Sigma_g$ is 
\begin{align} 
Z(\Sigma_g)=R^{2-2g}\sum\limits_{i=1}^N  n_i^{2-2g}
\end{align}
if $k=\Cb$ or
\begin{align} 
Z(\Sigma_g)=R^{2-2g}\sum\limits_{i=1}^N f(i,g) n_i^{2-2g}, \hspace{4mm}
f(i,g)=\begin{cases} 1 & (D_i = \Rb) \\ 2 & (D_i=\Cb_{\Rb}) \\ 2^{2-2g} & (D_i=\Hb_{\Rb})\end{cases}
\end{align}
if $k=\Rb$.
\end{theorem}
Another example of a Frobenius algebra is given by the complex group algebra. Recall an algebra can be built from any finite group $G$ by taking formal linear combinations of the group elements. This algebra, denoted $\Cb G$, has elements $f = \sum_{h \in G} f(h)h$, $f(h)\in \Cb$ and product defined according to
\begin{align} \label{eq:prod_group}
(f \cdot f') (h) = \sum_{l \in G}f(l)f'(l^{-1}h).
\end{align}
A Frobenius form is $\varepsilon(f)=R|G| f(1)$, where $|G|$ is the order of the group. This form is the unique symmetric special Frobenius form such that $R\beta=1$.
 The Peter-Weyl decomposition~\cite{Dieck} gives an isomorphism with a complex matrix algebra satisfying the conditions of theorem~\ref{theo:state_sum_algebra}. The general form of the invariant associated with the group algebra is therefore
\begin{align}\label{eq:group_inv}
Z(\Sigma_g)=R^{2-2g}\sum_{i \in I}(\dim i)^{2-2g},
\end{align}  
where each $i$ labels an irreducible group representation, a result that is given for a Lie group in \cite{witten}.  Expression \eqref{eq:group_inv} agrees with the results of \cite{Fukuma} when $R=1$.

\section{Planar and spherical state sum models} \label{sec:diagram}

\subsection{Planar models}

A more general algebraic framework can be used for state sum models if a more sophisticated method to define the weight of the model is employed. In this generalisation some of the conditions on the data of a naive state sum model are relaxed. This section describes this generalisation in the simplest case of a planar model, which is a state sum model on a portion of the plane $\Rb^2$.

The new framework uses a diagrammatic calculus to determine the partition function. The first step is to construct the graph dual to the triangulation. A distinguished role is played by the horizontal direction and it is assumed that the edges of the dual graph are not horizontal at the boundary of the manifold or at the vertices of the graph. Then an amplitude for this graph is determined by associating a local factor for each vertex and also a local factor for each point on a line at which the vertical height is a maximum or minimum.

The algebraic data is again non-degenerate $C$, $B$ and $R$, the generalisation being the replacement of the symmetry requirements \eqref{eq:C_cycle} and \eqref{eq:metric} on $C_{abc}$ and $B^{ab}$ with the one equation
\begin{align}\label{eq:BCequation}
C_{abc}\,B^{cd}=B^{de}C_{eab}.
\end{align}
The matrix $B^{ab}$ is no longer required to be symmetric. Since it is non-degenerate it has an inverse $B_{ab}$ defined by \eqref{eq:snake}.
Using the inverse, equation \eqref{eq:BCequation} can equivalently be written as either of the two equations
\begin{align} \label{eq:gen_cycle}
C_{eab}\,B_{dc}B^{de}=C_{abc}=C_{bce}\,B_{ad}B^{ed}.
\end{align}
Note that if $B$ is symmetric, condition \eqref{eq:gen_cycle} reduces to cyclicity as presented in \eqref{eq:C_cycle}.

First it is explained how the the building blocks, the maps $C$, $B$ and $B^{-1}$, are written in diagrammatic form. This is depicted below:
$$
\begingroup%
  \makeatletter%
  \providecommand\color[2][]{%
    \errmessage{(Inkscape) Color is used for the text in Inkscape, but the package 'color.sty' is not loaded}%
    \renewcommand\color[2][]{}%
  }%
  \providecommand\transparent[1]{%
    \errmessage{(Inkscape) Transparency is used (non-zero) for the text in Inkscape, but the package 'transparent.sty' is not loaded}%
    \renewcommand\transparent[1]{}%
  }%
  \providecommand\rotatebox[2]{#2}%
  \ifx\svgwidth\undefined%
    \setlength{\unitlength}{372.1612793bp}%
    \ifx\svgscale\undefined%
      \relax%
    \else%
      \setlength{\unitlength}{\unitlength * \real{\svgscale}}%
    \fi%
  \else%
    \setlength{\unitlength}{\svgwidth}%
  \fi%
  \global\let\svgwidth\undefined%
  \global\let\svgscale\undefined%
  \makeatother%
  \begin{picture}(1,0.11204979)%
    \put(0,0){\includegraphics[width=\unitlength]{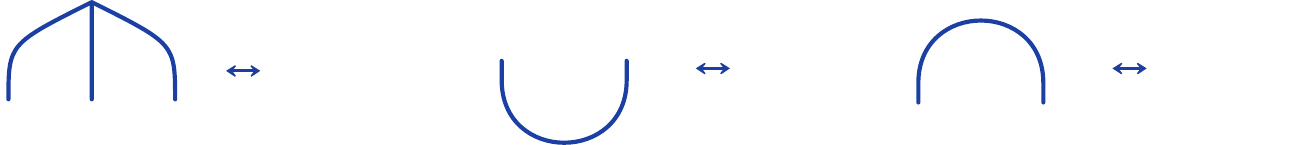}}%
    \put(-0.00076621,0.00527968){\color[rgb]{0,0,0}\makebox(0,0)[lb]{\smash{$a$}}}%
    \put(0.06456274,0.00512119){\color[rgb]{0,0,0}\makebox(0,0)[lb]{\smash{$b$}}}%
    \put(0.12882714,0.00563156){\color[rgb]{0,0,0}\makebox(0,0)[lb]{\smash{$c$}}}%
    \put(0.22524845,0.04972964){\color[rgb]{0,0,0}\makebox(0,0)[lb]{\smash{$C_{abc}$}}}%
    \put(0.33144624,0.0523929){\color[rgb]{0,0,0}\makebox(0,0)[lb]{\smash{,}}}%
    \put(0.3821534,0.07947112){\color[rgb]{0,0,0}\makebox(0,0)[lb]{\smash{$a$}}}%
    \put(0.47650841,0.08079993){\color[rgb]{0,0,0}\makebox(0,0)[lb]{\smash{$b$}}}%
    \put(0.5900319,0.05210918){\color[rgb]{0,0,0}\makebox(0,0)[lb]{\smash{$B^{ab}$}}}%
    \put(0.65258498,0.05184316){\color[rgb]{0,0,0}\makebox(0,0)[lb]{\smash{,}}}%
    \put(0.70546884,0.00235776){\color[rgb]{0,0,0}\makebox(0,0)[lb]{\smash{$a$}}}%
    \put(0.80141362,0.00255541){\color[rgb]{0,0,0}\makebox(0,0)[lb]{\smash{$b$}}}%
    \put(0.91247282,0.05210918){\color[rgb]{0,0,0}\makebox(0,0)[lb]{\smash{$B_{ab}$}}}%
    \put(0.97690825,0.05116912){\color[rgb]{0,0,0}\makebox(0,0)[lb]{\smash{.}}}%
  \end{picture}%
\endgroup%

$$
 The defining relation \eqref{eq:snake} between $B$ and its inverse is translated into the snake identity
$$
\hspace{15mm}
\begingroup%
  \makeatletter%
  \providecommand\color[2][]{%
    \errmessage{(Inkscape) Color is used for the text in Inkscape, but the package 'color.sty' is not loaded}%
    \renewcommand\color[2][]{}%
  }%
  \providecommand\transparent[1]{%
    \errmessage{(Inkscape) Transparency is used (non-zero) for the text in Inkscape, but the package 'transparent.sty' is not loaded}%
    \renewcommand\transparent[1]{}%
  }%
  \providecommand\rotatebox[2]{#2}%
  \ifx\svgwidth\undefined%
    \setlength{\unitlength}{204.21323242bp}%
    \ifx\svgscale\undefined%
      \relax%
    \else%
      \setlength{\unitlength}{\unitlength * \real{\svgscale}}%
    \fi%
  \else%
    \setlength{\unitlength}{\svgwidth}%
  \fi%
  \global\let\svgwidth\undefined%
  \global\let\svgscale\undefined%
  \makeatother%
  \begin{picture}(1,0.26728199)%
    \put(0,0){\includegraphics[width=\unitlength]{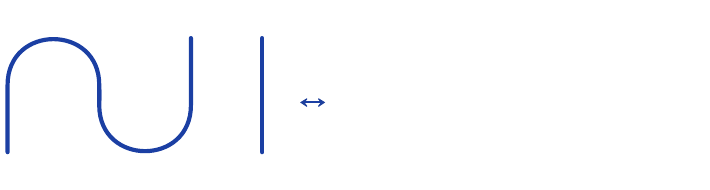}}%
    \put(-0.00139637,0.00403609){\color[rgb]{0,0,0}\makebox(0,0)[lb]{\smash{$a$}}}%
    \put(0.25583694,0.23586379){\color[rgb]{0,0,0}\makebox(0,0)[lb]{\smash{$b$}}}%
    \put(0.35507833,0.23665762){\color[rgb]{0,0,0}\makebox(0,0)[lb]{\smash{$b$}}}%
    \put(0.35746016,0.00483003){\color[rgb]{0,0,0}\makebox(0,0)[lb]{\smash{$a$}}}%
    \put(0.50191845,0.10878811){\color[rgb]{0,0,0}\makebox(0,0)[lb]{\smash{$B_{ac}B^{cb}=\delta_{a}^b$}}}%
    \put(0.30267889,0.10962891){\color[rgb]{0,0,0}\makebox(0,0)[lb]{\smash{$=$}}}%
  \end{picture}%
\endgroup%

\label{fig:inverse_metric}
$$
Either side of \eqref{eq:BCequation} can be taken as the definition of $C_{ab}{}^{d}$, the components of a multiplication map $m$ in equation \eqref{eq:multiplicationmap}. The diagrammatic counterpart is below. Similar expressions are used to define a vertex with two or three legs pointing upwards.
$$
\centering
\begingroup%
  \makeatletter%
  \providecommand\color[2][]{%
    \errmessage{(Inkscape) Color is used for the text in Inkscape, but the package 'color.sty' is not loaded}%
    \renewcommand\color[2][]{}%
  }%
  \providecommand\transparent[1]{%
    \errmessage{(Inkscape) Transparency is used (non-zero) for the text in Inkscape, but the package 'transparent.sty' is not loaded}%
    \renewcommand\transparent[1]{}%
  }%
  \providecommand\rotatebox[2]{#2}%
  \ifx\svgwidth\undefined%
    \setlength{\unitlength}{275.20251465bp}%
    \ifx\svgscale\undefined%
      \relax%
    \else%
      \setlength{\unitlength}{\unitlength * \real{\svgscale}}%
    \fi%
  \else%
    \setlength{\unitlength}{\svgwidth}%
  \fi%
  \global\let\svgwidth\undefined%
  \global\let\svgscale\undefined%
  \makeatother%
  \begin{picture}(1,0.20128155)%
    \put(0,0){\includegraphics[width=\unitlength]{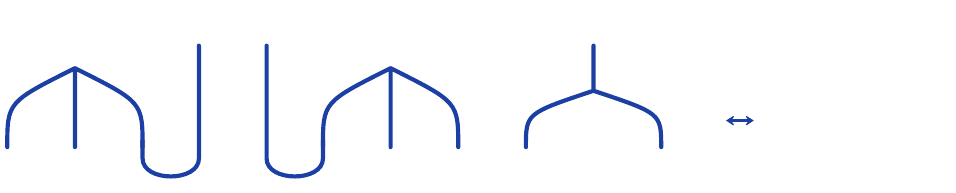}}%
    \put(-0.00103616,0.00299496){\color[rgb]{0,0,0}\makebox(0,0)[lb]{\smash{$a$}}}%
    \put(0.06965996,0.0035841){\color[rgb]{0,0,0}\makebox(0,0)[lb]{\smash{$b$}}}%
    \put(0.20044757,0.17855681){\color[rgb]{0,0,0}\makebox(0,0)[lb]{\smash{$d$}}}%
    \put(0.27114364,0.17796767){\color[rgb]{0,0,0}\makebox(0,0)[lb]{\smash{$d$}}}%
    \put(0.4019313,0.00299496){\color[rgb]{0,0,0}\makebox(0,0)[lb]{\smash{$a$}}}%
    \put(0.47203823,0.0035841){\color[rgb]{0,0,0}\makebox(0,0)[lb]{\smash{$b$}}}%
    \put(0.54332344,0.0035841){\color[rgb]{0,0,0}\makebox(0,0)[lb]{\smash{$a$}}}%
    \put(0.68294807,0.00417325){\color[rgb]{0,0,0}\makebox(0,0)[lb]{\smash{$b$}}}%
    \put(0.612252,0.17796776){\color[rgb]{0,0,0}\makebox(0,0)[lb]{\smash{$d$}}}%
    \put(0.83082062,0.06779968){\color[rgb]{0,0,0}\makebox(0,0)[lb]{\smash{$C_{ab}{}^d$}}}%
    \put(0.2302863,0.07204744){\color[rgb]{0,0,0}\makebox(0,0)[lb]{\smash{$=$}}}%
    \put(0.50190606,0.07170235){\color[rgb]{0,0,0}\makebox(0,0)[lb]{\smash{$=$}}}%
  \end{picture}%
\endgroup%

$$ 
Equation \eqref{eq:gen_cycle} can now be easily described -- note that keeping track of the index order is essential:
$$
\hspace{1mm}
\begingroup%
  \makeatletter%
  \providecommand\color[2][]{%
    \errmessage{(Inkscape) Color is used for the text in Inkscape, but the package 'color.sty' is not loaded}%
    \renewcommand\color[2][]{}%
  }%
  \providecommand\transparent[1]{%
    \errmessage{(Inkscape) Transparency is used (non-zero) for the text in Inkscape, but the package 'transparent.sty' is not loaded}%
    \renewcommand\transparent[1]{}%
  }%
  \providecommand\rotatebox[2]{#2}%
  \ifx\svgwidth\undefined%
    \setlength{\unitlength}{414.00576172bp}%
    \ifx\svgscale\undefined%
      \relax%
    \else%
      \setlength{\unitlength}{\unitlength * \real{\svgscale}}%
    \fi%
  \else%
    \setlength{\unitlength}{\svgwidth}%
  \fi%
  \global\let\svgwidth\undefined%
  \global\let\svgscale\undefined%
  \makeatother%
  \begin{picture}(1,0.09850453)%
    \put(0,0){\includegraphics[width=\unitlength]{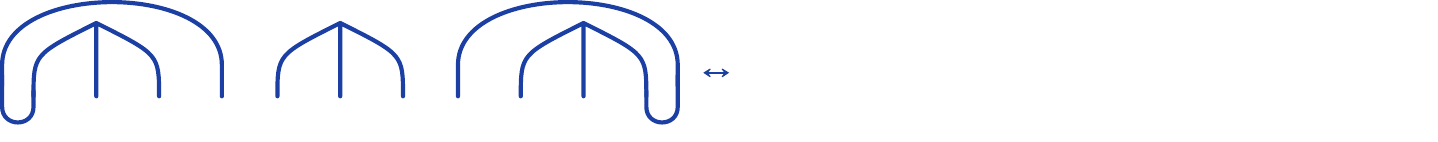}}%
    \put(0.18716547,0.00216457){\color[rgb]{0,0,0}\makebox(0,0)[lb]{\smash{$a$}}}%
    \put(0.23026896,0.00199084){\color[rgb]{0,0,0}\makebox(0,0)[lb]{\smash{$b$}}}%
    \put(0.27411648,0.00205801){\color[rgb]{0,0,0}\makebox(0,0)[lb]{\smash{$c$}}}%
    \put(0.31107714,0.0021924){\color[rgb]{0,0,0}\makebox(0,0)[lb]{\smash{$a$}}}%
    \put(0.3563567,0.00215307){\color[rgb]{0,0,0}\makebox(0,0)[lb]{\smash{$b$}}}%
    \put(0.4000697,0.0022874){\color[rgb]{0,0,0}\makebox(0,0)[lb]{\smash{$c$}}}%
    \put(0.29154251,0.04504916){\color[rgb]{0,0,0}\makebox(0,0)[lb]{\smash{$=$}}}%
    \put(0.521491,0.04118448){\color[rgb]{0,0,0}\makebox(0,0)[lb]{\smash{$C_{eab}B_{dc}B^{de}=C_{abc}=C_{bce}B_{ad}B^{ed}$}}}%
    \put(0.91762081,0.04118448){\color[rgb]{0,0,0}\makebox(0,0)[lb]{\smash{.}}}%
    \put(0.06159402,0.00253767){\color[rgb]{0,0,0}\makebox(0,0)[lb]{\smash{$a$}}}%
    \put(0.10469751,0.00236394){\color[rgb]{0,0,0}\makebox(0,0)[lb]{\smash{$b$}}}%
    \put(0.14854502,0.00243111){\color[rgb]{0,0,0}\makebox(0,0)[lb]{\smash{$c$}}}%
    \put(0.16594039,0.04504916){\color[rgb]{0,0,0}\makebox(0,0)[lb]{\smash{$=$}}}%
  \end{picture}%
\endgroup%

\label{fig:diagram_cycle}
$$ 

The data $C_{abc}$ and $B^{ab}$ together with the vertex amplitude $R\in k$ determine a new type of state sum model called a diagrammatic state sum model. This is a generalisation of the naive state sum model construction of \S\ref{sec:lattice_tft}.
Starting with a triangulation of a compact subset  $M\subset\Rb^2$, a state sum model for $M$ is constructed from the planar graph $G$ formed by the dual vertices and dual edges. Given fixed states on the boundary edges of $M$, the graph is evaluated to give the quantum amplitude $|G|\in k$. Simple examples for the $M$ consisting of one and two triangles are shown in figure \ref{fig:diag-ssm}.

\begin{figure}
\centering
\begin{subfigure}[t!]{0.47\textwidth}
                	\centering
		\vspace{6mm}
		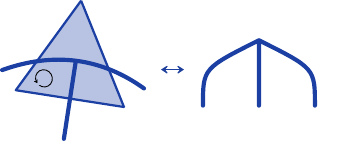 
		\vspace{4mm}
		\caption{\emph{Triangle amplitude.} The analogue of figure~\ref{fig:tri_label}. The diagram is determined by the dual graph with a choice of legs pointing upwards or downwards.}
\end{subfigure}
\hspace{5mm}
\begin{subfigure}[t!]{0.47\textwidth}
                	\centering
		\vspace{2mm}
		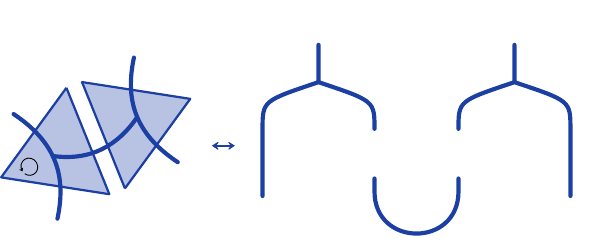
		\vspace{2.4mm}
		\caption{\emph{Gluing triangles.} The lack of rotational symmetry -- since the allowed homeomorphisms must preserve the boundary -- means that matrix $B$ is not assumed to be symmetric.}
		\label{fig:diag-ssm-glue}
\end{subfigure}
\caption{ }
\label{fig:diag-ssm}
\end{figure} 

Note that for the state sum model to be well-defined, the dual edges on the boundary have to be pointing either upwards or downwards. Due to the identities for $C$ and $B$, the interior of the graph can be moved by a homeomorphism (fixing the boundary) to any convenient graph in order to construct the required algebraic expression. Thus on figure \ref{fig:diag-ssm-glue} the left-hand vertex has been perturbed so that both vertices correspond to the multiplication map.
Note that the plane $\Rb^2$ is considered to have a standard orientation, so that $M$ is an oriented manifold. The general formula for the partition function of this diagrammatic state sum model is
\begin{equation}Z(M)=R^V |G|\label{eq:diag-partition}\end{equation}
with $V$ the number of interior vertices.

Now the Pachner moves are introduced. A Pachner move preserves the boundary of a triangulation and it is assumed that the corresponding dual edges do not change in a neighbourhood of the boundary, so remain either upward or downward-pointing.
\begin{definition} A planar state sum model is a diagrammatic state sum model for any compact $M\subset \Rb^2$ satisfying the Pachner moves.
\end{definition}

The planar state sum models depend on the details of the diagram in the neighbourhood of the boundary. Thus the partition function of a disk is no longer symmetric under cyclic permutations of the boundary edges, but has a more refined mapping property that generalises \eqref{eq:gen_cycle}. Note that this is why the diagrammatic state sum models escape the conclusion of \S\ref{sec:lattice_tft} that $B$ is symmetric for the naive models. These mappings of boundaries and the boundary data are not studied further in this paper. It will be assumed that any mapping of surfaces is the identity mapping in a neighbourhood of the boundary. 

The result below is a refinement of theorem~\ref{theo:state_sum_algebra} and its proof develops the properties of the graphical calculus.
\begin{theorem} \label{theo:diagram}
Non-degenerate diagrammatic state sum model data $(C,B,R)$ determine a planar state sum if and only if the multiplication map $m$, the bilinear form $B$ and the distinguished element $\beta = m(B)$ determine on $A$ the structure of a special Frobenius algebra with identity element $1=R\beta$.
\end{theorem}

\begin{proof}
The proof of theorem~\ref{theo:state_sum_algebra} will be followed very closely. The essential difference relies on the translation of Pachner moves into the new diagrammatic model. 

Suppose that $(C,B,R)$ is the data for a planar state sum model. As before, define $A$ to be the vector space spanned by $S$.
 Consider the 2-2 move depicted in figure~\ref{fig:Pach1}. Its graphical counterpart is given below.
$$
\begingroup%
  \makeatletter%
  \providecommand\color[2][]{%
    \errmessage{(Inkscape) Color is used for the text in Inkscape, but the package 'color.sty' is not loaded}%
    \renewcommand\color[2][]{}%
  }%
  \providecommand\transparent[1]{%
    \errmessage{(Inkscape) Transparency is used (non-zero) for the text in Inkscape, but the package 'transparent.sty' is not loaded}%
    \renewcommand\transparent[1]{}%
  }%
  \providecommand\rotatebox[2]{#2}%
  \ifx\svgwidth\undefined%
    \setlength{\unitlength}{396.30935059bp}%
    \ifx\svgscale\undefined%
      \relax%
    \else%
      \setlength{\unitlength}{\unitlength * \real{\svgscale}}%
    \fi%
  \else%
    \setlength{\unitlength}{\svgwidth}%
  \fi%
  \global\let\svgwidth\undefined%
  \global\let\svgscale\undefined%
  \makeatother%
  \begin{picture}(1,0.18927401)%
    \put(0,0){\includegraphics[width=\unitlength]{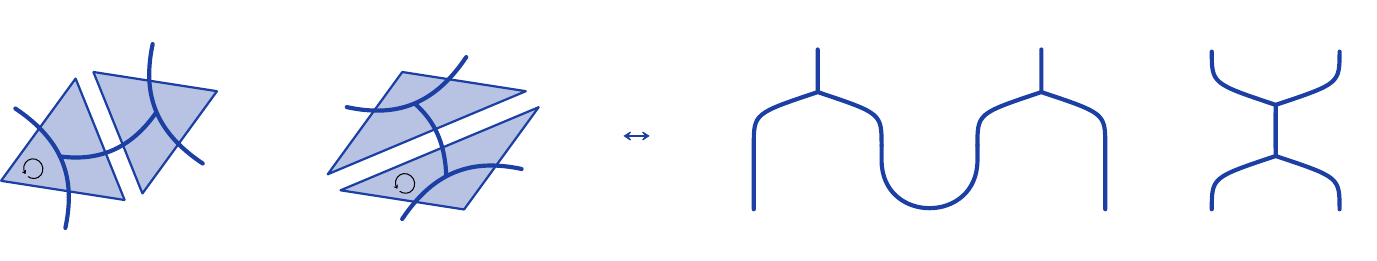}}%
    \put(0.05617569,0.01967112){\color[rgb]{0,0,0}\makebox(0,0)[lb]{\smash{$b$}}}%
    \put(0.01772008,0.11621934){\color[rgb]{0,0,0}\makebox(0,0)[lb]{\smash{$a$}}}%
    \put(0.11876835,0.13871992){\color[rgb]{0,0,0}\makebox(0,0)[lb]{\smash{$d$}}}%
    \put(0.14372358,0.08021829){\color[rgb]{0,0,0}\makebox(0,0)[lb]{\smash{$c$}}}%
    \put(0.30204618,0.0159892){\color[rgb]{0,0,0}\makebox(0,0)[lb]{\smash{$b$}}}%
    \put(0.3785483,0.07408178){\color[rgb]{0,0,0}\makebox(0,0)[lb]{\smash{$c$}}}%
    \put(0.34663834,0.13708348){\color[rgb]{0,0,0}\makebox(0,0)[lb]{\smash{$d$}}}%
    \put(0.25827223,0.12153764){\color[rgb]{0,0,0}\makebox(0,0)[lb]{\smash{$a$}}}%
    \put(0.58800877,0.17267542){\color[rgb]{0,0,0}\makebox(0,0)[lb]{\smash{$a$}}}%
    \put(0.54218928,0.00207973){\color[rgb]{0,0,0}\makebox(0,0)[lb]{\smash{$b$}}}%
    \put(0.79746909,0.00289789){\color[rgb]{0,0,0}\makebox(0,0)[lb]{\smash{$c$}}}%
    \put(0.7512405,0.17349364){\color[rgb]{0,0,0}\makebox(0,0)[lb]{\smash{$d$}}}%
    \put(0.83469735,0.08553665){\color[rgb]{0,0,0}\makebox(0,0)[lb]{\smash{$=$}}}%
    \put(0.87397121,0.17349364){\color[rgb]{0,0,0}\makebox(0,0)[lb]{\smash{$a$}}}%
    \put(0.96724663,0.17349364){\color[rgb]{0,0,0}\makebox(0,0)[lb]{\smash{$d$}}}%
    \put(0.87519854,0.00207973){\color[rgb]{0,0,0}\makebox(0,0)[lb]{\smash{$b$}}}%
    \put(0.96806472,0.00207973){\color[rgb]{0,0,0}\makebox(0,0)[lb]{\smash{$c$}}}%
    \put(0.19422986,0.08959682){\color[rgb]{0,0,0}\makebox(0,0)[lb]{\smash{$=$}}}%
  \end{picture}%
\endgroup%

\label{fig:diagram_pach22}
$$
\noindent 
Using first the non-degeneracy of $B$ by contracting each side with $B_{ea}$ and second the definition of the multiplication components, one can simplify the identity above to obtain
$$
\centering
\begingroup%
  \makeatletter%
  \providecommand\color[2][]{%
    \errmessage{(Inkscape) Color is used for the text in Inkscape, but the package 'color.sty' is not loaded}%
    \renewcommand\color[2][]{}%
  }%
  \providecommand\transparent[1]{%
    \errmessage{(Inkscape) Transparency is used (non-zero) for the text in Inkscape, but the package 'transparent.sty' is not loaded}%
    \renewcommand\transparent[1]{}%
  }%
  \providecommand\rotatebox[2]{#2}%
  \ifx\svgwidth\undefined%
    \setlength{\unitlength}{130.59786377bp}%
    \ifx\svgscale\undefined%
      \relax%
    \else%
      \setlength{\unitlength}{\unitlength * \real{\svgscale}}%
    \fi%
  \else%
    \setlength{\unitlength}{\svgwidth}%
  \fi%
  \global\let\svgwidth\undefined%
  \global\let\svgscale\undefined%
  \makeatother%
  \begin{picture}(1,0.47600143)%
    \put(0,0){\includegraphics[width=\unitlength]{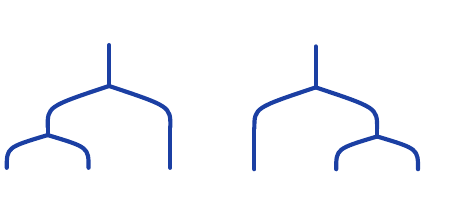}}%
    \put(-0.00218345,0.00755238){\color[rgb]{0,0,0}\makebox(0,0)[lb]{\smash{$a$}}}%
    \put(0.18031007,0.00631109){\color[rgb]{0,0,0}\makebox(0,0)[lb]{\smash{$b$}}}%
    \put(0.35907908,0.00755238){\color[rgb]{0,0,0}\makebox(0,0)[lb]{\smash{$c$}}}%
    \put(0.54405583,0.00631109){\color[rgb]{0,0,0}\makebox(0,0)[lb]{\smash{$a$}}}%
    \put(0.72468714,0.00755256){\color[rgb]{0,0,0}\makebox(0,0)[lb]{\smash{$b$}}}%
    \put(0.90407698,0.00817339){\color[rgb]{0,0,0}\makebox(0,0)[lb]{\smash{$c$}}}%
    \put(0.22254522,0.42811466){\color[rgb]{0,0,0}\makebox(0,0)[lb]{\smash{$d$}}}%
    \put(0.67824339,0.42664793){\color[rgb]{0,0,0}\makebox(0,0)[lb]{\smash{$d$}}}%
    \put(0.43266288,0.15363947){\color[rgb]{0,0,0}\makebox(0,0)[lb]{\smash{$=$}}}%
  \end{picture}%
\endgroup%

\label{fig:diagram_mult2}
$$
\noindent The multiplication map is therefore associative, as in theorem~\ref{theo:state_sum_algebra}. 

Next, \eqref{eq:BCequation} implies
\begin{align} \label{eq:associative_bilinear}
B^{-1}(e_a \cdot e_b, e_c)=B^{-1}(e_a, e_b \cdot e_c).
\end{align}
 This means that a functional $\varepsilon\colon A\to k$ can be defined by $\varepsilon(x)= B^{-1}(x,1)$. However, there are no additional symmetry requirements that $\varepsilon$ must obey. 

To simplify the exposition of the 1-3 Pachner move, a 2-2 move was performed on the two left-most triangles of figure~\ref{fig:Pach2}. The relation
$$
\begingroup%
  \makeatletter%
  \providecommand\color[2][]{%
    \errmessage{(Inkscape) Color is used for the text in Inkscape, but the package 'color.sty' is not loaded}%
    \renewcommand\color[2][]{}%
  }%
  \providecommand\transparent[1]{%
    \errmessage{(Inkscape) Transparency is used (non-zero) for the text in Inkscape, but the package 'transparent.sty' is not loaded}%
    \renewcommand\transparent[1]{}%
  }%
  \providecommand\rotatebox[2]{#2}%
  \ifx\svgwidth\undefined%
    \setlength{\unitlength}{399.08691406bp}%
    \ifx\svgscale\undefined%
      \relax%
    \else%
      \setlength{\unitlength}{\unitlength * \real{\svgscale}}%
    \fi%
  \else%
    \setlength{\unitlength}{\svgwidth}%
  \fi%
  \global\let\svgwidth\undefined%
  \global\let\svgscale\undefined%
  \makeatother%
  \begin{picture}(1,0.32747391)%
    \put(0,0){\includegraphics[width=\unitlength]{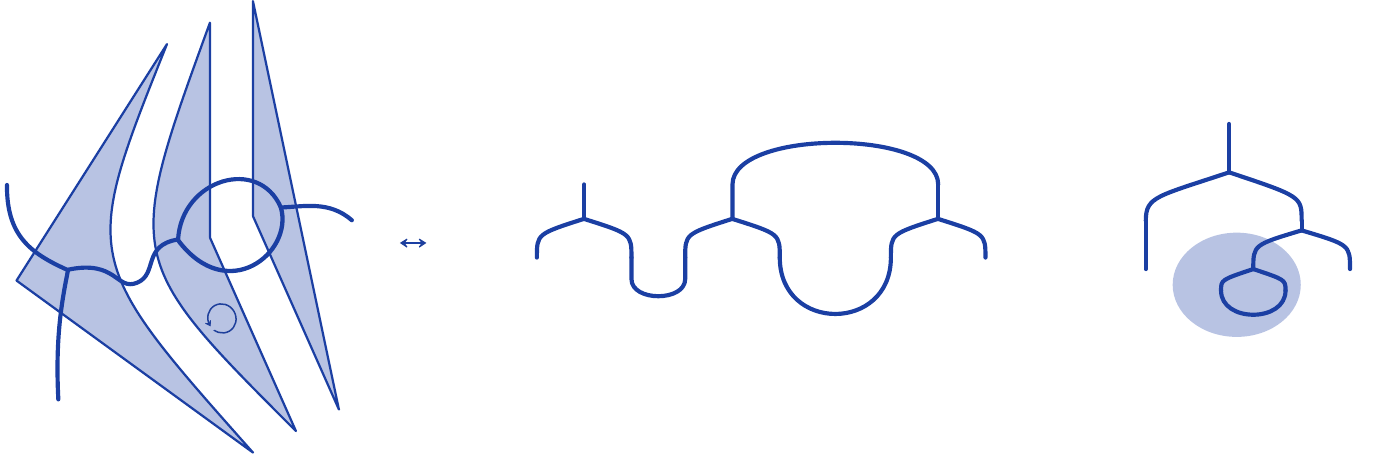}}%
    \put(-0.00071452,0.21049615){\color[rgb]{0,0,0}\makebox(0,0)[lb]{\smash{$a$}}}%
    \put(0.03625468,0.00511142){\color[rgb]{0,0,0}\makebox(0,0)[lb]{\smash{$b$}}}%
    \put(0.25441362,0.1434641){\color[rgb]{0,0,0}\makebox(0,0)[lb]{\smash{$c$}}}%
    \put(0.41529072,0.21496497){\color[rgb]{0,0,0}\makebox(0,0)[lb]{\smash{$a$}}}%
    \put(0.3819778,0.10730741){\color[rgb]{0,0,0}\makebox(0,0)[lb]{\smash{$b$}}}%
    \put(0.70576299,0.10649476){\color[rgb]{0,0,0}\makebox(0,0)[lb]{\smash{$c$}}}%
    \put(0.32957081,0.14468283){\color[rgb]{0,0,0}\makebox(0,0)[lb]{\smash{$R$}}}%
    \put(0.88045271,0.25762179){\color[rgb]{0,0,0}\makebox(0,0)[lb]{\smash{$a$}}}%
    \put(0.82113946,0.10040108){\color[rgb]{0,0,0}\makebox(0,0)[lb]{\smash{$b$}}}%
    \put(0.8552649,0.11136991){\color[rgb]{0,0,0}\makebox(0,0)[lb]{\smash{$R$}}}%
    \put(0.96861,0.09999476){\color[rgb]{0,0,0}\makebox(0,0)[lb]{\smash{$c$}}}%
    \put(0.76067961,0.14427877){\color[rgb]{0,0,0}\makebox(0,0)[lb]{\smash{$=$}}}%
  \end{picture}%
\endgroup%

\label{fig:diag_pach_31}
$$ 
\noindent is obtained. It was simplified using the definition of multiplication components and associativity. The 1-3 Pachner move predicts the expression above must equal $C_{bc}{}^a$. Since $C$ is assumed to be non-degenerate one concludes the highlighted new element, $R\beta$ with $\beta=e_a \cdot e_b\,B^{ab}$, must satisfy $R\beta=1$. Since $A$ has a unit it is an algebra and is therefore a  special  Frobenius algebra.

Conversely, given a special Frobenius algebra with multiplication $m$ and a linear functional $\varepsilon$, a non-degenerate bilinear form is defined by $B^{-1}= \varepsilon \circ m$, with property \eqref{eq:gen_cycle}. As previously stated, the fact the algebra is unital implies the non-degeneracy of $C$, while associativity and the relation $R\beta=1$ guarantee invariance under Pachner moves. The diagrammatic state sum model created is therefore planar.
\end{proof}

A point that is worth noting from the proof is that in the diagrammatic calculus, a power of $R$ is associated to every closed region in the diagram. If the diagram comes from a triangulation, then the closed regions are dual to the vertices of a triangulation.

It is also worth noting that having $\beta$ proportional to the identity is a non-trivial restriction on Frobenius algebras. The following arguments show that this condition implies the algebra must be separable. Note that some presentations of these state sum models \cite{Lauda,Runkel} assume from the outset the algebra is of this type. There are a number of equivalent definitions of the separability condition; the most convenient one for the purpose of this work is as follows \cite{KadisonStolin}, where the vector space $A\otimes A$ is a bimodule over $A$ with the actions $x\triangleright (u\otimes v)=(x\cdot u)\otimes v$ and $(u\otimes v)\triangleleft x=u\otimes(v\cdot x)$.
\begin{definition}[Separable algebra]
An algebra $A$ is called separable if there exists $t \in A \otimes A$ such that $x\triangleright t=t\triangleleft x$ for all $x \in A$ and $m(t)=1 \in A$.
\end{definition}

The relevance of this definition to the state sum models is given in the following lemma.
\begin{lemma}\label{separable} A special Frobenius algebra is a separable algebra.
\end{lemma} 

\begin{proof} Define $R\in k$ by $\beta=R^{-1}1$. Using the basis $\{e_a\}$ of the Frobenius algebra $A$ with Frobenius form $\varepsilon$, define $B_{ab}=\varepsilon(e_a\cdot e_b)$, $B^{ab}B_{bc}=\delta^a_c$, and set $t=R\,e_a\otimes e_b\,B^{ab}=RB$. Then the identity $\varepsilon (y\cdot e_a)\,e_b\,B^{ab} =y$ for all $y\in A$ follows. Using this identity twice, one finds $\varepsilon(y\cdot x\cdot e_a)\,e_b\,B^{ab}=y\cdot x=\varepsilon(y\cdot e_a)\,e_b\cdot x\,B^{ab}$, which can be depicted diagrammatically as 
$$
\begingroup%
  \makeatletter%
  \providecommand\color[2][]{%
    \errmessage{(Inkscape) Color is used for the text in Inkscape, but the package 'color.sty' is not loaded}%
    \renewcommand\color[2][]{}%
  }%
  \providecommand\transparent[1]{%
    \errmessage{(Inkscape) Transparency is used (non-zero) for the text in Inkscape, but the package 'transparent.sty' is not loaded}%
    \renewcommand\transparent[1]{}%
  }%
  \providecommand\rotatebox[2]{#2}%
  \ifx\svgwidth\undefined%
    \setlength{\unitlength}{201.36540557bp}%
    \ifx\svgscale\undefined%
      \relax%
    \else%
      \setlength{\unitlength}{\unitlength * \real{\svgscale}}%
    \fi%
  \else%
    \setlength{\unitlength}{\svgwidth}%
  \fi%
  \global\let\svgwidth\undefined%
  \global\let\svgscale\undefined%
  \makeatother%
  \begin{picture}(1,0.22623053)%
    \put(0,0){\includegraphics[width=\unitlength]{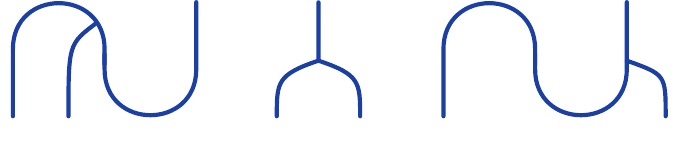}}%
    \put(-0.00141611,0.00836089){\color[rgb]{0,0,0}\makebox(0,0)[lb]{\smash{$y$}}}%
    \put(0.31426321,0.11159446){\color[rgb]{0,0,0}\makebox(0,0)[lb]{\smash{$=$}}}%
    \put(0.07804143,0.00836089){\color[rgb]{0,0,0}\makebox(0,0)[lb]{\smash{$x$}}}%
    \put(0.37600721,0.00836089){\color[rgb]{0,0,0}\makebox(0,0)[lb]{\smash{$y$}}}%
    \put(0.49519352,0.00836089){\color[rgb]{0,0,0}\makebox(0,0)[lb]{\smash{$x$}}}%
    \put(0.61437983,0.00836089){\color[rgb]{0,0,0}\makebox(0,0)[lb]{\smash{$y$}}}%
    \put(0.54684092,0.11165569){\color[rgb]{0,0,0}\makebox(0,0)[lb]{\smash{$=$}}}%
    \put(0.93221,0.00836089){\color[rgb]{0,0,0}\makebox(0,0)[lb]{\smash{$x$}}}%
    \put(0.99241506,0.1207863){\color[rgb]{0,0,0}\makebox(0,0)[lb]{\smash{.}}}%
  \end{picture}%
\endgroup%

$$
Then, the non-degeneracy of $\varepsilon$ guarantees that  $x \triangleright t=t \triangleleft x$ for all $x\in A$: 
$$
\begingroup%
  \makeatletter%
  \providecommand\color[2][]{%
    \errmessage{(Inkscape) Color is used for the text in Inkscape, but the package 'color.sty' is not loaded}%
    \renewcommand\color[2][]{}%
  }%
  \providecommand\transparent[1]{%
    \errmessage{(Inkscape) Transparency is used (non-zero) for the text in Inkscape, but the package 'transparent.sty' is not loaded}%
    \renewcommand\transparent[1]{}%
  }%
  \providecommand\rotatebox[2]{#2}%
  \ifx\svgwidth\undefined%
    \setlength{\unitlength}{113.8421875bp}%
    \ifx\svgscale\undefined%
      \relax%
    \else%
      \setlength{\unitlength}{\unitlength * \real{\svgscale}}%
    \fi%
  \else%
    \setlength{\unitlength}{\svgwidth}%
  \fi%
  \global\let\svgwidth\undefined%
  \global\let\svgscale\undefined%
  \makeatother%
  \begin{picture}(1,0.2656821)%
    \put(0,0){\includegraphics[width=\unitlength]{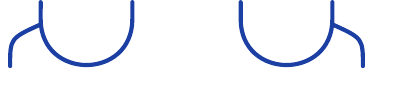}}%
    \put(0.42615874,0.16886726){\color[rgb]{0,0,0}\makebox(0,0)[lb]{\smash{$=$}}}%
    \put(0.88995869,0.00724001){\color[rgb]{0,0,0}\makebox(0,0)[lb]{\smash{$x$}}}%
    \put(0.97428595,0.20400362){\color[rgb]{0,0,0}\makebox(0,0)[lb]{\smash{.}}}%
    \put(-0.00250484,0.00724001){\color[rgb]{0,0,0}\makebox(0,0)[lb]{\smash{$x$}}}%
  \end{picture}%
\endgroup%

$$
Also, $m(t)=R\beta=1$.
\end{proof}

For a field $k$ of characteristic zero, separability for an algebra is equivalent to it being both finite dimensional and semisimple \cite{KadisonStolin,Aguiar}. Therefore, if $k=\Rb$ or $\Cb$ these Frobenius algebras are easily classified. 

Consider the complex algebra $A=\Mb_{n}(\Cb)$ with Frobenius form $\varepsilon(a)=\Tr(xa)$ for some fixed invertible element $x \in A$. This determines the non-degenerate bilinear form $B^{-1}(a,b)=\Tr(xab)$. Let $\lbrace e_{lm} \rbrace_{l,m=1,n}$ be the basis of elementary matrices such that $(e_{lm})_{rs}=\delta_{lr}\delta_{ms}$. Then $B$ must be given by
\begin{align}
B=\sum_{lm} e_{lm}x^{-1} \otimes e_{ml} \in A \otimes A.
\end{align}
The defining equation $B^{-1}B=1$ is satisfied since the cyclicity of the trace guarantees $\sum_{lm}\Tr(xae_{lm}x^{-1})e_{ml}=\sum_{lm}\Tr(ae_{lm})e_{ml}=a$ for all $a \in A$. Moreover, the distinguished element satisfies $\beta=\Tr(x^{-1})1$. This identity follows from noticing that $p(a)=\sum_{lm} e_{lm}ae_{ml}=\Tr(a)1$, where the map $p$ is proportional to a projector $A \to A$ with the centre of $A$, $\mathcal{Z}(A)$, as its image. Thus our example will define a planar state sum model if $R^{-1}=\Tr(x^{-1})$. This particular example will be used to prove the theorem below.

\begin{theorem} \label{theo:diagram_semi_simple}
A planar state sum model over the field $k = \Cb$ or $\Rb$ is isomorphic by a change of basis to one in which the algebra is a direct sum of
matrix algebras over $\Cb$ or division rings $\Rb,\Cb_{\Rb},\Hb_{\Rb}$ and the Frobenius form is determined by a fixed invertible element $x =\oplus_i x_i \in A$. For a complex algebra
\begin{align} \label{eq:complex_semi}
A = \bigoplus\limits_{i=1}^N \Mb_{n_i}(\Cb), 
\end{align}
the functional takes the form
\begin{align}
\varepsilon(a) = \sum\limits_{i=1}^N \Tr (x_ia_i).
\end{align}
The element $x$ must satisfy the relations $R\Tr(x^{-1}_i)=1$ for all $i=1,\cdots,N$. For a real algebra
\begin{align} \label{eq:real_semi}
A = \bigoplus\limits_{i=1}^N \Mb_{n_i}(D_i) \text{ with }D_i=\Rb,\Cb_{\Rb},\Hb_{\Rb}
\end{align}
the Frobenius form is given by
\begin{align}
\varepsilon(a) = \sum\limits_{i=1}^N \Real  \Tr(x_ia_i).
\end{align}
The element $x$ must satisfy the relations
\begin{align}
R^{-1}=
\begin{cases}
\Tr(x_i^{-1}) & (D_i = \Rb) \\
2\Tr(x_i^{-1}) &  (D_i = \Cb_{\Rb}) \\
4 \Real\Tr(x_i^{-1}) &  (D_i = \Hb_{\Rb})
\end{cases}
\end{align}
 for all $i=1,\cdots,N$.
\end{theorem}
\begin{proof}
The classification of Frobenius forms on an algebra~\cite{Lauda,Kock} shows that any two Frobenius forms $\varepsilon$, $\tilde{\varepsilon}$ are related by an invertible element $x \in A$ as $\varepsilon(a)=\tilde{\varepsilon}(xa)$. Thus, for the complex case, one can write
\begin{align}
\varepsilon(a)=\sum_i\Tr(x_ia_i)
\end{align}
using the decomposition $x=\oplus_i x_i$ and lemma \ref{lem:Frobenius_forms}. From the example of a simple matrix algebra previously studied, one concludes $\beta_i=\Tr(x^{-1}_i)1_i$, with $1_i$ the unit element in $\Mb_{n_i}(\Cb)$. Consequently, setting $R\beta=1$ gives the relations $R\Tr(x_i^{-1})=1$ for all $i$. 

As established in \S\ref{sec:lattice_tft}, $\Real\Tr$ is a Frobenius functional for a matrix algebra over a real division ring. Thus, for an algebra \eqref{eq:real_semi}, one can write
\begin{align}
\varepsilon(a)=\sum_i\Real\Tr(x_ia_i).
\end{align}
It is easy to verify the bilinear form $B$ associated with this Frobenius functional satisfies
\begin{align}
B=\sum_{i,lm,w_i}\,w_i\,e_{lm}^{i}\,x_i^{-1}\otimes \, w_i^{\ast} e_{ml}^i
\end{align}
using the basis defined in lemma~\ref{lem:Frobenius_forms}; one then finds 
\begin{align}
m(B)=\sum_i\sum_{w_i} w_i\Tr(x_i^{-1}) w_i^{\ast}1_i \hspace{1mm}.
\end{align}
For the identity $R\beta=1$ to hold it is therefore necessary to have $R^{-1}=\sum_{w_i}w_i\Tr(x_i^{-1}) w_i^{\ast}$ for all $i$. If $D_i=\Rb$ or $\Cb_{\Rb}$, then $w_i^{\ast}$ and $\Tr(x_i^{-1})$ commute, which means the expression reduces to $R^{-1}=\Tr(x_i^{-1})$ and $R^{-1}=2\Tr(x_i^{-1})$ respectively. If $D_i=\Hb_{\Rb}$, the expression reduces to $R^{-1}=4\Real\Tr(x_i^{-1})$ -- the non-real components of the trace are automatically cancelled.

\end{proof}

As one might expect, the study of state sum models done in  \S\ref{sec:lattice_tft} for the disk can be regarded as a special case of theorem \ref{theo:diagram_semi_simple}.

\begin{corollary} \label{cor:FHK}
An FHK state sum model on the disk over the field $k = \Cb$ or $\Rb$ is a planar state sum model in the conditions of theorem~\ref{theo:diagram_semi_simple} where the Frobenius form is symmetric. If the algebra is of the form \eqref{eq:complex_semi} then $x=\oplus_i Rn_i1_i$; if it is of the form \eqref{eq:real_semi} then $x=\oplus_i R|D_i|n_i1_i$. 
\end{corollary}
\begin{proof}
This is a special case of theorem~\ref{theo:diagram_semi_simple} where $\varepsilon$ must be symmetric. This means $x$ must be a central element and can, therefore, be written as $x=\oplus_i \mu_i 1_i$. The constants $\mu_i$ must be in $\Cb$ if the underlying field is $\Cb$ or if $D_i=\Cb_{\Rb}$; otherwise, they must be real numbers (recall that only real numbers commute with all the quaternions). Each of these constants must then satisfy $R^{-1}=\mu_i^{-1}n_i$ in the complex case or $R^{-1}=\mu_i^{-1}|D_i|n_i$ in the real one. In other words $x=\oplus_iRn_i1_i$ or $x=\oplus_iR|D_i|n_i1_i$, respectively.
\end{proof}
This result implies that the Frobenius form for an FHK state sum model is uniquely determined by the  algebra $A$ and the constant $R$.

\subsection{Spherical models}\label{sec:spherical}

Suppose that $M$ is a subset of the sphere, $M\subset S^2$, with a chosen orientation. Then a state sum model is defined for every orientation-preserving isomorphism of $S^2-\{p\}$ to $\Rb^2$, with $p$ the `point at infinity', which should be chosen not to lie in the dual graph of the triangulation of $M$. Moving $p$ around corresponds to the spherical move  \cite{BarrettWestbury}
\begin{equation}\label{eq:blob}
\begin{aligned}
\begingroup%
  \makeatletter%
  \providecommand\color[2][]{%
    \errmessage{(Inkscape) Color is used for the text in Inkscape, but the package 'color.sty' is not loaded}%
    \renewcommand\color[2][]{}%
  }%
  \providecommand\transparent[1]{%
    \errmessage{(Inkscape) Transparency is used (non-zero) for the text in Inkscape, but the package 'transparent.sty' is not loaded}%
    \renewcommand\transparent[1]{}%
  }%
  \providecommand\rotatebox[2]{#2}%
  \ifx\svgwidth\undefined%
    \setlength{\unitlength}{141.73972168bp}%
    \ifx\svgscale\undefined%
      \relax%
    \else%
      \setlength{\unitlength}{\unitlength * \real{\svgscale}}%
    \fi%
  \else%
    \setlength{\unitlength}{\svgwidth}%
  \fi%
  \global\let\svgwidth\undefined%
  \global\let\svgscale\undefined%
  \makeatother%
  \begin{picture}(1,0.58091433)%
    \put(0,0){\includegraphics[width=\unitlength]{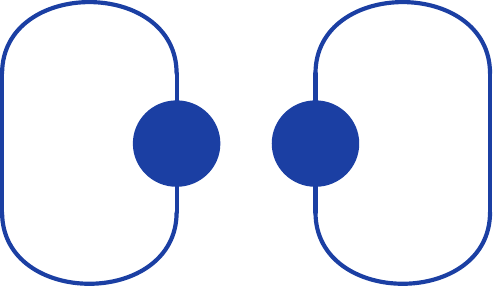}}%
    \put(0.47174085,0.2609378){\color[rgb]{0,0,0}\makebox(0,0)[lb]{\smash{$=$}}}%
  \end{picture}%
\endgroup%

\end{aligned}
\end{equation}
where $\begin{aligned}
\begingroup%
  \makeatletter%
  \providecommand\color[2][]{%
    \errmessage{(Inkscape) Color is used for the text in Inkscape, but the package 'color.sty' is not loaded}%
    \renewcommand\color[2][]{}%
  }%
  \providecommand\transparent[1]{%
    \errmessage{(Inkscape) Transparency is used (non-zero) for the text in Inkscape, but the package 'transparent.sty' is not loaded}%
    \renewcommand\transparent[1]{}%
  }%
  \providecommand\rotatebox[2]{#2}%
  \ifx\svgwidth\undefined%
    \setlength{\unitlength}{8.40000381bp}%
    \ifx\svgscale\undefined%
      \relax%
    \else%
      \setlength{\unitlength}{\unitlength * \real{\svgscale}}%
    \fi%
  \else%
    \setlength{\unitlength}{\svgwidth}%
  \fi%
  \global\let\svgwidth\undefined%
  \global\let\svgscale\undefined%
  \makeatother%
  \begin{picture}(1,1.61712173)%
    \put(0,0){\includegraphics[width=\unitlength]{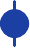}}%
  \end{picture}%
\endgroup%
\end{aligned}$ consists of a diagram that is the same on both sides of the equation. This move can be understood as making the arc on the left-hand side larger until it passes the point at infinity on the sphere, when it then re-enters the planar diagram as an arc on the right-hand side.

A sufficient condition that guarantees  \eqref{eq:blob} holds for any matrix representing $\begin{aligned}
\begingroup%
  \makeatletter%
  \providecommand\color[2][]{%
    \errmessage{(Inkscape) Color is used for the text in Inkscape, but the package 'color.sty' is not loaded}%
    \renewcommand\color[2][]{}%
  }%
  \providecommand\transparent[1]{%
    \errmessage{(Inkscape) Transparency is used (non-zero) for the text in Inkscape, but the package 'transparent.sty' is not loaded}%
    \renewcommand\transparent[1]{}%
  }%
  \providecommand\rotatebox[2]{#2}%
  \ifx\svgwidth\undefined%
    \setlength{\unitlength}{8.40000381bp}%
    \ifx\svgscale\undefined%
      \relax%
    \else%
      \setlength{\unitlength}{\unitlength * \real{\svgscale}}%
    \fi%
  \else%
    \setlength{\unitlength}{\svgwidth}%
  \fi%
  \global\let\svgwidth\undefined%
  \global\let\svgscale\undefined%
  \makeatother%
  \begin{picture}(1,1.61712173)%
    \put(0,0){\includegraphics[width=\unitlength]{blob.pdf}}%
  \end{picture}%
\endgroup%
\end{aligned}$ is 
\begin{equation} B_{ca}B^{cb}=B_{ac}B^{bc}.\label{eq:spherical}\end{equation} 

The meaning of \eqref{eq:spherical} is easier to understand in the context of Frobenius algebras. 
\begin{definition}[Nakayama automorphism]\label{Nakayama}
A Frobenius algebra has an automorphism $\sigma \colon A \to A$ determined uniquely by the relation $\varepsilon(x \cdot y)=\varepsilon(\sigma(y) \cdot x)$ for all $x,y \in A$. 
\end{definition}

\begin{lemma}\label{lem:spherical}
Let $A$ be a Frobenius algebra. Then the following are equivalent:
\begin{enumerate}[(i)]\item Equation \eqref{eq:spherical} \label{item:equation}
\item $\sigma^2=\id$ \label{item:sigma}
\item $B^{-1}$ decomposes into a direct sum of a symmetric bilinear form and an antisymmetric bilinear form.\label{item:B}
\end{enumerate}
\end{lemma}
\begin{proof}
Note that equation \eqref{eq:spherical} can be rewritten as $(B^{-1}B^{\tr})^2=\id$, using matrix notation. The definition of $\sigma$ then implies that $\varepsilon(e_a \cdot e_b)=\varepsilon(\sigma(e_b)\cdot e_a)$ or, equivalently, $B_{ab}=\sigma_b{}^c B_{ca}$. By contracting both sides with $B^{ad}$ one can conclude that $\sigma_{b}{}^d=B_{ab}B^{ad}$ or, as matrices, $\sigma=B^{-1}B^{\tr}$. The equivalence between \eqref{item:equation} and \eqref{item:sigma} is then immediate.

Suppose $B^{-1}$ is as in \eqref{item:B}. Then the vectors $v$ that lie in the symmetric or antisymmetric subspaces satisfy $B^{-1}v=\pm (B^{-1})^{\tr}v$. If $B^{\tr}$ is applied to this equation the identity $B^{\tr}B^{-1}v=\pm v$ is obtained, which is equivalent to  $(B^{\tr}B^{-1})^2=\id$, which implies  \eqref{item:equation}. On the other hand, if \eqref{eq:spherical} is satisfied then $(B^{\tr}B^{-1})^2=\id$. The eigenspaces with eigenvalues $\pm1$ give the direct sum decomposition of  \eqref{item:B}.
%
%
\end{proof}
 
For the case of triangulations of $M=S^2$ (with no boundary) the condition \eqref{eq:spherical} is not required. In these cases, $\begin{aligned}
\begingroup%
  \makeatletter%
  \providecommand\color[2][]{%
    \errmessage{(Inkscape) Color is used for the text in Inkscape, but the package 'color.sty' is not loaded}%
    \renewcommand\color[2][]{}%
  }%
  \providecommand\transparent[1]{%
    \errmessage{(Inkscape) Transparency is used (non-zero) for the text in Inkscape, but the package 'transparent.sty' is not loaded}%
    \renewcommand\transparent[1]{}%
  }%
  \providecommand\rotatebox[2]{#2}%
  \ifx\svgwidth\undefined%
    \setlength{\unitlength}{8.40000381bp}%
    \ifx\svgscale\undefined%
      \relax%
    \else%
      \setlength{\unitlength}{\unitlength * \real{\svgscale}}%
    \fi%
  \else%
    \setlength{\unitlength}{\svgwidth}%
  \fi%
  \global\let\svgwidth\undefined%
  \global\let\svgscale\undefined%
  \makeatother%
  \begin{picture}(1,1.61712173)%
    \put(0,0){\includegraphics[width=\unitlength]{blob.pdf}}%
  \end{picture}%
\endgroup%
\end{aligned}$ in  \eqref{eq:blob} is proportional to the identity matrix and so equation  \eqref{eq:blob} holds for any special Frobenius algebra. For the rest of this section and in \S\ref{sec:crossing}, only surfaces without boundary are considered and so the spherical condition is not needed. However the status of the spherical condition is addressed in a more general framework in \S\ref{sec:cat}.

\begin{definition} A state sum model for a triangulation of $S^2$ is said to be spherical if it is determined by the data of a planar state sum model.
\end{definition}

The partition function of a sphere can be calculated from any triangulation. The result
\begin{align} \label{eq:diag_sphere_inv}
Z(S^2)=R\,\varepsilon(1)=\begin{cases}R\Tr(x) &(k=\Cb)\\R\Real\Tr(x)&(k=\Rb)\end{cases}
\end{align}
follows from the classification given by theorem~\ref{theo:diagram_semi_simple}.  For $k=\Cb$, this result can also be written as $Z(S^2)=N\Tr(x)/\Tr(x^{-1})$.

\section{Models with crossings} \label{sec:crossing}

The diagrammatic method is extended to surfaces by the use of an immersion of the surface into $\Rb^3$. The dual of a triangulation of an oriented surface $\Sigma$ is a graph on the surface, which can be considered as a ribbon graph by taking the ribbon to be a suitable neighbourhood of the graph (called a regular neighbourhood \cite{hirsch-srn}) in the surface. This ribbon graph is therefore immersed in $\Rb^3$. The state sum model partition function is evaluated by taking a suitable invariant of this ribbon graph under the equivalence relation of regular homotopy. 

These concepts will be described in the case of smooth surfaces and immersions, for which there is a well-developed literature. As is standard in knot theory, the graphs can be described by the diagrams that result from a projection of $\Rb^3$ to $\Rb^2$ and the equivalence is a set of Reidemeister-like moves on diagrams. Then it is noted that the diagrams and their moves in fact also make sense as piecewise-linear diagrams, which is  more natural for triangulations. We leave it as a challenge to the reader to develop the theory using the piecewise-linear formulation of regular homotopy \cite{HaefligerPoenaru} from the beginning. 

A smooth immersion is a map $\phi\colon M\to N$ having a derivative that is injective at every point. Thus an immersion is locally an embedding. A regular homotopy from $\phi_0$ to $\phi_1$ is a family of immersions $\phi_t$, $t\in[0,1]$, that defines a smooth map $H(x,t)=\phi_t(x)\colon M\times[0,1]\to N$.  

Surfaces and curves immersed in $\Rb^3$ are studied in \cite{pinkall}, from which several key results are used. Let $\phi\colon\Sigma\to\Rb^3$ be a surface immersion and $G\subset\Sigma$ the graph dual to a triangulation of $\Sigma$. Then $\gamma=\phi_{|G}\colon G\to\Rb^3$ is an immersion of the graph $G$ and in the generic case this is an embedding, which means that there is an arbitrarily small regular homotopy to an embedding.  If there is a regular homotopy $\gamma_t$ between two embedded graphs $\gamma_0$ and $\gamma_1$, then the regular homotopy can be adjusted so that $\gamma_t$ is an embedding except at a finite set of values of $t$, where there is one intersection point. As $t$ varies through one of these values, one segment of an edge of the graph passes through another (see figure \ref{fig:regular-homotopy}).

\begin{figure}[t!]
                	\centering		
		\includegraphics{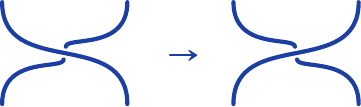}
		\caption{\emph{Regular homotopy.} Immersions of graphs in $\Rb^3$ allow for intersections. Regular homotopy thus allows a diagram under-crossing to be transfomed into an over-crossing.}
		\label{fig:regular-homotopy}
\end{figure} 

The graph $\gamma$ is described by a diagram obtained by projecting $\Rb^3$ to $\Rb^2$. It is assumed that this projection is generic, so that the graph is immersed in $\Rb^2$ with transverse self-intersections of edges. Since regular homotopy allows the edges to pass through each other, there is no need to record whether the crossings are over- or under-crossings. Diagrams are thus obtained from the usual diagrams of knot theory by setting over- and under-crossings equal, as is done in the theory of virtual knots \cite{KauffmanManturov}.

The graph $\gamma$ has a ribbon structure obtained by taking a suitably small regular neighbourhood $K$ of $\gamma$ in $\Sigma$, thus $\gamma\subset K\subset \Sigma$.
The formalism is simplified if the projection to $\Rb^2$ preserves the ribbon structure of the graph. As is standard in knot theory \cite{kauffman-regular-isotopy}, an embedded ribbon graph can be adjusted by a regular homotopy so that the projection of the ribbon to $\Rb^2$ is an orientation-preserving immersion. This is called `blackboard framing'. Then using blackboard-framed knots throughout, it is not necessary to include the ribbon in the planar diagrams. 

The state sum model is defined from the diagram in the plane by augmenting the formalism for a spherical state sum with a crossing map $\lambda \colon A \otimes A \to A \otimes A$ where one edge of the graph crosses another as shown.  
$$	
\begingroup%
  \makeatletter%
  \providecommand\color[2][]{%
    \errmessage{(Inkscape) Color is used for the text in Inkscape, but the package 'color.sty' is not loaded}%
    \renewcommand\color[2][]{}%
  }%
  \providecommand\transparent[1]{%
    \errmessage{(Inkscape) Transparency is used (non-zero) for the text in Inkscape, but the package 'transparent.sty' is not loaded}%
    \renewcommand\transparent[1]{}%
  }%
  \providecommand\rotatebox[2]{#2}%
  \ifx\svgwidth\undefined%
    \setlength{\unitlength}{106.83815918bp}%
    \ifx\svgscale\undefined%
      \relax%
    \else%
      \setlength{\unitlength}{\unitlength * \real{\svgscale}}%
    \fi%
  \else%
    \setlength{\unitlength}{\svgwidth}%
  \fi%
  \global\let\svgwidth\undefined%
  \global\let\svgscale\undefined%
  \makeatother%
  \begin{picture}(1,0.53023692)%
    \put(0,0){\includegraphics[width=\unitlength]{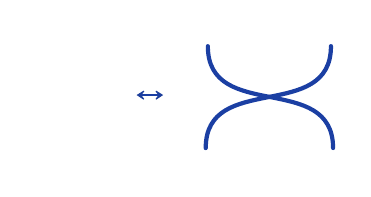}}%
    \put(-0.00266907,0.25241022){\color[rgb]{0,0,0}\makebox(0,0)[lb]{\smash{$\lambda_{ab}{}^{cd}$}}}%
    \put(0.53352247,0.00771467){\color[rgb]{0,0,0}\makebox(0,0)[lb]{\smash{$a$}}}%
    \put(0.87850342,0.00905148){\color[rgb]{0,0,0}\makebox(0,0)[lb]{\smash{$b$}}}%
    \put(0.53753359,0.47170068){\color[rgb]{0,0,0}\makebox(0,0)[lb]{\smash{$c$}}}%
    \put(0.86914347,0.47036364){\color[rgb]{0,0,0}\makebox(0,0)[lb]{\smash{$d$}}}%
  \end{picture}%
\endgroup%

$$ 
The partition function is calculated using the  analogue of the formula \eqref{eq:diag-partition} for the planar state sum models, with $|\gamma|$ the invariant of the ribbon graph described above,
\begin{equation}Z(M)=R^V |\gamma|.\label{eq:diag-spin}\end{equation}

An example of a planar diagram for the torus triangulated using two triangles is shown in figure \ref{fig:torus-diagrams}. The middle diagram shows a projection of the graph that is not blackboard-framed but the final diagram is the result of applying a regular homotopy so that the graph is blackboard-framed.

\begin{figure}[t!]
                \centering		
		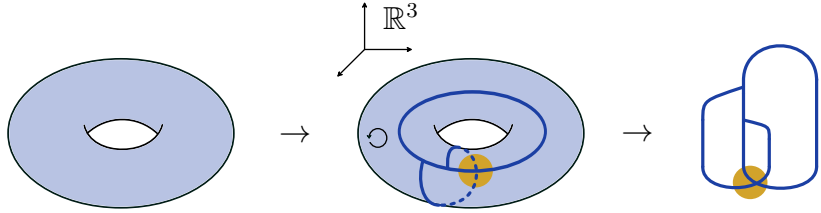
		\caption{\emph{Torus immersion.} A diagrammatic state sum model for the torus created from a triangulation with two triangles.}\label{fig:torus-diagrams}
\end{figure} 

The ribbon structure is preserved under the equivalence relation of regular homotopy. The usual Reidemeister moves for knots do not preserve the ribbon structure, so one has to use a modified set of moves for ribbon knots, described in \cite{kauffman-regular-isotopy,freyd-yetter}. The moves for graphs are described in \cite{Yetter,Kauffman-handbook} and the extension from ribbon knots to ribbon graphs is described in \cite{Reshetikhin-Turaev}.

A diagrammatic state sum model that is invariant under these moves is called a spin state sum model -- the most general state sum model with crossings considered in this paper. A diagram with $n$ downward- and $m$ upward-pointing legs defines a map $\otimes^n A \to \otimes^m A$, with the convention that $\otimes^0 A= k$. Therefore, diagrams should be read bottom-to-top and the use of explicit indices has been dropped.

\begin{definition} \label{def:spin-model} 
A spin state sum model $(C,B,R,\lambda)$ is a state sum model with the data $(C,B,R)$ of a planar model, together with a crossing map $\lambda$. The additional axioms the map $\lambda$ obeys are the
 \begin{enumerate}
\setcounter{enumi}{0}
\item \label{fig:axiom_form} compatibility with $B$, \hskip3.35cm 
		$\vcenter{\hbox{
\begingroup%
  \makeatletter%
  \providecommand\color[2][]{%
    \errmessage{(Inkscape) Color is used for the text in Inkscape, but the package 'color.sty' is not loaded}%
    \renewcommand\color[2][]{}%
  }%
  \providecommand\transparent[1]{%
    \errmessage{(Inkscape) Transparency is used (non-zero) for the text in Inkscape, but the package 'transparent.sty' is not loaded}%
    \renewcommand\transparent[1]{}%
  }%
  \providecommand\rotatebox[2]{#2}%
  \ifx\svgwidth\undefined%
    \setlength{\unitlength}{120.86936035bp}%
    \ifx\svgscale\undefined%
      \relax%
    \else%
      \setlength{\unitlength}{\unitlength * \real{\svgscale}}%
    \fi%
  \else%
    \setlength{\unitlength}{\svgwidth}%
  \fi%
  \global\let\svgwidth\undefined%
  \global\let\svgscale\undefined%
  \makeatother%
  \begin{picture}(1,0.21629127)%
    \put(0,0){\includegraphics[width=\unitlength]{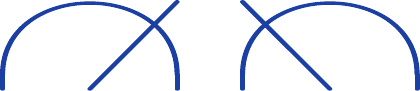}}%
    \put(0.46480483,0.06210709){\color[rgb]{0,0,0}\makebox(0,0)[lb]{\smash{$=$}}}%
  \end{picture}%
\endgroup%
}}$
\item compatibility with $C$, \hskip3.15cm \label{fig:axiom_mult} 
		$\vcenter{\hbox{ 
\begingroup%
  \makeatletter%
  \providecommand\color[2][]{%
    \errmessage{(Inkscape) Color is used for the text in Inkscape, but the package 'color.sty' is not loaded}%
    \renewcommand\color[2][]{}%
  }%
  \providecommand\transparent[1]{%
    \errmessage{(Inkscape) Transparency is used (non-zero) for the text in Inkscape, but the package 'transparent.sty' is not loaded}%
    \renewcommand\transparent[1]{}%
  }%
  \providecommand\rotatebox[2]{#2}%
  \ifx\svgwidth\undefined%
    \setlength{\unitlength}{124.33878174bp}%
    \ifx\svgscale\undefined%
      \relax%
    \else%
      \setlength{\unitlength}{\unitlength * \real{\svgscale}}%
    \fi%
  \else%
    \setlength{\unitlength}{\svgwidth}%
  \fi%
  \global\let\svgwidth\undefined%
  \global\let\svgscale\undefined%
  \makeatother%
  \begin{picture}(1,0.3232729)%
    \put(0,0){\includegraphics[width=\unitlength]{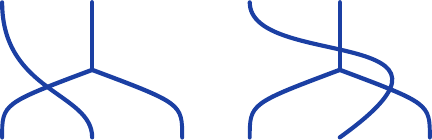}}%
    \put(0.45951805,0.05690025){\color[rgb]{0,0,0}\makebox(0,0)[lb]{\smash{$=$}}}%
  \end{picture}%
\endgroup%
}}$
\item Reidemeister II move (RII),\hskip2.35cm  \label{fig:axiom_square} 
 		$\vcenter{\hbox{ 
\begingroup%
  \makeatletter%
  \providecommand\color[2][]{%
    \errmessage{(Inkscape) Color is used for the text in Inkscape, but the package 'color.sty' is not loaded}%
    \renewcommand\color[2][]{}%
  }%
  \providecommand\transparent[1]{%
    \errmessage{(Inkscape) Transparency is used (non-zero) for the text in Inkscape, but the package 'transparent.sty' is not loaded}%
    \renewcommand\transparent[1]{}%
  }%
  \providecommand\rotatebox[2]{#2}%
  \ifx\svgwidth\undefined%
    \setlength{\unitlength}{75.75253296bp}%
    \ifx\svgscale\undefined%
      \relax%
    \else%
      \setlength{\unitlength}{\unitlength * \real{\svgscale}}%
    \fi%
  \else%
    \setlength{\unitlength}{\svgwidth}%
  \fi%
  \global\let\svgwidth\undefined%
  \global\let\svgscale\undefined%
  \makeatother%
  \begin{picture}(1,0.55428761)%
    \put(0,0){\includegraphics[width=\unitlength]{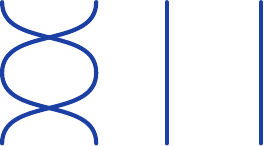}}%
    \put(0.43014569,0.2438102){\color[rgb]{0,0,0}\makebox(0,0)[lb]{\smash{$=$
}}}%
  \end{picture}%
\endgroup%
}}$
\item Reidemeister III move (RIII),\hskip1.2cm  \label{fig:axiom_Reid_3}
 		$\vcenter{\hbox{ 
\begingroup%
  \makeatletter%
  \providecommand\color[2][]{%
    \errmessage{(Inkscape) Color is used for the text in Inkscape, but the package 'color.sty' is not loaded}%
    \renewcommand\color[2][]{}%
  }%
  \providecommand\transparent[1]{%
    \errmessage{(Inkscape) Transparency is used (non-zero) for the text in Inkscape, but the package 'transparent.sty' is not loaded}%
    \renewcommand\transparent[1]{}%
  }%
  \providecommand\rotatebox[2]{#2}%
  \ifx\svgwidth\undefined%
    \setlength{\unitlength}{124.40087891bp}%
    \ifx\svgscale\undefined%
      \relax%
    \else%
      \setlength{\unitlength}{\unitlength * \real{\svgscale}}%
    \fi%
  \else%
    \setlength{\unitlength}{\svgwidth}%
  \fi%
  \global\let\svgwidth\undefined%
  \global\let\svgscale\undefined%
  \makeatother%
  \begin{picture}(1,0.32278936)%
    \put(0,0){\includegraphics[width=\unitlength]{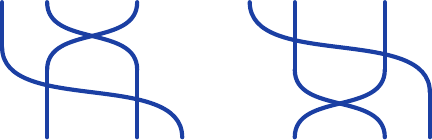}}%
    \put(0.44945526,0.14184983){\color[rgb]{0,0,0}\makebox(0,0)[lb]{\smash{$=$}}}%
  \end{picture}%
\endgroup%
}}$
\item  \label{fig:axiom_left_right} ribbon condition,\hskip4.2cm  
		$\vcenter{\hbox{ 
\begingroup%
  \makeatletter%
  \providecommand\color[2][]{%
    \errmessage{(Inkscape) Color is used for the text in Inkscape, but the package 'color.sty' is not loaded}%
    \renewcommand\color[2][]{}%
  }%
  \providecommand\transparent[1]{%
    \errmessage{(Inkscape) Transparency is used (non-zero) for the text in Inkscape, but the package 'transparent.sty' is not loaded}%
    \renewcommand\transparent[1]{}%
  }%
  \providecommand\rotatebox[2]{#2}%
  \ifx\svgwidth\undefined%
    \setlength{\unitlength}{75.31722412bp}%
    \ifx\svgscale\undefined%
      \relax%
    \else%
      \setlength{\unitlength}{\unitlength * \real{\svgscale}}%
    \fi%
  \else%
    \setlength{\unitlength}{\svgwidth}%
  \fi%
  \global\let\svgwidth\undefined%
  \global\let\svgscale\undefined%
  \makeatother%
  \begin{picture}(1,0.59036449)%
    \put(0,0){\includegraphics[width=\unitlength]{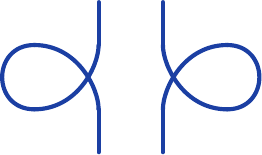}}%
    \put(0.44899274,0.26275822){\color[rgb]{0,0,0}\makebox(0,0)[lb]{\smash{$=$}}}%
  \end{picture}%
\endgroup%
}} \hspace{3mm}$.
\end{enumerate}
\end{definition}
Either side of axiom \ref{fig:axiom_left_right} defines a map, $\varphi \colon A \to A$, and the axioms \ref{fig:axiom_form}, \ref{fig:axiom_square} and \ref{fig:axiom_Reid_3}
 imply, via the Whitney trick \cite{kauffman-regular-isotopy}, that $\varphi^2=\id$. Either diagram in axiom  \ref{fig:axiom_left_right} is called a curl.

There are two issues to settle: the possible dependence of the state sum model on the triangulation of the surface, and on the immersion $\phi$. The former is the easiest to resolve: since any planar state sum model is invariant under Pachner moves the following lemma is automatically verified.

\begin{lemma}\label{lem:spin-triangulation} The partition function of a spin state sum model is independent of the triangulation of the surface.
\end{lemma}

An interesting class of examples arises for $G$-graded algebras $A=\bigoplus_{h \in G} A_h$ where $G$ is an abelian group. Crossing maps can then be constructed from bicharacters \cite{Bahturin}. A bicharacter $\tilde{\lambda} \colon G \times G \to k$ is defined by 
\begin{align} 
\tilde{\lambda}(h,jl)&=\tilde{\lambda}(h,j)\tilde{\lambda}(h,l), \label{eq:graded_mult}\\
1&=\tilde{\lambda}(h,j)\tilde{\lambda}(j,h). \label{eq:graded_inverse}
\end{align}   
The candidate for a crossing map $\lambda$ is then determined by setting 
\begin{equation}\lambda(a_h \otimes b_j)=\tilde{\lambda}(h,j)\; b_j \otimes a_h \in A_j \otimes A_h.\end{equation}
With this definition, it is straightforward to conclude properties \eqref{eq:graded_mult} and \eqref{eq:graded_inverse} of a bicharacter $\tilde{\lambda}$ are in correspondence with the crossing axioms \ref{fig:axiom_mult} and \ref{fig:axiom_square} of definition \ref{def:spin-model}. On the other hand, axiom \ref{fig:axiom_Reid_3} is automatically verified since $\tilde{\lambda}$ is $k$-valued. The remaining conditions, however, impose new constraints on a bicharacter. Write $A_h\perp A_j$ if $\varepsilon(a_h \cdot b_j)=0$ for all $a_h\in A_h$, $b_j\in A_j$.
\begin{lemma} \label{lem:graded} A graded Frobenius algebra with a bicharacter  $\tilde\lambda$ determines a spin state sum model if and only if
\begin{enumerate}
\item\label{it:cr-axiom1} For each $h,j\in G$, either  $A_h\perp A_j$ or $\tilde{\lambda}(h,l)=\tilde{\lambda}(l,j)$  for all $l \in G$.
\item\label{it:cr-axiom2} The Nakayama automorphism $\sigma$ obeys $\sigma^2=\id$.
\end{enumerate}
\end{lemma}
\begin{proof} 
Applying the maps in axiom \ref{fig:axiom_form} of definition \ref{def:spin-model} to $a_h\otimes c_l\otimes b_j$ gives
\begin{equation}\tilde\lambda(l,j)\,\varepsilon(a_h\cdot b_j)\,c_l=\tilde\lambda(h,l)\,\varepsilon(a_h\cdot b_j)\,c_l,\end{equation}
 which is equivalent to condition \ref{it:cr-axiom1}.

The element $B$ can be written as a sum of linearly independent terms as $B=\sum y_m\otimes z_n$, in which the gradings $m$ and $n$ may vary. An equivalent relation to condition \ref{it:cr-axiom1} is that for each term $y_m\otimes z_n$ in the sum,
\begin{equation}\tilde\lambda(n,l)=\tilde\lambda(l,m) \text{ for all } l\in G.\label{eq:bichar-identity}\end{equation}
This can be proved by using an equivalent form of axiom \ref{fig:axiom_form} of definition \ref{def:spin-model} given by rotating both diagrams in the expression by $\pi$. Then applying the maps on both sides of the equation to $a_l$ gives the identity
\begin{equation}\sum\tilde\lambda(l,m)\, y_m\otimes a_l\otimes z_n=\sum\tilde\lambda(n,l)\, y_m\otimes a_l\otimes z_n.\end{equation}

The curl on the right-hand side of axiom \ref{fig:axiom_left_right} is the map 
\begin{equation}\label{eq:bichar-curl} a_l\mapsto\sum \varepsilon(a_l \cdot z_n)\tilde\lambda(l,m)\,y_m.
\end{equation}
However, from \eqref{eq:bichar-identity}, $\tilde\lambda(l,m)=\tilde\lambda(n,l)$ and for the non-zero terms in \eqref{eq:bichar-curl}, $\tilde\lambda(l,l)=\tilde\lambda(l,n)$. Together these imply $\tilde\lambda(l,m)=\tilde\lambda(l,l)$. Hence the curl is
\begin{equation}\varphi(a_l)=\tilde\lambda(l,l)\sum\varepsilon(z_n \cdot \sigma^{-1}(a_l))\,y_m=\tilde\lambda(l,l)\sigma^{-1}(a_l). \label{eq:graded_group_Naka}
\end{equation}
Since axiom \ref{fig:axiom_left_right} is equivalent to $\varphi^2=\id$, and \eqref{eq:graded_inverse} implies $\tilde\lambda(l,l)^2=1$,  the axioms of definition \ref{def:spin-model} imply that $\sigma^2=\text{id}$. Conversely, $\sigma^2=\text{id}$ together with axioms \ref{fig:axiom_form} to \ref{fig:axiom_Reid_3} imply that axiom \ref{fig:axiom_left_right} is satisfied.
\end{proof}

The spin state sum models are analysed fully in \S \ref{sec:spinssm}. 

\subsection{Curl-free models}\label{sec:curlfreemodels}

This section discusses a particular class of spin state sum models for which the data satisfy one additional axiom, $\varphi=\id$. These models are called curl-free.
 Diagrammatically, this is the
\begin{enumerate}
\setcounter{enumi}{5}
\item  \label{ax:RI} Reidemeister I move (RI),\hskip1cm
		$\vcenter{\hbox{
\begingroup%
  \makeatletter%
  \providecommand\color[2][]{%
    \errmessage{(Inkscape) Color is used for the text in Inkscape, but the package 'color.sty' is not loaded}%
    \renewcommand\color[2][]{}%
  }%
  \providecommand\transparent[1]{%
    \errmessage{(Inkscape) Transparency is used (non-zero) for the text in Inkscape, but the package 'transparent.sty' is not loaded}%
    \renewcommand\transparent[1]{}%
  }%
  \providecommand\rotatebox[2]{#2}%
  \ifx\svgwidth\undefined%
    \setlength{\unitlength}{50.54191895bp}%
    \ifx\svgscale\undefined%
      \relax%
    \else%
      \setlength{\unitlength}{\unitlength * \real{\svgscale}}%
    \fi%
  \else%
    \setlength{\unitlength}{\svgwidth}%
  \fi%
  \global\let\svgwidth\undefined%
  \global\let\svgscale\undefined%
  \makeatother%
  \begin{picture}(1,0.93768129)%
    \put(0,0){\includegraphics[width=\unitlength]{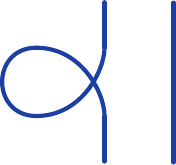}}%
    \put(0.69298967,0.40858098){\color[rgb]{0,0,0}\makebox(0,0)[lb]{\smash{$=$}}}%
  \end{picture}%
\endgroup%
}}\hspace{3mm}.$
\end{enumerate}

The main issue is the dependence of the partition function on the immersion. Consider a standard immersion $\phi_0$ that is an embedding of the closed oriented surface of genus $g$ into $\Rb^3$. A triangulation of the surface $\Sigma$ can be constructed by identifying the edges of a $4g$-sided polygon, as in figure \ref{fig:4gpoly}, and dividing it into triangles without introducing any new vertices. Let $S\subset\Sigma$ be the subset obtained by removing a disk neighbourhood of the vertex of the polygon from $\Sigma$.  The embedding is such that $S$ projects to $\Rb^2$ by the immersion shown in figure \ref{fig:surface-diagram}. The dual graph to the triangulation is shown in the figure \ref{fig:dual-graph} with all of the graph vertices consolidated into one. 
\begin{figure}
\centering
\begin{subfigure}[t!]{0.47\textwidth}
		\centering
		\includegraphics{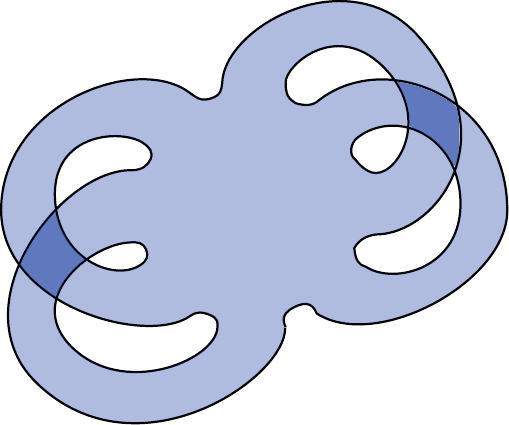}
		\vspace{4mm}
		\caption{\emph{Standard projection.} The projection of $S$ into $\Rb^2$ for the standard embedding of $\Sigma_2$. Dark-shaded regions represent areas of intersection.}
                \label{fig:surface-diagram}
\end{subfigure} 
\hspace{5mm}
\begin{subfigure}[t!]{0.47\textwidth}
		\centering
		\vspace{9mm}
		\includegraphics{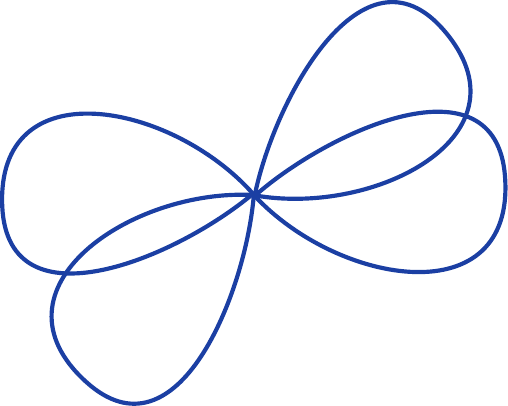}
		\vspace{4mm}
		\caption{\emph{Standard diagram} $\gamma_g$. The dual graph of $\Sigma_2$ resulting from the standard embedding in $\Rb^3$, denoted $\gamma_2$. Thickening to a ribbon graph results in figure \ref{fig:surface-diagram}. The standard diagram for a surface $\Sigma_g$ is denoted $\gamma_g$. }
                \label{fig:dual-graph}
\end{subfigure} 
\caption{}
\end{figure}
\begin{lemma}\label{lem:embedding} The partition function of a curl-free state sum model for a closed surface is independent of the immersion $\phi$.
\end{lemma}
\begin{proof} Consider  a ribbon graph $K\subset\Sigma$ and an immersion $\phi\colon \Sigma\to \Rb^3$. The immersion of the ribbon graph is moved by regular homotopy to $\psi\colon K\to\Rb^3$ that is blackboard-framed with respect to the projection $P\colon\Rb^3\to\Rb^2$, $P(x,y,z)=(x,y)$. Further, a neighbourhood of the consolidated vertex in the diagram can be moved to match a neighbourhood of the vertex in figure \ref{fig:dual-graph}.  Then each ribbon loop of $K$ can be moved independently, keeping the neighbourhood of the vertex fixed. 

According to the  Whitney-Graustein theorem \cite{kauffman-regular-isotopy}, immersed circles in $\Rb^2$ under regular homotopy (in $\Rb^2$) are classified by the Whitney degree, which is the integer that measures the number of windings of the tangent vector to the circle. This regular homotopy extends to a regular homotopy of the ribbon graph in $\Rb^2$. Then it lifts to a regular homotopy of the ribbon graph $\psi$ in $\Rb^3$, by keeping the $z$-coordinate constant in the homotopy. Therefore each loop of $K$ is regular-homotopic to the corresponding loop of figure \ref{fig:dual-graph}, but with a number of curls. The curls can be cancelled using the ribbon condition and the move RI. Thus the partition function is the same as for $\phi_0$. 
\end{proof}

These results imply the partition function of a curl-free model is indeed a topological invariant. Let $f\colon\Sigma'\to\Sigma$ be a diffeomorphism. If $\Sigma$ is a triangulated surface and $\phi\colon\Sigma\to\Rb^3$ is an immersion, then $f$ induces a triangulation and an immersion for $\Sigma'$ such that their dual graph diagrams in the plane coincide.

Some examples of curl-free models are studied in the rest of this section. First it is shown how the naive state sum models of \S\ref{sec:lattice_tft} fit within the new formalism.
\begin{example}\label{ex:FHK}
An FHK state sum model as defined in \S\ref{sec:lattice_tft} is a curl-free state sum model where the choice of crossing is canonical. In other words, the map $\lambda \colon A \otimes A \to A \otimes A$ takes $a \otimes b \mapsto b \otimes a$. 
\end{example}

Next, examples determined by a bicharacter are studied.
Axiom \ref{ax:RI} is $\varphi(a_l)=a_l$; one can therefore conclude that  $\sigma$ preserves the $G$-grading and, according to \eqref{eq:graded_group_Naka}, obeys the eigenvector equation $\sigma(a_l)=\tilde{\lambda}(l,l)\, a_l$. 

Explicit examples of matrix algebras that can be equipped with this type of crossing are now presented.
\begin{example}[Algebras $A=\Mb_n(k)$, $k=\Rb,\Cb$] \label{ex:non_symmetric_Frobenius}
\noindent Let $\varepsilon(a)=R(p-q)\Tr(ua)$ with $u=\text{diag}(p,q)$ the diagonal matrix with the first $p>0$ diagonal entries equal to $+1$ and the remaining $q=n-p>0$ entries equal to $-1$, such that $p\neq q$. The algebra $A$ has a natural $\Zb_2$-grading $A_0 \bigoplus A_1$. Each matrix splits into a block-diagonal and a block-anti-diagonal part.
\begin{align}
\left(\begin{array}{cc}
a_{p\times p} & b_{p \times q}\\ c_{q \times p} & d_{q \times q}
\end{array}\right)
=
\left(\begin{array}{cc}
a_{p\times p} & 0\\ 0 & d_{q \times q}
\end{array}\right)
\oplus
\left(\begin{array}{cc}
0 & b_{p \times q}\\ c_{q \times p} & 0
\end{array}\right)
\in A_0 \bigoplus A_1.
\end{align}
\noindent It is easy to verify there is a unique $\Zb_2$-bicharacter $\tilde{\lambda}$ that can be constructed for this algebra for which $\lambda$ is a curl-free crossing. Identity \eqref{eq:graded_mult} implies $\tilde{\lambda}(0,h)=\tilde{\lambda}(h,0)=1$. Identity \eqref{eq:graded_inverse} implies $\tilde{\lambda}(1,1)=\pm1$ but the choice $\tilde{\lambda}(1,1)=1$ is not allowed as $\varphi=\sigma \neq \text{id}$ would follow. The components of $\lambda$ are therefore determined by the relation $\tilde{\lambda}(h,j)=(-1)^{hj}$ with $h,j =0,1$. 
The bilinear form can be written as
\begin{align}
B = \frac{1}{R(p-q)} \sum_{lm,h}e_{lm}^{h}u \otimes e_{ml}^{h}
\end{align}
where the label $h$ identifies whether the elementary matrices belong to $A_0$ or $A_1$. The standard diagram $\gamma_g$ to be associated with a closed surface of genus $g$ (see figure~\ref{fig:dual-graph}) can be used to write the partition function \eqref{eq:diag-spin} as
\begin{align}
Z(\Sigma_g)=R\varepsilon(\eta^g) \text{ with }\eta=\frac{1}{R^2(p-q)^2} \sum_{lmrs,hj}e_{lm}^{h}ue_{rs}^{j}ue_{ml}^{h}e_{sr}^{j}\, .
\end{align}
Notice that $\sigma(a)=uau$. The simplification $\sum_{lm,h}e_{lm}^{h}ue_{rs}^{j}ue_{ml}^{h}=\Tr(\sigma(e_{rs}^{j}))1$ follows. If $j=1$, $\Tr(\sigma(e_{rs}^{j}))$ vanishes since $\sigma$ preserves the grading and block-anti-diagonal matrices are traceless. If $j=0$ then $\Tr(\sigma(e_{rs}^{j}))=\delta_{rs}$ and consequently $z=R^{-2}(p-q)^{-2}1$. Therefore, the partition function reads
\begin{equation}Z(\Sigma_g)=R^{2-2g}(p-q)^{2-2g}.\end{equation}
This formula does not coincide with the partition function \eqref{eq:partitionfunction} determined for $\Mb_n(k)$ seen as an FHK state sum model,  but it does reduce to it by setting  $q=0$. (For $q=0$ the canonical crossing is the acceptable choice.)  
\end{example}

A natural question is whether algebras with symmetric Frobenius forms, and a crossing respecting the conditions of definition~\ref{def:spin-model} and the curl-free condition always give rise to an FHK state sum model. This is not, however, the case as it can be seen by the explicit example below.

\begin{example}[Algebras $A=\Mb_n(\Cb)$]
As studied in \cite{Bahturin}, $\Mb_n(\Cb)$ can be regarded as a $\Gamma_n$-graded algebra where the group is a direct product of two cyclic groups of order $n$: $\Gamma_n = \langle a\rangle \times \langle b\rangle$. This means $\Mb_{n}(\Cb)$ is decomposed into $n^2$ components and it is natural to pick as a basis $n^2$ matrices that respect this decomposition. Let $\xi \in \Cb$ be a  primitive $n$-th root of unity and define the matrices $X_a=\text{diag}(\xi^{n-1},\cdots,\xi,1)$ and $Y_b=e_{n1}+\sum_{m=1}^{n-1}e_{m(m+1)}$, as in \cite{Bahturin}. Then, $X_a^iY_b^j$ generates the $a^ib^j$ component of the algebra and the expression
\begin{align}
\tilde{\lambda}(a^ib^j,a^{i'}b^{j'})=\xi^{ij'-i'j}
\end{align}
defines a bicharacter \cite{Bahturin}. 

It must be verified that the remaining conditions of lemma \ref{lem:graded} and the Reidemeister I move hold. From \eqref{eq:graded_group_Naka}, $\sigma(X_a^iY_b^j)=\tilde{\lambda}(a^ib^j,a^{i}b^{j})X_a^iY_b^j$. Since $\tilde{\lambda}(a^ib^j,a^{i}b^{j})=1$ it follows that $\sigma= \text{id}$, or, equivalently, that the Frobenius form is symmetric. In other words, $\varepsilon(y)=Rn\Tr(y)$.

For condition \ref{it:cr-axiom1} of lemma \ref{lem:graded} one first shows that $\varepsilon (X_a^iY_b^jX_a^{i'}Y_b^{j'})=0$ unless $i+i'=j+j'=0$, a fact that follows from the identity $X_a^iY_b^jX_a^{i'}Y_b^{j'}=\xi^{-ji'}X_a^{i+i'}Y_b^{j+j'}$ and the symmetry of $\varepsilon$. When $i+i'=j+j'=0$ the required identity for the bicharacter reduces to $\tilde{\lambda}(h,j)=\tilde{\lambda}(j,h^{-1})$ for all $h,j$, which is always true.

Only for $n=1$ does $\lambda$ coincide with the canonical crossing. The invariant created is $Z(\Sigma_g)=R^{2-2g}n^2$, which differs from expression \eqref{eq:partitionfunction}. 
\end{example}
\subsection{Spin models}\label{sec:spinssm}

The purpose of this section is to study spin state sum models and show that these are defined on a surface with a spin structure. Our ultimate objective is to introduce a crossing that  distinguishes topologically-inequivalent spin structures and several examples of such algebras will be studied. 

The usual notion of spin structure is defined for oriented smooth manifolds using the tangent bundle. Each immersed curve $c$ on the manifold lifts to a curve in the frame bundle $F$ and the spin structure $s\in H^1(F,\Zb_2)$ assigns to this an element $s(c)\in\Zb_2$. This assignment can be characterised by a skein relation on a vector space generated by curves on the manifold \cite{JBspin}, a description that does not require the use of the tangent bundle (and so generalises to piecewise-linear manifolds).
 
On an oriented surface there is an even simpler description \cite{kirby} of a spin structure as a quadratic form on the  first homology with $\Zb_2$ coefficients, $q\colon H_1(\Sigma,\Zb_2)\to\Zb_2$ . The quadratic form $q$ is defined by taking an embedded curve $c$ to represent a cycle and setting $q(c)=s(c)+1 \mo 2$. The quadratic form satisfies the relation $q(x+y)=q(x)+q(y)+x.y$, with $x.y$ the intersection form for mod 2 homology, and so is determined by its values on a basis of $H_1(\Sigma,\Zb_2)$.

The immersions of a smooth surface into $\Rb^3$ are classified in \cite{pinkall}, where it is shown that there are $2^{2g}$ inequivalent regular homotopy equivalence classes. Each immersion $\phi\colon\Sigma\to\Rb^3$ determines an induced spin structure on $\Sigma$ by pulling-back the unique spin structure on $\Rb^3$. The induced spin structure is invariant under a regular homotopy (since the homotopy is differentiable). There are $2^{2g}$ spin structures on an oriented surface and these classify the equivalence classes of immersions uniquely. This can be seen by explicitly constructing an immersion that corresponds to each spin structure. A spin structure on $\Sigma$ is determined uniquely by a spin structure on the subset $S\subset\Sigma$ obtained by removing a disk. The surface $S$ can be embedded in $\Rb^3$ so that the projection to $\Rb^2$ is an immersion as in figure \ref{fig:surface-diagram}, or a modification of it by putting a curl in any of the $2g$ ribbon loops. The spin structure is read off from this diagram: $s(c)$ is the Whitney degree mod 2 for the projection of $c$ to $\Rb^2$.  For example, for the embedding $\phi_0$ each circle $c$ in figure \ref{fig:dual-graph} has no curls and so $q(c)=0$.  It is worth noting that this explicit construction of $q$ does not require a smooth structure and makes sense also for a piecewise-linear surface.

\begin{lemma} \label{spin-embedding} The partition function of a spin state sum model on $\Sigma$ depends on the immersion $\phi\colon \Sigma\to\Rb^3$  only via the spin structure induced on $\Sigma$.
\end{lemma}

\begin{proof}
The proof is similar to the proof of lemma  \ref{lem:embedding}, except that the curls can only be cancelled in pairs. Each curve in the graph can be moved to coincide with the curve from $\phi_0$ except that each curve contains a number of curls. These curls can be cancelled pairwise so that each curve has either one or zero curls; this is the data in the induced spin structure.\end{proof}

For example, one diagram for each of the four equivalence classes for the torus are shown in figure \ref{fig:torus-spin-diagrams}. The corresponding spin structures have $(q(c_1),q(c_2))=(0,0),(1,0),(0,1),(1,1)$ for the two embedded cycles $c_1$, $c_2$ forming a basis of $H_1(\Sigma_1,\Zb_2)$.

Lemmas \ref{lem:spin-triangulation} and \ref{spin-embedding} imply the partition function is an invariant of a surface with spin structure. Let $f\colon\Sigma'\to\Sigma$ be a diffeomorphism and $\phi\colon\Sigma\to\Rb^3$ an immersion inducing a spin structure $s$. Then the immersion $\phi\circ f$ induces the spin structure $f^*s$ on $\Sigma'$. Note that the invariance of the partition function can also be checked directly, without using the Pachner moves, by examining the effect of Dehn twists \cite{lickorish,lickorish-erratum} on the surface.

\begin{figure}[t!]
		\centering
		\includegraphics[width=\textwidth]{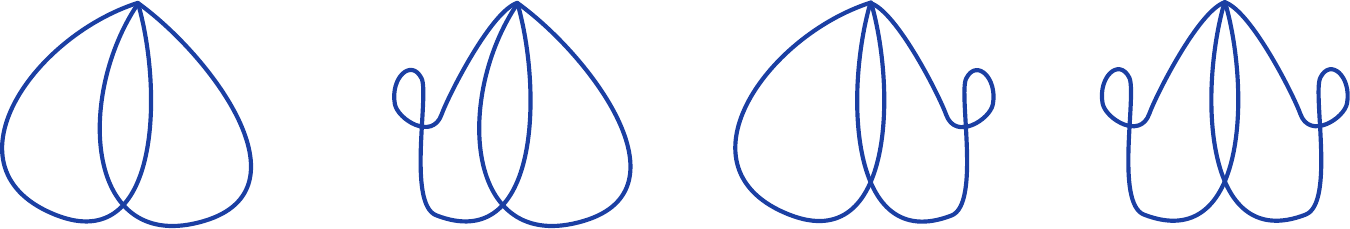}
		\caption{\emph{Torus spin models.} Dual graph diagrams for four immersions of the torus that are inequivalent under regular homotopy. The labels $c_1$ and $c_2$ are in correspondence with the left and right cycles in each diagram, respectively. The first three diagrams to the left are topologically-equivalent.}
                \label{fig:torus-spin-diagrams}
\end{figure} 

To calculate examples of spin models, an explicit formula is needed for the partition function that is manifestly an invariant. To establish this non-trivial result (theorem \ref{theo:main}), the algebraic consequences of the axioms for the spin models are studied. 

A straightforward first consequence is that $\varphi$ as defined in \ref{def:spin-model} is also an isomorphism of the algebra $A$ determined by the data $(C,B,R)$, which is to say, $\varphi(a\cdot b)=\varphi(a)\cdot\varphi(b)$ for all $a,b \in A$. The diagrammatic proof of this identity can be found below\footnote{The existence of the isomorphism $\varphi$ raises the question of uniqueness in the model. One could ask if compositions of $\varphi$ and $B$ or $\varphi$ and $C$ would give rise to alternative and valid spin state sum model data. It is, however, a simple exercise to verify that the original data is the only one that manifestly satisfies all the necessary axioms.}.
 $$		
\begingroup%
  \makeatletter%
  \providecommand\color[2][]{%
    \errmessage{(Inkscape) Color is used for the text in Inkscape, but the package 'color.sty' is not loaded}%
    \renewcommand\color[2][]{}%
  }%
  \providecommand\transparent[1]{%
    \errmessage{(Inkscape) Transparency is used (non-zero) for the text in Inkscape, but the package 'transparent.sty' is not loaded}%
    \renewcommand\transparent[1]{}%
  }%
  \providecommand\rotatebox[2]{#2}%
  \ifx\svgwidth\undefined%
    \setlength{\unitlength}{266.59709473bp}%
    \ifx\svgscale\undefined%
      \relax%
    \else%
      \setlength{\unitlength}{\unitlength * \real{\svgscale}}%
    \fi%
  \else%
    \setlength{\unitlength}{\svgwidth}%
  \fi%
  \global\let\svgwidth\undefined%
  \global\let\svgscale\undefined%
  \makeatother%
  \begin{picture}(1,0.14397813)%
    \put(0,0){\includegraphics[width=\unitlength]{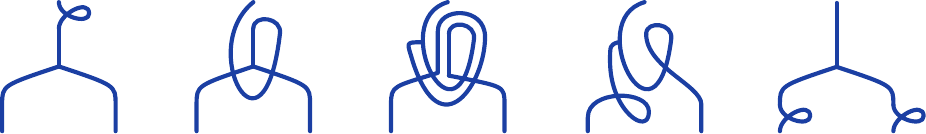}}%
    \put(0.14898576,0.0344589){\color[rgb]{0,0,0}\makebox(0,0)[lb]{\smash{$=$}}}%
    \put(0.78036525,0.03400505){\color[rgb]{0,0,0}\makebox(0,0)[lb]{\smash{$=$}}}%
    \put(0.35904061,0.0344589){\color[rgb]{0,0,0}\makebox(0,0)[lb]{\smash{$=$}}}%
    \put(0.57164399,0.0344589){\color[rgb]{0,0,0}\makebox(0,0)[lb]{\smash{$=$}}}%
  \end{picture}%
\endgroup%

\label{fig:phi_homomorphism}
$$ 
The next objective is to build the diagrammatic counterpart of expression \eqref{eq:invariant}, assigning $Z(\Sigma_g,s)$ to an orientable surface with spin structure. It is necessary to understand the analogue of the element $z=e_a\cdot e_b\cdot e_c \cdot e_d\hspace{1mm}\,B^{ac}\,B^{bd}$, introduced in equation~\eqref{eq:torus}, in the spin model. The difference is the possible introduction of curls in the diagrams.

A useful preliminary is the study of all the possible diagrams one can associate with the cylinder topology. These maps $A \to A$ are depicted below and denoted $p$, $n_1$ and $n_2$ respectively. 

$$
\centering
\begingroup%
  \makeatletter%
  \providecommand\color[2][]{%
    \errmessage{(Inkscape) Color is used for the text in Inkscape, but the package 'color.sty' is not loaded}%
    \renewcommand\color[2][]{}%
  }%
  \providecommand\transparent[1]{%
    \errmessage{(Inkscape) Transparency is used (non-zero) for the text in Inkscape, but the package 'transparent.sty' is not loaded}%
    \renewcommand\transparent[1]{}%
  }%
  \providecommand\rotatebox[2]{#2}%
  \ifx\svgwidth\undefined%
    \setlength{\unitlength}{302.92426758bp}%
    \ifx\svgscale\undefined%
      \relax%
    \else%
      \setlength{\unitlength}{\unitlength * \real{\svgscale}}%
    \fi%
  \else%
    \setlength{\unitlength}{\svgwidth}%
  \fi%
  \global\let\svgwidth\undefined%
  \global\let\svgscale\undefined%
  \makeatother%
  \begin{picture}(1,0.23684071)%
    \put(0,0){\includegraphics[width=\unitlength]{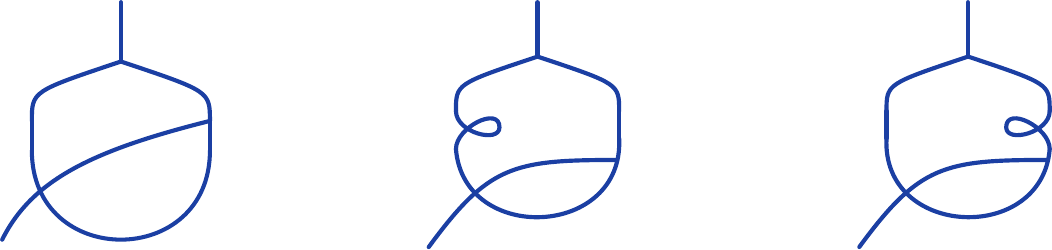}}%
  \end{picture}%
\endgroup%

$$ 
 Define two subspaces of $A$: $\mathcal{Z}_{\lambda}(A)$, the set of all elements $a \in A$ satisfying $m (b \otimes a)= m \circ \lambda (b \otimes a)$ for all $b \in A$, and analogously $\overline{\mathcal{Z}}_{\lambda}(A)$, the set of all elements $a \in A$ satisfying $m (b \otimes a)= m \circ \lambda (\varphi(b) \otimes a)$, for all $b \in A$.
\begin{lemma}
The map $R.p$ is a projector $A \to A$ with image $\mathcal{Z}_{\lambda}(A)$. Further, $p \circ \varphi = p$ and $n_2=n_1 \circ \varphi= \varphi \circ  n_1$. The map $R.n_1$ is a projector $A \to A$ with image $\overline{\mathcal{Z}}_{\lambda}(A)$.
\end{lemma} 
\begin{proof}
First, one must note that for all $a \in A$, $p(a) \in \mathcal{Z}_{\lambda}(A)$. 
$$
\centering
\begingroup%
  \makeatletter%
  \providecommand\color[2][]{%
    \errmessage{(Inkscape) Color is used for the text in Inkscape, but the package 'color.sty' is not loaded}%
    \renewcommand\color[2][]{}%
  }%
  \providecommand\transparent[1]{%
    \errmessage{(Inkscape) Transparency is used (non-zero) for the text in Inkscape, but the package 'transparent.sty' is not loaded}%
    \renewcommand\transparent[1]{}%
  }%
  \providecommand\rotatebox[2]{#2}%
  \ifx\svgwidth\undefined%
    \setlength{\unitlength}{291.42006836bp}%
    \ifx\svgscale\undefined%
      \relax%
    \else%
      \setlength{\unitlength}{\unitlength * \real{\svgscale}}%
    \fi%
  \else%
    \setlength{\unitlength}{\svgwidth}%
  \fi%
  \global\let\svgwidth\undefined%
  \global\let\svgscale\undefined%
  \makeatother%
  \begin{picture}(1,0.42491371)%
    \put(0,0){\includegraphics[width=\unitlength]{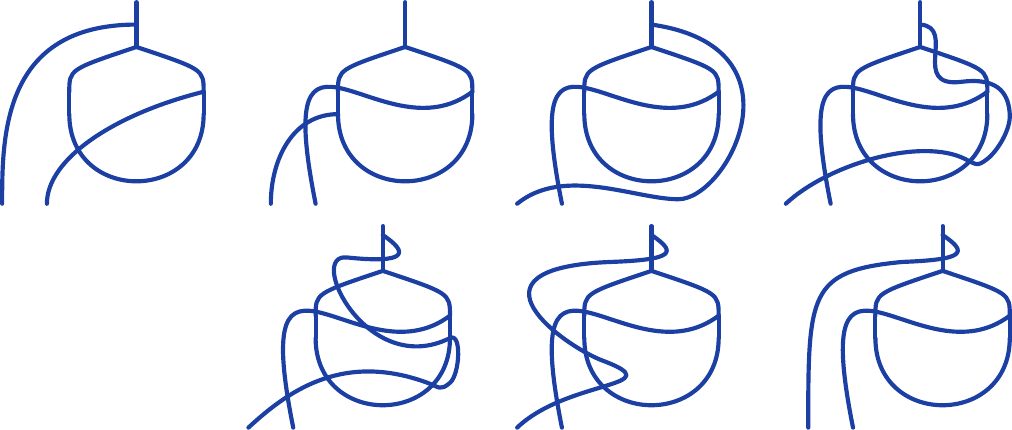}}%
    \put(0.22355523,0.2929243){\color[rgb]{0,0,0}\makebox(0,0)[lb]{\smash{$=$}}}%
    \put(0.48434715,0.2929243){\color[rgb]{0,0,0}\makebox(0,0)[lb]{\smash{$=$}}}%
    \put(0.22355523,0.08703594){\color[rgb]{0,0,0}\makebox(0,0)[lb]{\smash{$=$}}}%
    \put(0.48434715,0.08703594){\color[rgb]{0,0,0}\makebox(0,0)[lb]{\smash{$=$}}}%
    \put(0.74513907,0.2929243){\color[rgb]{0,0,0}\makebox(0,0)[lb]{\smash{$=$}}}%
    \put(0.74513907,0.07331005){\color[rgb]{0,0,0}\makebox(0,0)[lb]{\smash{$=$}}}%
  \end{picture}%
\endgroup%

$$

\noindent One can then further conclude that if $a \in \mathcal{Z}_{\lambda}(A)$ then $R.p(a)=a$. 
$$       
\centering       
\begingroup%
  \makeatletter%
  \providecommand\color[2][]{%
    \errmessage{(Inkscape) Color is used for the text in Inkscape, but the package 'color.sty' is not loaded}%
    \renewcommand\color[2][]{}%
  }%
  \providecommand\transparent[1]{%
    \errmessage{(Inkscape) Transparency is used (non-zero) for the text in Inkscape, but the package 'transparent.sty' is not loaded}%
    \renewcommand\transparent[1]{}%
  }%
  \providecommand\rotatebox[2]{#2}%
  \ifx\svgwidth\undefined%
    \setlength{\unitlength}{304.98046875bp}%
    \ifx\svgscale\undefined%
      \relax%
    \else%
      \setlength{\unitlength}{\unitlength * \real{\svgscale}}%
    \fi%
  \else%
    \setlength{\unitlength}{\svgwidth}%
  \fi%
  \global\let\svgwidth\undefined%
  \global\let\svgscale\undefined%
  \makeatother%
  \begin{picture}(1,0.20729045)%
    \put(0,0){\includegraphics[width=\unitlength]{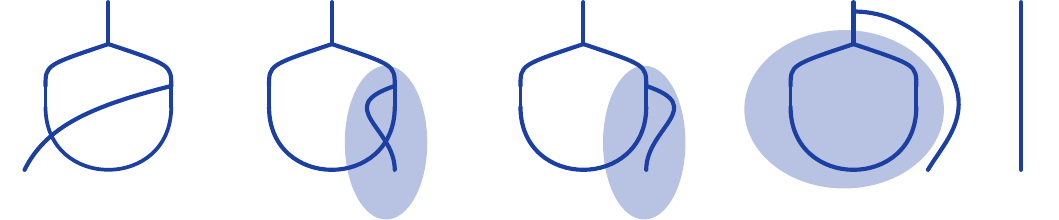}}%
    \put(-0.000935,0.10922393){\color[rgb]{0,0,0}\makebox(0,0)[lb]{\smash{$R$}}}%
    \put(0.18003358,0.10922393){\color[rgb]{0,0,0}\makebox(0,0)[lb]{\smash{$=$}}}%
    \put(0.21421402,0.10922393){\color[rgb]{0,0,0}\makebox(0,0)[lb]{\smash{$R$}}}%
    \put(0.41346454,0.10922393){\color[rgb]{0,0,0}\makebox(0,0)[lb]{\smash{$=$}}}%
    \put(0.4527504,0.10922393){\color[rgb]{0,0,0}\makebox(0,0)[lb]{\smash{$R$}}}%
    \put(0.6548447,0.10922393){\color[rgb]{0,0,0}\makebox(0,0)[lb]{\smash{$=$}}}%
    \put(0.71048607,0.10922393){\color[rgb]{0,0,0}\makebox(0,0)[lb]{\smash{$R$}}}%
    \put(0.91715658,0.10922393){\color[rgb]{0,0,0}\makebox(0,0)[lb]{\smash{$=$}}}%
    \put(0.01573871,0.01049258){\color[rgb]{0,0,0}\makebox(0,0)[lb]{\smash{$a$}}}%
    \put(0.36461351,0.01049258){\color[rgb]{0,0,0}\makebox(0,0)[lb]{\smash{$a$}}}%
    \put(0.60331732,0.01049258){\color[rgb]{0,0,0}\makebox(0,0)[lb]{\smash{$a$}}}%
    \put(0.86825232,0.01049258){\color[rgb]{0,0,0}\makebox(0,0)[lb]{\smash{$a$}}}%
    \put(0.95743836,0.01049258){\color[rgb]{0,0,0}\makebox(0,0)[lb]{\smash{$a$}}}%
  \end{picture}%
\endgroup%

$$
\noindent This is enough to establish $R.p$ as a projector onto $\mathcal{Z}_{\lambda}(A)$. The proof $p \circ \varphi = p$ is accomplished by direct composition. 
$$
\begingroup%
  \makeatletter%
  \providecommand\color[2][]{%
    \errmessage{(Inkscape) Color is used for the text in Inkscape, but the package 'color.sty' is not loaded}%
    \renewcommand\color[2][]{}%
  }%
  \providecommand\transparent[1]{%
    \errmessage{(Inkscape) Transparency is used (non-zero) for the text in Inkscape, but the package 'transparent.sty' is not loaded}%
    \renewcommand\transparent[1]{}%
  }%
  \providecommand\rotatebox[2]{#2}%
  \ifx\svgwidth\undefined%
    \setlength{\unitlength}{248.69101563bp}%
    \ifx\svgscale\undefined%
      \relax%
    \else%
      \setlength{\unitlength}{\unitlength * \real{\svgscale}}%
    \fi%
  \else%
    \setlength{\unitlength}{\svgwidth}%
  \fi%
  \global\let\svgwidth\undefined%
  \global\let\svgscale\undefined%
  \makeatother%
  \begin{picture}(1,0.35460913)%
    \put(0,0){\includegraphics[width=\unitlength]{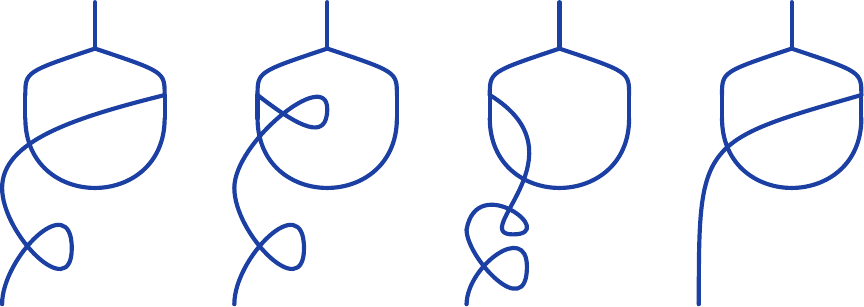}}%
    \put(0.22194136,0.21230292){\color[rgb]{0,0,0}\makebox(0,0)[lb]{\smash{$=$}}}%
    \put(0.49135198,0.21172843){\color[rgb]{0,0,0}\makebox(0,0)[lb]{\smash{$=$}}}%
    \put(0.75961371,0.21172843){\color[rgb]{0,0,0}\makebox(0,0)[lb]{\smash{$=$}}}%
  \end{picture}%
\endgroup%
	
$$
\noindent To show the identities $\varphi \circ n_1=n_2 = n_1 \circ \varphi$ hold one uses the fact $\varphi$ is an algebra automorphism.
$$
\begingroup%
  \makeatletter%
  \providecommand\color[2][]{%
    \errmessage{(Inkscape) Color is used for the text in Inkscape, but the package 'color.sty' is not loaded}%
    \renewcommand\color[2][]{}%
  }%
  \providecommand\transparent[1]{%
    \errmessage{(Inkscape) Transparency is used (non-zero) for the text in Inkscape, but the package 'transparent.sty' is not loaded}%
    \renewcommand\transparent[1]{}%
  }%
  \providecommand\rotatebox[2]{#2}%
  \ifx\svgwidth\undefined%
    \setlength{\unitlength}{297.08718262bp}%
    \ifx\svgscale\undefined%
      \relax%
    \else%
      \setlength{\unitlength}{\unitlength * \real{\svgscale}}%
    \fi%
  \else%
    \setlength{\unitlength}{\svgwidth}%
  \fi%
  \global\let\svgwidth\undefined%
  \global\let\svgscale\undefined%
  \makeatother%
  \begin{picture}(1,0.22344803)%
    \put(0,0){\includegraphics[width=\unitlength]{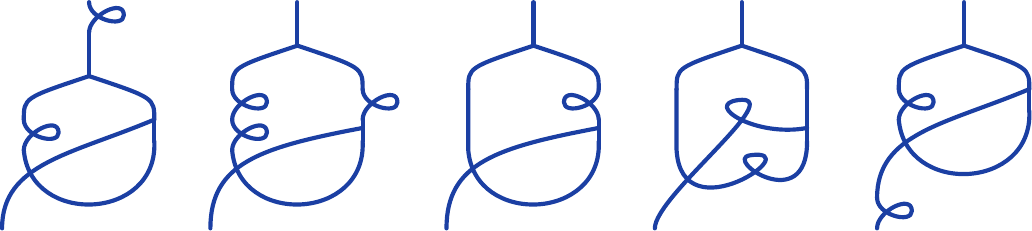}}%
    \put(0.16142672,0.07902445){\color[rgb]{0,0,0}\makebox(0,0)[lb]{\smash{$=$}}}%
    \put(0.39901958,0.07902445){\color[rgb]{0,0,0}\makebox(0,0)[lb]{\smash{$=$}}}%
    \put(0.59826032,0.07902445){\color[rgb]{0,0,0}\makebox(0,0)[lb]{\smash{$=$}}}%
    \put(0.79818116,0.07902445){\color[rgb]{0,0,0}\makebox(0,0)[lb]{\smash{$=$}}}%
  \end{picture}%
\endgroup%

$$
\noindent It is now shown that for all $a \in A$ the element $n_1(a)$ belongs to $\overline{\mathcal{Z}}_{\lambda}(A)$. 
$$
\begingroup%
  \makeatletter%
  \providecommand\color[2][]{%
    \errmessage{(Inkscape) Color is used for the text in Inkscape, but the package 'color.sty' is not loaded}%
    \renewcommand\color[2][]{}%
  }%
  \providecommand\transparent[1]{%
    \errmessage{(Inkscape) Transparency is used (non-zero) for the text in Inkscape, but the package 'transparent.sty' is not loaded}%
    \renewcommand\transparent[1]{}%
  }%
  \providecommand\rotatebox[2]{#2}%
  \ifx\svgwidth\undefined%
    \setlength{\unitlength}{291.4595459bp}%
    \ifx\svgscale\undefined%
      \relax%
    \else%
      \setlength{\unitlength}{\unitlength * \real{\svgscale}}%
    \fi%
  \else%
    \setlength{\unitlength}{\svgwidth}%
  \fi%
  \global\let\svgwidth\undefined%
  \global\let\svgscale\undefined%
  \makeatother%
  \begin{picture}(1,0.42692087)%
    \put(0,0){\includegraphics[width=\unitlength]{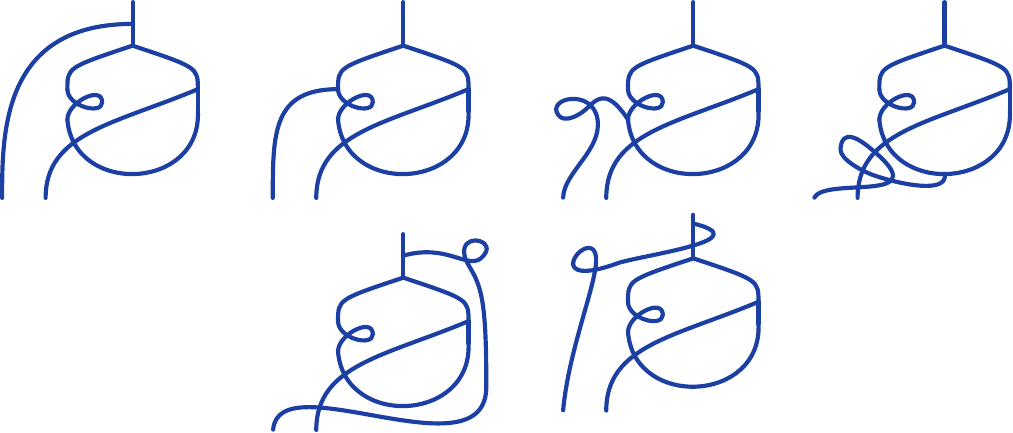}}%
    \put(0.21965378,0.31149388){\color[rgb]{0,0,0}\makebox(0,0)[lb]{\smash{$=$}}}%
    \put(0.49413441,0.31149388){\color[rgb]{0,0,0}\makebox(0,0)[lb]{\smash{$=$}}}%
    \put(0.78233908,0.31149388){\color[rgb]{0,0,0}\makebox(0,0)[lb]{\smash{$=$}}}%
    \put(0.21965378,0.07818535){\color[rgb]{0,0,0}\makebox(0,0)[lb]{\smash{$=$}}}%
    \put(0.49413441,0.07818535){\color[rgb]{0,0,0}\makebox(0,0)[lb]{\smash{$=$}}}%
  \end{picture}%
\endgroup%

\label{fig:nproof}	
$$
\noindent Finally it is established that if $a \in \overline{\mathcal{Z}}_{\lambda}(A)$ then $R.n_1(a)=a$. Then $R.n_1$ is a projector onto $\overline{\mathcal{Z}}_{\lambda}(A)$.
$$
\begingroup%
  \makeatletter%
  \providecommand\color[2][]{%
    \errmessage{(Inkscape) Color is used for the text in Inkscape, but the package 'color.sty' is not loaded}%
    \renewcommand\color[2][]{}%
  }%
  \providecommand\transparent[1]{%
    \errmessage{(Inkscape) Transparency is used (non-zero) for the text in Inkscape, but the package 'transparent.sty' is not loaded}%
    \renewcommand\transparent[1]{}%
  }%
  \providecommand\rotatebox[2]{#2}%
  \ifx\svgwidth\undefined%
    \setlength{\unitlength}{316.98049316bp}%
    \ifx\svgscale\undefined%
      \relax%
    \else%
      \setlength{\unitlength}{\unitlength * \real{\svgscale}}%
    \fi%
  \else%
    \setlength{\unitlength}{\svgwidth}%
  \fi%
  \global\let\svgwidth\undefined%
  \global\let\svgscale\undefined%
  \makeatother%
  \begin{picture}(1,0.23675949)%
    \put(0,0){\includegraphics[width=\unitlength]{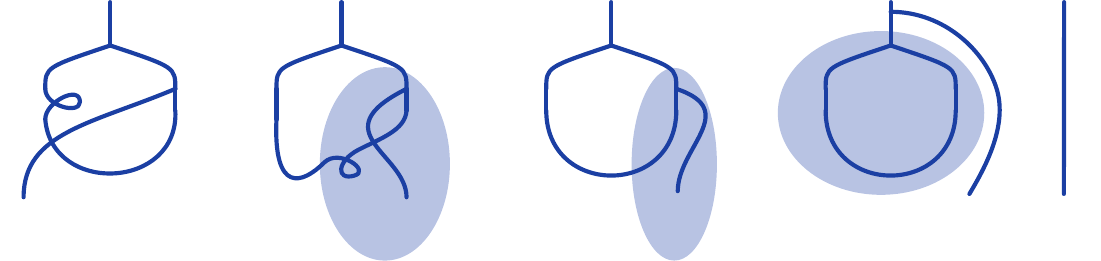}}%
    \put(-0.00089961,0.13664381){\color[rgb]{0,0,0}\makebox(0,0)[lb]{\smash{$R$}}}%
    \put(0.17270873,0.13728116){\color[rgb]{0,0,0}\makebox(0,0)[lb]{\smash{$=$}}}%
    \put(0.21553671,0.13728116){\color[rgb]{0,0,0}\makebox(0,0)[lb]{\smash{$R$}}}%
    \put(0.46039798,0.1379185){\color[rgb]{0,0,0}\makebox(0,0)[lb]{\smash{$R$}}}%
    \put(0.41502054,0.13728116){\color[rgb]{0,0,0}\makebox(0,0)[lb]{\smash{$=$}}}%
    \put(0.71405414,0.13728116){\color[rgb]{0,0,0}\makebox(0,0)[lb]{\smash{$R$}}}%
    \put(0.66357784,0.13728116){\color[rgb]{0,0,0}\makebox(0,0)[lb]{\smash{$=$}}}%
    \put(0.92857838,0.13728116){\color[rgb]{0,0,0}\makebox(0,0)[lb]{\smash{$=$}}}%
    \put(0.01514288,0.01766676){\color[rgb]{0,0,0}\makebox(0,0)[lb]{\smash{$a$}}}%
    \put(0.36090548,0.01766676){\color[rgb]{0,0,0}\makebox(0,0)[lb]{\smash{$a$}}}%
    \put(0.60823931,0.01766676){\color[rgb]{0,0,0}\makebox(0,0)[lb]{\smash{$a$}}}%
    \put(0.87323984,0.01766676){\color[rgb]{0,0,0}\makebox(0,0)[lb]{\smash{$a$}}}%
    \put(0.95904962,0.01766676){\color[rgb]{0,0,0}\makebox(0,0)[lb]{\smash{$a$}}}%
  \end{picture}%
\endgroup%

\label{fig:nproof2}	
$$
\end{proof}
\noindent The spin analogues of $z$ as defined in equation \eqref{eq:invariant} can be now found below. Denoted $\eta_1,\eta_2,\eta_3$ and $\chi$ they are preferred elements of the algebra -- the building blocks of the spin partition functions. 
\begin{equation*}
\hspace{3mm}
\begingroup%
  \makeatletter%
  \providecommand\color[2][]{%
    \errmessage{(Inkscape) Color is used for the text in Inkscape, but the package 'color.sty' is not loaded}%
    \renewcommand\color[2][]{}%
  }%
  \providecommand\transparent[1]{%
    \errmessage{(Inkscape) Transparency is used (non-zero) for the text in Inkscape, but the package 'transparent.sty' is not loaded}%
    \renewcommand\transparent[1]{}%
  }%
  \providecommand\rotatebox[2]{#2}%
  \ifx\svgwidth\undefined%
    \setlength{\unitlength}{388.38898926bp}%
    \ifx\svgscale\undefined%
      \relax%
    \else%
      \setlength{\unitlength}{\unitlength * \real{\svgscale}}%
    \fi%
  \else%
    \setlength{\unitlength}{\svgwidth}%
  \fi%
  \global\let\svgwidth\undefined%
  \global\let\svgscale\undefined%
  \makeatother%
  \begin{picture}(1,0.17052895)%
    \put(0,0){\includegraphics[width=\unitlength]{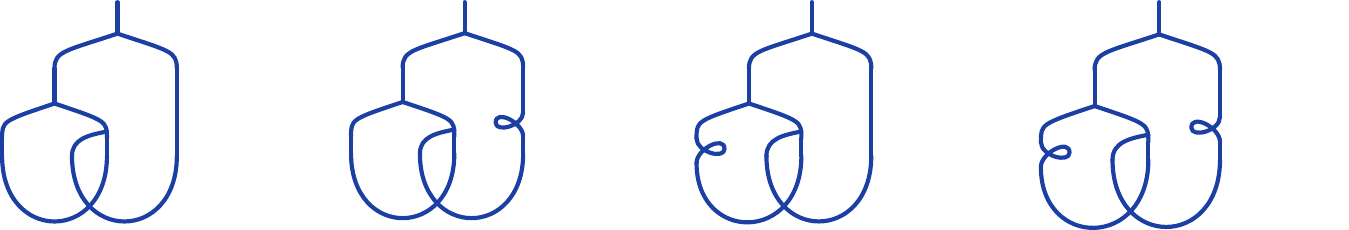}}%
    \put(0.14266247,0.06464588){\color[rgb]{0,0,0}\makebox(0,0)[lb]{\smash{$=$}}}%
    \put(0.17355933,0.06464588){\color[rgb]{0,0,0}\makebox(0,0)[lb]{\smash{$\eta_1$}}}%
    \put(0.4001363,0.06464588){\color[rgb]{0,0,0}\makebox(0,0)[lb]{\smash{$=$ }}}%
    \put(0.65761013,0.06464588){\color[rgb]{0,0,0}\makebox(0,0)[lb]{\smash{$=$}}}%
    \put(0.94598081,0.06464588){\color[rgb]{0,0,0}\makebox(0,0)[lb]{\smash{$\chi$}}}%
    \put(0.91508395,0.06464588){\color[rgb]{0,0,0}\makebox(0,0)[lb]{\smash{$=$}}}%
    \put(0.68850698,0.06464588){\color[rgb]{0,0,0}\makebox(0,0)[lb]{\smash{$\eta_3$}}}%
    \put(0.43103316,0.06464588){\color[rgb]{0,0,0}\makebox(0,0)[lb]{\smash{$\eta_2$}}}%
  \end{picture}%
\endgroup%

\end{equation*}

It is easy to verify the identity $\eta_1=\eta_2=\eta_3$ holds; the notation $\eta$ is used for any of these maps. To see how the result holds note that one of the relations, $\eta_1=\eta_2$, is trivial -- it follows from $p \circ \varphi=p$. The proof for the remaining equation, $\eta_3=\eta_1$ is depicted below. 
$$
\centering
\begingroup%
  \makeatletter%
  \providecommand\color[2][]{%
    \errmessage{(Inkscape) Color is used for the text in Inkscape, but the package 'color.sty' is not loaded}%
    \renewcommand\color[2][]{}%
  }%
  \providecommand\transparent[1]{%
    \errmessage{(Inkscape) Transparency is used (non-zero) for the text in Inkscape, but the package 'transparent.sty' is not loaded}%
    \renewcommand\transparent[1]{}%
  }%
  \providecommand\rotatebox[2]{#2}%
  \ifx\svgwidth\undefined%
    \setlength{\unitlength}{350.64528809bp}%
    \ifx\svgscale\undefined%
      \relax%
    \else%
      \setlength{\unitlength}{\unitlength * \real{\svgscale}}%
    \fi%
  \else%
    \setlength{\unitlength}{\svgwidth}%
  \fi%
  \global\let\svgwidth\undefined%
  \global\let\svgscale\undefined%
  \makeatother%
  \begin{picture}(1,0.20523171)%
    \put(0,0){\includegraphics[width=\unitlength]{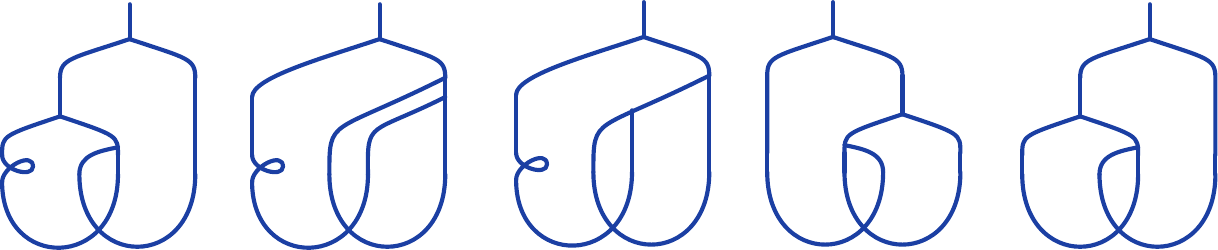}}%
    \put(0.17084275,0.06536852){\color[rgb]{0,0,0}\makebox(0,0)[lb]{\smash{$=$}}}%
    \put(0.38021163,0.06479236){\color[rgb]{0,0,0}\makebox(0,0)[lb]{\smash{$=$}}}%
    \put(0.59119318,0.06421621){\color[rgb]{0,0,0}\makebox(0,0)[lb]{\smash{$=$}}}%
    \put(0.80229059,0.06479236){\color[rgb]{0,0,0}\makebox(0,0)[lb]{\smash{$=$}}}%
  \end{picture}%
\endgroup%

$$
The two non-equivalent generalisations of $z$ have the following properties. 
\begin{lemma} \label{lem:properties_eta_chi}
The elements $\eta$ and $\chi$ are central and satisfy $\eta^2=\chi^2$.
\end{lemma}
\begin{proof}
\noindent One is able to easily conclude $\eta$ is an element of $\mathcal{Z}_{\lambda}(A)$. Given the increasing complexity of the diagrams requiring simplification, lines being transformed have been dashed. 

$$
\centering
\begingroup%
  \makeatletter%
  \providecommand\color[2][]{%
    \errmessage{(Inkscape) Color is used for the text in Inkscape, but the package 'color.sty' is not loaded}%
    \renewcommand\color[2][]{}%
  }%
  \providecommand\transparent[1]{%
    \errmessage{(Inkscape) Transparency is used (non-zero) for the text in Inkscape, but the package 'transparent.sty' is not loaded}%
    \renewcommand\transparent[1]{}%
  }%
  \providecommand\rotatebox[2]{#2}%
  \ifx\svgwidth\undefined%
    \setlength{\unitlength}{388.72768555bp}%
    \ifx\svgscale\undefined%
      \relax%
    \else%
      \setlength{\unitlength}{\unitlength * \real{\svgscale}}%
    \fi%
  \else%
    \setlength{\unitlength}{\svgwidth}%
  \fi%
  \global\let\svgwidth\undefined%
  \global\let\svgscale\undefined%
  \makeatother%
  \begin{picture}(1,0.22529832)%
    \put(0,0){\includegraphics[width=\unitlength]{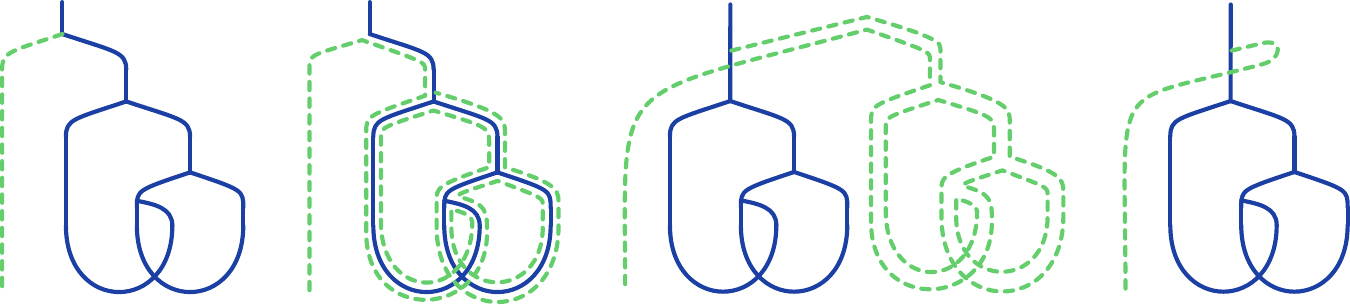}}%
    \put(0.19129257,0.06999821){\color[rgb]{0,0,0}\makebox(0,0)[lb]{\smash{$=$}}}%
    \put(0.42575708,0.06926327){\color[rgb]{0,0,0}\makebox(0,0)[lb]{\smash{$=$}}}%
    \put(0.79840137,0.06926327){\color[rgb]{0,0,0}\makebox(0,0)[lb]{\smash{$=$}}}%
  \end{picture}%
\endgroup%

$$
The first step uses multiplication associativity, and the multiplication and crossing compatibility a number of times. The second step uses axioms  \eqref{fig:axiom_mult} to \eqref{fig:axiom_Reid_3} and reflects the fact lines can be freely moved past each other as long as their boundaries remains fixed. The last step uses the condition $\varphi^2=\text{id}$ (note the number of times the $\varphi$ map appears is even). 

In addition, because $\eta$ is determined by a diagram closed from below (a diagram with no downward-pointing legs), $\eta a=a\eta$ for all $a \in A$.
$$
\begingroup%
  \makeatletter%
  \providecommand\color[2][]{%
    \errmessage{(Inkscape) Color is used for the text in Inkscape, but the package 'color.sty' is not loaded}%
    \renewcommand\color[2][]{}%
  }%
  \providecommand\transparent[1]{%
    \errmessage{(Inkscape) Transparency is used (non-zero) for the text in Inkscape, but the package 'transparent.sty' is not loaded}%
    \renewcommand\transparent[1]{}%
  }%
  \providecommand\rotatebox[2]{#2}%
  \ifx\svgwidth\undefined%
    \setlength{\unitlength}{265.2bp}%
    \ifx\svgscale\undefined%
      \relax%
    \else%
      \setlength{\unitlength}{\unitlength * \real{\svgscale}}%
    \fi%
  \else%
    \setlength{\unitlength}{\svgwidth}%
  \fi%
  \global\let\svgwidth\undefined%
  \global\let\svgscale\undefined%
  \makeatother%
  \begin{picture}(1,0.39661731)%
    \put(0,0){\includegraphics[width=\unitlength]{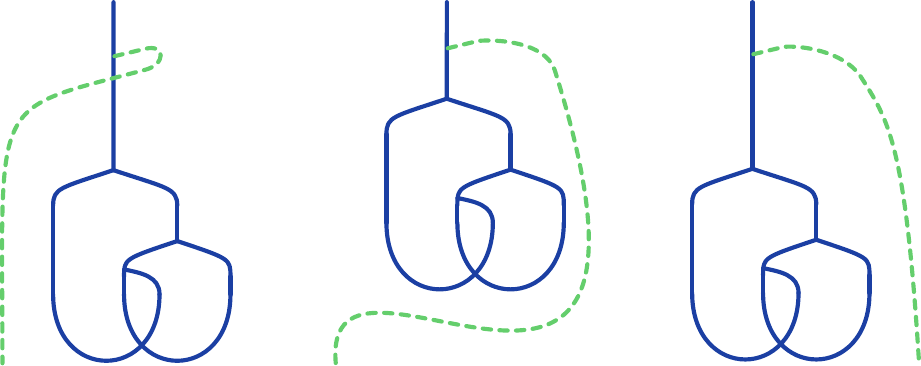}}%
    \put(0.30429864,0.10769326){\color[rgb]{0,0,0}\makebox(0,0)[lb]{\smash{$=$}}}%
    \put(0.68137255,0.10769326){\color[rgb]{0,0,0}\makebox(0,0)[lb]{\smash{$=$}}}%
  \end{picture}%
\endgroup%

$$
In other words, $\eta \in \mathcal{Z}(A)$. Establishing the same result for $\chi$ is entirely analogous. 

The last and most lengthy part of the proof comes from determining a non-trivial identity: $\chi^2=\eta^2$. To make the exposition more clear each line of the proof begins with the transformed-to-be diagram line dashed.
$$
\begingroup%
  \makeatletter%
  \providecommand\color[2][]{%
    \errmessage{(Inkscape) Color is used for the text in Inkscape, but the package 'color.sty' is not loaded}%
    \renewcommand\color[2][]{}%
  }%
  \providecommand\transparent[1]{%
    \errmessage{(Inkscape) Transparency is used (non-zero) for the text in Inkscape, but the package 'transparent.sty' is not loaded}%
    \renewcommand\transparent[1]{}%
  }%
  \providecommand\rotatebox[2]{#2}%
  \ifx\svgwidth\undefined%
    \setlength{\unitlength}{373.40422363bp}%
    \ifx\svgscale\undefined%
      \relax%
    \else%
      \setlength{\unitlength}{\unitlength * \real{\svgscale}}%
    \fi%
  \else%
    \setlength{\unitlength}{\svgwidth}%
  \fi%
  \global\let\svgwidth\undefined%
  \global\let\svgscale\undefined%
  \makeatother%
  \begin{picture}(1,0.19058238)%
    \put(0,0){\includegraphics[width=\unitlength]{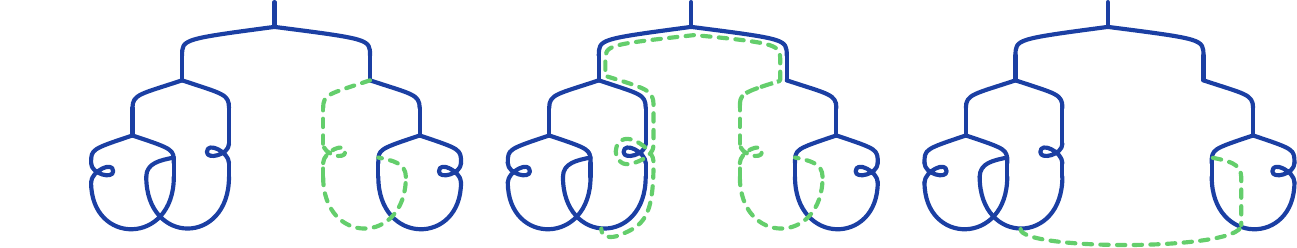}}%
    \put(-0.00076367,0.09492803){\color[rgb]{0,0,0}\makebox(0,0)[lb]{\smash{$\chi^2$}}}%
    \put(0.03137309,0.09492803){\color[rgb]{0,0,0}\makebox(0,0)[lb]{\smash{$=$}}}%
    \put(0.36345289,0.09492803){\color[rgb]{0,0,0}\makebox(0,0)[lb]{\smash{$=$}}}%
    \put(0.68482044,0.09492803){\color[rgb]{0,0,0}\makebox(0,0)[lb]{\smash{$=$}}}%
  \end{picture}%
\endgroup%
	
$$
$$
\hspace{5mm}
\begingroup%
  \makeatletter%
  \providecommand\color[2][]{%
    \errmessage{(Inkscape) Color is used for the text in Inkscape, but the package 'color.sty' is not loaded}%
    \renewcommand\color[2][]{}%
  }%
  \providecommand\transparent[1]{%
    \errmessage{(Inkscape) Transparency is used (non-zero) for the text in Inkscape, but the package 'transparent.sty' is not loaded}%
    \renewcommand\transparent[1]{}%
  }%
  \providecommand\rotatebox[2]{#2}%
  \ifx\svgwidth\undefined%
    \setlength{\unitlength}{365.97004395bp}%
    \ifx\svgscale\undefined%
      \relax%
    \else%
      \setlength{\unitlength}{\unitlength * \real{\svgscale}}%
    \fi%
  \else%
    \setlength{\unitlength}{\svgwidth}%
  \fi%
  \global\let\svgwidth\undefined%
  \global\let\svgscale\undefined%
  \makeatother%
  \begin{picture}(1,0.20815369)%
    \put(0,0){\includegraphics[width=\unitlength]{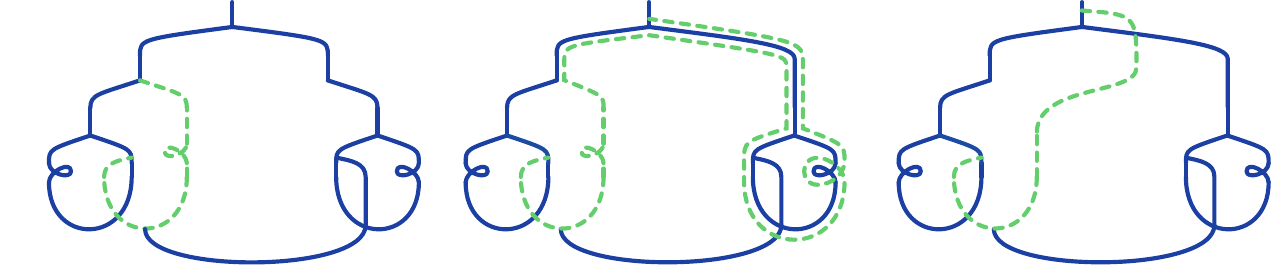}}%
    \put(-0.00077918,0.10494738){\color[rgb]{0,0,0}\makebox(0,0)[lb]{\smash{$=$}}}%
    \put(0.33804636,0.10494738){\color[rgb]{0,0,0}\makebox(0,0)[lb]{\smash{$=$}}}%
    \put(0.67687194,0.10494738){\color[rgb]{0,0,0}\makebox(0,0)[lb]{\smash{$=$}}}%
  \end{picture}%
\endgroup%
	
$$
$$
\hspace{4mm}
\begingroup%
  \makeatletter%
  \providecommand\color[2][]{%
    \errmessage{(Inkscape) Color is used for the text in Inkscape, but the package 'color.sty' is not loaded}%
    \renewcommand\color[2][]{}%
  }%
  \providecommand\transparent[1]{%
    \errmessage{(Inkscape) Transparency is used (non-zero) for the text in Inkscape, but the package 'transparent.sty' is not loaded}%
    \renewcommand\transparent[1]{}%
  }%
  \providecommand\rotatebox[2]{#2}%
  \ifx\svgwidth\undefined%
    \setlength{\unitlength}{363.95307617bp}%
    \ifx\svgscale\undefined%
      \relax%
    \else%
      \setlength{\unitlength}{\unitlength * \real{\svgscale}}%
    \fi%
  \else%
    \setlength{\unitlength}{\svgwidth}%
  \fi%
  \global\let\svgwidth\undefined%
  \global\let\svgscale\undefined%
  \makeatother%
  \begin{picture}(1,0.22500166)%
    \put(0,0){\includegraphics[width=\unitlength]{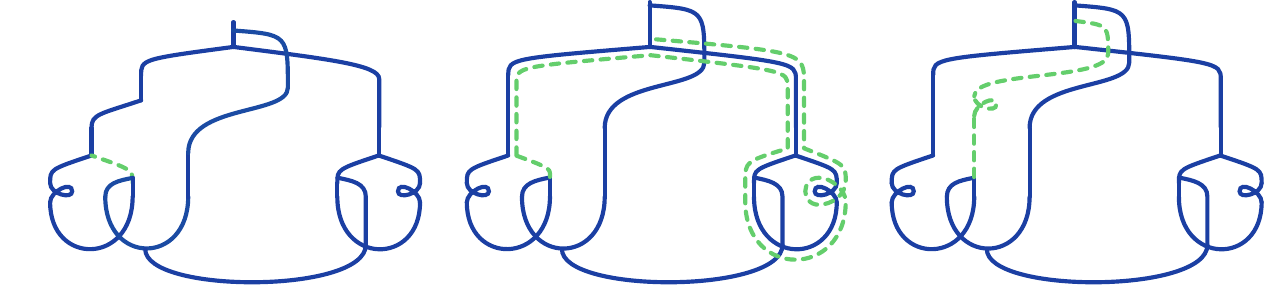}}%
    \put(-0.0007835,0.11029635){\color[rgb]{0,0,0}\makebox(0,0)[lb]{\smash{$=$}}}%
    \put(0.33991977,0.11029635){\color[rgb]{0,0,0}\makebox(0,0)[lb]{\smash{$=$}}}%
    \put(0.68062304,0.11029635){\color[rgb]{0,0,0}\makebox(0,0)[lb]{\smash{$=$}}}%
  \end{picture}%
\endgroup%
	
$$
$$
\hspace{4mm}
\begingroup%
  \makeatletter%
  \providecommand\color[2][]{%
    \errmessage{(Inkscape) Color is used for the text in Inkscape, but the package 'color.sty' is not loaded}%
    \renewcommand\color[2][]{}%
  }%
  \providecommand\transparent[1]{%
    \errmessage{(Inkscape) Transparency is used (non-zero) for the text in Inkscape, but the package 'transparent.sty' is not loaded}%
    \renewcommand\transparent[1]{}%
  }%
  \providecommand\rotatebox[2]{#2}%
  \ifx\svgwidth\undefined%
    \setlength{\unitlength}{366.06345215bp}%
    \ifx\svgscale\undefined%
      \relax%
    \else%
      \setlength{\unitlength}{\unitlength * \real{\svgscale}}%
    \fi%
  \else%
    \setlength{\unitlength}{\svgwidth}%
  \fi%
  \global\let\svgwidth\undefined%
  \global\let\svgscale\undefined%
  \makeatother%
  \begin{picture}(1,0.23406178)%
    \put(0,0){\includegraphics[width=\unitlength]{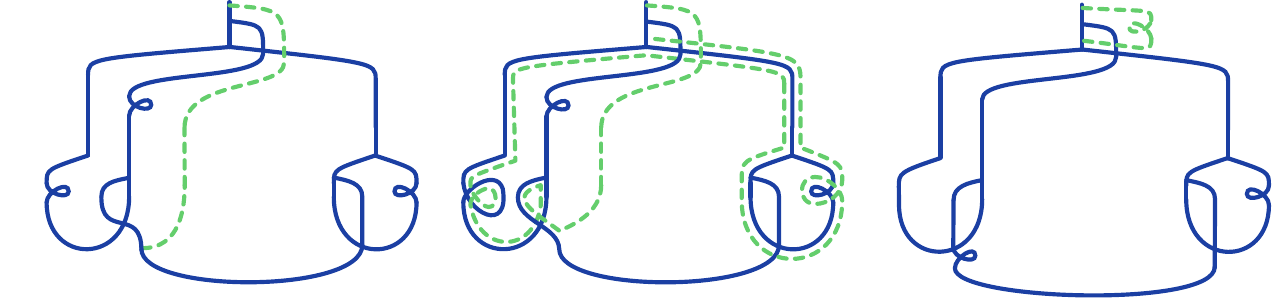}}%
    \put(-0.00077899,0.11375792){\color[rgb]{0,0,0}\makebox(0,0)[lb]{\smash{$=$}}}%
    \put(0.3379601,0.11375792){\color[rgb]{0,0,0}\makebox(0,0)[lb]{\smash{$=$}}}%
    \put(0.67669919,0.11375792){\color[rgb]{0,0,0}\makebox(0,0)[lb]{\smash{$=$}}}%
  \end{picture}%
\endgroup%
	
$$
$$
\begingroup%
  \makeatletter%
  \providecommand\color[2][]{%
    \errmessage{(Inkscape) Color is used for the text in Inkscape, but the package 'color.sty' is not loaded}%
    \renewcommand\color[2][]{}%
  }%
  \providecommand\transparent[1]{%
    \errmessage{(Inkscape) Transparency is used (non-zero) for the text in Inkscape, but the package 'transparent.sty' is not loaded}%
    \renewcommand\transparent[1]{}%
  }%
  \providecommand\rotatebox[2]{#2}%
  \ifx\svgwidth\undefined%
    \setlength{\unitlength}{353.62915039bp}%
    \ifx\svgscale\undefined%
      \relax%
    \else%
      \setlength{\unitlength}{\unitlength * \real{\svgscale}}%
    \fi%
  \else%
    \setlength{\unitlength}{\svgwidth}%
  \fi%
  \global\let\svgwidth\undefined%
  \global\let\svgscale\undefined%
  \makeatother%
  \begin{picture}(1,0.36049677)%
    \put(0,0){\includegraphics[width=\unitlength]{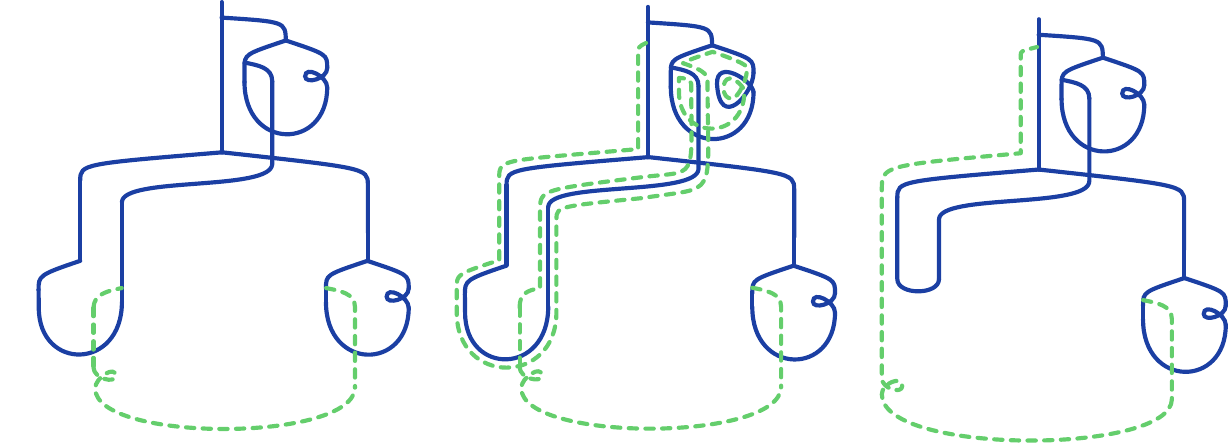}}%
    \put(-0.00080637,0.15119267){\color[rgb]{0,0,0}\makebox(0,0)[lb]{\smash{$=$}}}%
    \put(0.33853217,0.15119267){\color[rgb]{0,0,0}\makebox(0,0)[lb]{\smash{$=$}}}%
    \put(0.67787071,0.15119267){\color[rgb]{0,0,0}\makebox(0,0)[lb]{\smash{$=$}}}%
  \end{picture}%
\endgroup%
	
$$
$$
\hspace{5mm}
\begingroup%
  \makeatletter%
  \providecommand\color[2][]{%
    \errmessage{(Inkscape) Color is used for the text in Inkscape, but the package 'color.sty' is not loaded}%
    \renewcommand\color[2][]{}%
  }%
  \providecommand\transparent[1]{%
    \errmessage{(Inkscape) Transparency is used (non-zero) for the text in Inkscape, but the package 'transparent.sty' is not loaded}%
    \renewcommand\transparent[1]{}%
  }%
  \providecommand\rotatebox[2]{#2}%
  \ifx\svgwidth\undefined%
    \setlength{\unitlength}{389.7890625bp}%
    \ifx\svgscale\undefined%
      \relax%
    \else%
      \setlength{\unitlength}{\unitlength * \real{\svgscale}}%
    \fi%
  \else%
    \setlength{\unitlength}{\svgwidth}%
  \fi%
  \global\let\svgwidth\undefined%
  \global\let\svgscale\undefined%
  \makeatother%
  \begin{picture}(1,0.31745902)%
    \put(0,0){\includegraphics[width=\unitlength]{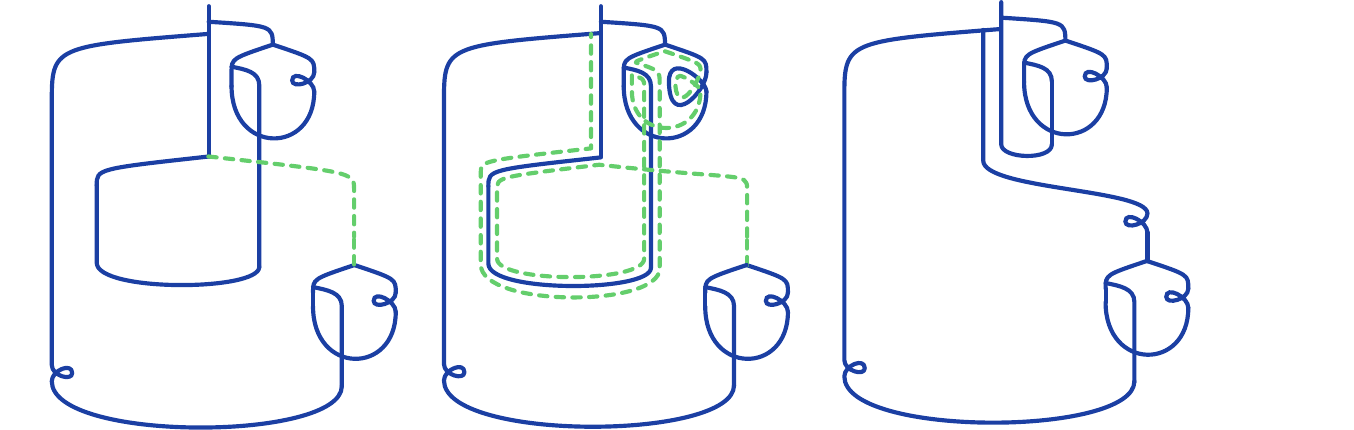}}%
    \put(-0.00073157,0.13709755){\color[rgb]{0,0,0}\makebox(0,0)[lb]{\smash{$=$}}}%
    \put(0.28660333,0.13709755){\color[rgb]{0,0,0}\makebox(0,0)[lb]{\smash{$=$}}}%
    \put(0.57393823,0.13709755){\color[rgb]{0,0,0}\makebox(0,0)[lb]{\smash{$=$}}}%
    \put(0.89205901,0.13709755){\color[rgb]{0,0,0}\makebox(0,0)[lb]{\smash{$=$}}}%
    \put(0.92284489,0.13709755){\color[rgb]{0,0,0}\makebox(0,0)[lb]{\smash{$\eta^2$}}}%
  \end{picture}%
\endgroup%
	
$$
\end{proof}
The partition function for a surface with spin structure can now be presented. Each handle contributes the element $\chi$ or $\eta$ depending on the spin structure; the partition function is thus
\begin{align}Z(\Sigma_g,s)=R\varepsilon(\eta^{g-l}\chi^{l}).\end{align}
However, the properties described in lemma~\ref{lem:properties_eta_chi} mean that all that matters is $l\,\mo 2$ which reflects the fact the partition function is a homeomorphism invariant. The only homeomorphism invariant of a spin structure is the Arf invariant of the quadratic form, $\Arf(q)=l\,\mo 2\in\Zb_2$. 
It is most convenient to express this invariant of the spin structure as a sign $P(s)=(-1)^{\Arf(q)}$ called the parity. The spin structure is called even if $P(s)=1$ and odd if $P(s)=-1$.  

 These results are collected together to give the main result for this section.
\begin{theorem} \label{theo:main}
Let  $(C,B,R,\lambda)$ be a spin state sum model. Then the partition function $Z$ of a triangulated surface $\Sigma_g$ of genus $g$ immersed in $\Rb^3$ depends only on $g$ and the parity of the spin structure $s$. Moreover,
\begin{align}
Z(\Sigma_g,s)=\left\lbrace
\begin{array}{lcl}
R\varepsilon(\eta^g)&\quad& \text{($s$ even parity)}\\
R\varepsilon(\chi\eta^{g-1})&& \text{($s$ odd parity).}
\end{array}
\right. 
\end{align}
\end{theorem}
Note that $Z(\Sigma_0)$ is independent of the choice of $\lambda$, as is to be expected. According to the classification of planar state sum models given in theorem~\ref{theo:diagram_semi_simple}, $\eta \in \mathcal{Z}(A)$ implies that $\eta=\oplus_i \,\eta_i\,1_i$ for some constants $\eta_i \in \Rb, \Cb_{\Rb}$ or $\Cb$. The expression for $\chi$ will therefore be $\chi=\oplus_i \,\text{sgn}_i\,\eta_i\,1_i$, where $\text{sgn}_i=\pm 1$, since $\chi$ is also a central element and $\eta^2=\chi^2$. In particular this means simple matrix algebras can at most attribute different signs to spin structures of different parity.

An algebraic condition that guarantees topologically-inequivalent spin structures cannot be distinguished is $\eta=\chi$. It is now shown that the canonical crossing map gives rise to spin state sum models that fall into this class.
\begin{corollary}
Let $\lambda \colon A \otimes A \to A \otimes A$ be such that $a \otimes b \mapsto b \otimes a$. Then $\chi=\eta$, implying the partition function does not depend on the spin structure.
\end{corollary}
\begin{proof}
For a crossing of the form above it is easy to conclude $\varphi=\sigma$ where $\sigma$ represents the Nakayama automorphism associated with the Frobenius form $\varepsilon$. The set $\overline{\mathcal{Z}}_{\lambda}(A)$ coincides in this case with the set of elements $a \in A$ satisfying $ab=\sigma(b)a$ for all $b \in A$. Recall that if an algebra $A$ satisfies the conditions of theorem~\ref{theo:diagram} then $\sigma$ is an inner automorphism: $\sigma(a)=xax^{-1}$. Then it is possible to conclude $n_2(a)=n_1(a)$ for all $a \in A$ by the argument
\begin{align} 
n_2(a)=\sigma \circ n_1 (a)=x n_1(a)x^{-1}=x \sigma(x^{-1})n_1(a)=n_1(a).
\end{align}
The diagrammatic form of $\eta$ and $\chi$ implies that $\eta=\chi$ if the maps $n_1$ and $n_2$ coincide. Then theorem \ref{theo:main} implies that the partition function does not distinguish spin parity. 

\end{proof}
Our conclusions so far do not guarantee the existence of crossing maps satisfying $\eta \neq \chi$. The last efforts in this section therefore concentrate on presenting various examples of such algebras.

\begin{example}[Algebras $A=\Mb_n(\Cb_{\Rb})$] \label{ex:complex_crossing}\hspace{1mm}\\
\noindent These algebras are naturally $\Zb_2$-graded: $\Mb_n(\Cb_{\Rb}) =A_0 \bigoplus A_1$ with $A_0=\Mb_{n}(\Rb)$ and $A_1=\hat{\imath}\Mb_{n}(\Rb)$. There is a unique non-trivial crossing that can be constructed from a $\Zb_2$-bicharacter $\tilde{\lambda}$ (see example~\ref{ex:non_symmetric_Frobenius}). The components of the crossing are determined by the relation $\tilde{\lambda}(h,j)=(-1)^{hj}$. If the Frobenius form is taken to be symmetric then
\begin{align}
B=(2Rn)^{-1}\sum_{lm}\left(e_{lm} \otimes e_{ml}-\hat{\imath}e_{lm} \otimes \hat{\imath}e_{ml}\right). 
\end{align}
Note the constant $2Rn$ arises from the definition of the Frobenius form, $\varepsilon(a)=2Rn\Real \Tr(a)$. This information can be used to determine the relation 
\begin{align}\label{eq:eta_com}
\eta=(2Rn)^{-2}\left(\sum_{lmrs}e_{lm}e_{rs}e_{ml}e_{sr}\right)\left(\sum_{hj}\tilde{\lambda}(h,j)\right).
\end{align}
Further identifying the first sum as the unit element and the second one as the constant $2$, one concludes that $\eta=2(2Rn)^{-2}1$. The element $\chi$ is constructed in an analogous fashion; however, the term $\sum_{hj}\tilde{\lambda}(h,j)$ is replaced with $\sum_{hj}\tilde{\lambda}(h,h)\tilde{\lambda}(h,j)\tilde{\lambda}(j,j)=-2$. Hence, $\chi=-2(2Rn)^{-2}1$. The invariant produced distinguishes spin structures of different parity and takes the form
\begin{align}\label{eq:first_spin_inv}
Z(\Sigma_g,s)=P(s)2^{1-g}R^{2-2g}n^{2-2g}
\end{align}
This expression can be compared with the relation found for FHK state sums models where $Z(\Sigma_g)=2R^{2-2g}n^{2-2g}$. Note that in this canonical case $\eta=\chi=z$ and the term $\sum_{hj}\tilde{\lambda}(h,j)$ in equation \eqref{eq:eta_com} is equal to $4$; hence $z=(Rn)^{-2}1$, as was previously remarked in \S \ref{sec:lattice_tft}.

The result \eqref{eq:first_spin_inv} can be seen as a special case of the one obtained for $\varepsilon(a)=2R(p-q)\Real\Tr(u a)$, $u=\text{diag}(p,q)$, $p\neq q$ and the same algebra grading:
\begin{align}
Z(\Sigma_g,s)=P(s)2^{1-g}R^{2-2g}(p-q)^{2-2g}.
\end{align}
To reach this conclusion note the expression for $B$ is now
\begin{align}
B=(2R(p-q))^{-1}\sum_{lm}\left(e_{lm} u\otimes e_{ml}-\hat{\imath}e_{lm} u\otimes \hat{\imath}e_{ml}\right).
\end{align}
By recognising the Nakayama automorphism $\sigma$ associated with $\varepsilon$ satisfies $\sigma(a)=uau$,  $\eta$ can be written as 
\begin{align} \label{eq:eta_non_com}
\eta=(2R(p-q))^{-2}\left(\sum_{lmrs}e_{lm}\sigma(e_{rs})e_{ml}e_{sr}\right)\left(\sum_{hj}\tilde{\lambda}(h,j)\right).
\end{align}
The action of $\sigma$ separates elements of $\Mb_{n}(\Rb)$ into two types (see example~\ref{ex:non_symmetric_Frobenius}): $\sigma(a)=a$ if $a$ is block-diagonal and $\sigma(a)=-a$ if $a$ is block-anti-diagonal. However, if $a$ is block-anti-diagonal $\sum_{lm}e_{lm} ae_{ml}=0$ effectively reducing \eqref{eq:eta_non_com} to \eqref{eq:eta_com} with the replacement $n \to p-q$. A similar reasoning holds for $\chi$.

\end{example}
\begin{example}[Algebras $A=\Mb_n(\Hb_{\Rb})$] \label{ex:quaternions_crossing} \hspace{1mm}\\
\noindent Consider the group of the quaternions $\widehat{\mathcal{K}}=\left\lbrace 1,\hat{\imath},\hat{\jmath},\hat{k}, -1,-\hat{\imath},-\hat{\jmath},-\hat{k}\right\rbrace  \ni w$ and define $A_w=w\Mb_{n}(\Rb)$. Then $A_w=A_{-w}$ and $A_wA_t=A_{wt}$, so that the algebra  $A=\Mb_n(\Hb_{\Rb})$ is graded by the quotient group $\mathcal{K}=\widehat{\mathcal{K}}/\{\pm1\}$, which is isomorphic to the Klein group $\Zb_2 \times \Zb_2$.
 The grading is conveniently written $\Mb_n(\Hb_{\Rb}) = \oplus_{w}A_w$ with $w\in\left\lbrace 1,\hat{\imath},\hat{\jmath},\hat{k}\right\rbrace$. 

The following table encodes the components of all the possible $\mathcal{K}$-bicharacters $\tilde{\lambda}$ that would give rise to a crossing. Each triple $(\alpha,\beta,\gamma)$ with $\alpha,\beta, \gamma \in \lbrace -1,1\rbrace$ determines one such bicharacter.
\begin{center}
  \begin{tabular}{| c | c | c | c | c |}
    \hline
    $\tilde{\lambda}(w,t)$	& $\, 1 \,$	& $\hat{\imath}$	& $\hat{\jmath}$	& $\hat{k}$\\ \hline
    $1$			& $1$ & $1$			& $1$ 			& $1$ \\ \hline
   $\hat{\imath}$	& $1$	&  $\alpha\beta$	& $\alpha$		& $\beta$ \\ \hline
   $\hat{\jmath}$	& $1$ 	&  $\alpha$		& $\alpha\gamma$	& $\gamma$ \\ \hline
   $\hat{k}$			& $1$ 	&  $\beta$		& $\gamma$		& $\beta\gamma$\\
    \hline
  \end{tabular}
\end{center}
If the Frobenius form is taken to be symmetric then
\begin{align}
B=(4Rn)^{-1}\sum_{lm,w}\left(w \,e_{lm} \otimes w^{\ast}\,e_{ml}\right).
\end{align}
Note the constant $4Rn$ arises from the definition of the Frobenius form, $\varepsilon(a)=4Rn\Real \Tr(a)$. This information can be used to determine the identity 
\begin{align}
\eta=(4Rn)^{-2}\left(\sum_{lmrs}e_{lm}e_{rs}e_{ml}e_{sr}\right)\left(\sum_{wt}\tilde{\lambda}(w,t)wtw^{\ast}t^{\ast}\right).
\end{align}
The first sum is simply the identity element 1. The second, with some algebraic manipulation, can be seen to satisfy
\begin{align}
\sum_{wt}\tilde{\lambda}(w,t)wtw^{\ast}t^{\ast}=\frac{11}{2}-2\Lambda+ \frac{\Lambda^2}{2}
\end{align}
with $\Lambda = \alpha + \beta + \gamma$. The element $\chi$ is constructed in an analogous fashion; however, the term $\sum_{wt}\tilde{\lambda}(w,t)wtw^{\ast}t^{\ast}$ is replaced with 
\begin{equation}
\begin{aligned}
\sum_{wt}\tilde{\lambda}(w,w)\tilde{\lambda}(w,t)\tilde{\lambda}(t,t)wtw^{\ast}t^{\ast}=-\frac{7}{2}+7\Lambda+\frac{3}{2}\Lambda^2-\Lambda^3&\\
=\sum_{wt}\tilde{\lambda}(w,t)wtw^{\ast}t^{\ast}-\left(\Lambda -1\right)\left(\Lambda -3\right)\left(\Lambda +3\right)&
\end{aligned}
\end{equation}
It is easy to verify $\Lambda \in \lbrace -3,-1,1,3\rbrace$ (the canonical crossing corresponds to the choice $\Lambda=3$). Therefore only the crossings satisfying $\Lambda=-1$ distinguish spin structures. 
The partition functions are  
\begin{align} \label{eq:quaternion_spin}
Z(\Sigma_g,s)=\begin{cases}4(Rn)^{2-2g} &(\Lambda=-3) \\ P(s)2^{2-g}(Rn)^{2-2g}&(\Lambda=-1) \\ (2Rn)^{2-2g} &(\Lambda=+1 \text{ or } +3) \end{cases}
\end{align}

The result \eqref{eq:quaternion_spin} can be slightly generalised -- in a manner identical to that described in example \ref{ex:complex_crossing} -- by replacing the symmetric bilinear form with $\varepsilon(a)=4R(p-q)\Real\Tr(u a)$, $u=\text{diag}(p,q)$, $p\neq q$. The partition functions read as \eqref{eq:quaternion_spin} but with the replacement $n \to p-q$.
\end{example}

The final example presents all possible crossings  for some low-dimensional commutative algebras. Some of these models have the property that $\eta\ne\pm \chi$, in contrast to the previous examples.
\begin{example}[Algebras $\oplus_1^m\Cb$ for $m=2,3,4$] \hspace{1mm}\\ \label{ex:cyclic}
The algebra $A=\Cb\oplus\ldots\oplus\Cb$ is isomorphic to the group algebra of the cyclic group, $A\cong\Cb C_m$, and this isomorphism is useful in presenting the results. The notation $C_m=\lbrace e,h,\cdots,h^{m-1}\rbrace$ is used. For $m=2$ there are two possible state sum models: either $\lambda$ is canonical, in which case the partition function is the $n=1$ case of theorem \ref{lem:inv}, or it is the spin model given by the $n=1$ case of example \ref{ex:complex_crossing}.

Two crossings are also compatible with $\Cb C_3$: one, the canonical; the other, giving rise to $\eta=R^{-2}\left(\frac{2}{3}e+\frac{1}{6}(h+h^2)\right)$ and $\chi=\frac{R^{-2}}{2}(h+h^2)$ and therefore to a new spin invariant:
\begin{align}
Z(\Sigma_g,s)=(1+P(s)2^{1-g})R^{2-2g}.
\end{align} 
Defined according to $\lambda (h^j \otimes h^l)= \lambda^{jl}{}_{op}\,h^{o}\otimes h^p$, the components of the non-trivial crossing are presented in a matrix format: $\lambda^{jl}{}_{op}$ is the $(op)$ entry of a matrix $\lambda^{jl}$. Note axiom \eqref{fig:axiom_form} of a crossing implies $(\lambda^{1j})_{op}=\delta^1_p\delta^j_o$ while axiom \eqref{fig:axiom_mult} determines $\lambda^{lj}=(\lambda^{jl})^{\tr}$. The remaining matrices read
\begin{align}
\lambda^{22}=\frac{1}{2}\left(\begin{array}{ccc} 0 & 0 & 0 \\ 0 & 1 & 1 \\ 0 & 1 & -1\end{array}\right), \hspace{2mm}
\lambda^{23}=\frac{1}{2}\left(\begin{array}{ccc} 0 & 0 & 0 \\ 0 & 1 & -1 \\ 0 & 1 & 1\end{array}\right) .
\end{align}

Finally, the $\Cb C_4$ case is richer in complexity. There are a total of twelve crossings allowed giving rise to the following invariants:
\begin{align}
Z(\Sigma_g,s)=
\begin{cases}
2^{2-2g}R^{2-2g}&\left(\eta=\chi=(2R)^{-2}e\right)\\
\\
4R^{2-2g}&\left(\eta=\chi=R^{-2}e\right)\\
\\
P(s)2^{2-g}R^{2-2g}&\left(\eta=-\chi=2^{-1}R^{-2}e\right)\\
\\
(2+P(s)2^{1-g})R^{2-2g}&(\eta=\frac{R^{-2}}{4}(3e\pm h^2) \\
& \hspace{1.2mm}\chi =\frac{R^{-2}}{4}(e\pm 3h^2) )
\end{cases}
\end{align}
These crossings were found by solving all of the constraints using computer algebra. The program used by the authors is available from the arXiv version of this paper as an ancillary file.
\end{example}

\section{Categorical generalisations}\label{sec:cat}

The axioms for the models are motivated by the definitions of various types of categories. This means that the state sum models have abstract generalisations in a category framework. The notion of a Frobenius algebra in a monoidal category was introduced by Street \cite{Street}. In this definition the algebra is an object $A$ in the category $\mathcal C$. This object has duals and obeys the axioms of a Frobenius algebra. The additional axiom for a special Frobenius algebra can also be translated into the categorical language. It is convenient (but slightly less general) to assume that all of the objects in the category have duals, i.e. $\mathcal C$ is a pivotal category (also called a sovereign category). Thus the construction of \S\ref{sec:diagram} can be generalised in a straightforward way to show that a special Frobenius algebra in a pivotal category gives a categorical analogue of a planar state sum model. In this categorical analogue, the evaluation of the dual graph in \eqref{eq:diag-partition} is determined by composition in the category instead of linear algebra.

The partition function for the corresponding spherical state sum is then
\begin{equation}Z(S^2)=R^2\dim A,\end{equation}
using $\dim A$ for the categorical dimension (quantum dimension) of $A$. The spherical symmetry arises because the pivotal subcategory of  $\mathcal C$ generated by $A, m, 1$ is in fact a spherical category.

An example of this construction is given by starting with an object $V$ in the pivotal category $\mathcal C$. Then the object $A=V^*\otimes V$ is the categorical generalisation of the matrices over $V$. There is a natural multiplication map $m$ and a unit that makes $A$ a Frobenius algebra in $\mathcal C$. In this case, $\dim A=\dim V\dim V^*$ and $R^{-1}=\dim V$. Then the partition function reduces to
\begin{equation} Z(S^2)=R\dim V^*=\frac{\dim V^*}{\dim V},\label{eq:pivotal-partition}\end{equation}
 generalising the case of $n\times n$ matrices for $k=\Cb$ from \eqref{eq:diag_sphere_inv}. 

The spherical condition can be accomodated more generally by requiring $\mathcal C$ to be a spherical category. Then the spherical condition will hold for all morphisms in the category, which can be viewed as `defects' \cite{Runkel} for the state sum model. For the example  $A=V^*\otimes V$, this implies that $Z(S^2)=1$.

The axioms of definition \ref{def:spin-model} for the spin state sum models are contained in the axioms for a special Frobenius algebra in a symmetric ribbon category. Thus any special Frobenius algebra in a symmetric ribbon category will determine a spin state sum model. However, we do not have any interesting examples that are more general than the matrix ones given in section \ref{sec:spinssm}.

\section*{Acknowledgments} The authors thank Ingo Runkel for helpful comments.
SOGT thanks Antony Lee for discussions and acknowledges financial support from Funda\c{c}\~{a}o para a Ci\^{e}ncia e
Tecnologia, Portugal through grant no. SFRH/BD/68757/2010. JWB was supported by STFC particle physics theory consolidated grant ST/L000393/1.
\bibliographystyle{unsrt}
\bibliography{references}

\end{document}